\documentclass[12pt,a4paper]{amsart}
\usepackage{amssymb}
\newtheorem{theorem}{Theorem}
\newtheorem{proposition}[theorem]{Proposition}
\newtheorem{corollary}[theorem]{Corollary}
\newtheorem{lemma}[theorem]{Lemma}
\numberwithin{theorem}{section}
\theoremstyle{remark}
\newtheorem*{example}{Example}
\newtheorem*{remark}{Remark}
\newtheorem*{definition}{Definition}
\numberwithin{equation}{section}
\newcommand{\ffrac}[2]{{\mbox{\large$\frac{#1}{#2}$}}}
\newcommand{\im}{{\mathrm{im}}\,}
\newcommand{\rank}{{\mathrm{rank}}\,}
\newcommand{\End}{{\mathrm{End}}}
\newcommand{\Hom}{{\mathrm{Hom}}}
\newcommand{\YY}{\raisebox{-1.2pt}{\Large$\oslash$}}
\newcommand{\intprod}{\;\rule{5pt}{.3pt}\rule{.3pt}{7pt}\;}
\begin{document}
\title{Prolongation on Contact Manifolds}
\author{Michael Eastwood}
\address{Mathematical Sciences Institute, Australian National University,
\newline\indent ACT 0200, Australia}
\email{meastwoo@member.ams.org}
\author{A.~Rod Gover}
\address{Department of Mathematics, University of Auckland,
Private Bag 92019,\newline\indent Auckland, New Zealand}
\email{r.gover@auckland.ac.nz}
\thanks{This research was supported by the Australian Research Council, the
Royal Society of New Zealand (Marsden Grant 06-UOA-029), and the New Zealand
Institute of Mathematics and its Applications. The first author also thanks
the University of Auckland for hospitality during the initiation of this work
and during its continuation as a Maclaurin Fellow.}
\copyrightinfo{2004}{American Mathematical Society}
\begin{abstract}
On contact manifolds we describe a notion of (contact) finite-type for linear
partial differential operators satisfying a natural condition on their leading
terms. A large class of linear differential operators are of finite-type in
this sense but are not well understood by currently available techniques. We
resolve this in the following sense. For any such $D$ we construct a partial
connection $\nabla_H$ on a (finite rank) vector bundle with the property that
sections in the null space of $D $ correspond bijectively, and via an explicit
map, with sections parallel for the partial connection. It follows that the
solution space of $D$ is finite dimensional and bounded by the corank of the
holonomy algebra of $\nabla_H$. The treatment is via a uniform procedure, even
though in most cases no normal Cartan connection is available.
\end{abstract}
\renewcommand{\subjclassname}{\textup{2000} Mathematics Subject Classification}
\maketitle

\section{Introduction} 

The prolongations of a $k^{\rm th}$ order linear differential operator between
vector bundles arise by differentiating the given operator $D:E\to F$, and
forming a new system comprising $D$ along with auxiliary operators that capture
some of this derived data. To exploit this effectively it is crucial to
determine what part of this information should be retained, and then how
best to manage it. With this understood, for many classes of operators the
resulting prolonged operator can expose key properties of the original
differential operator and its equation. Motivated by questions related to
integrability and deformations of structure, a theory of overdetermined
equations and prolonged systems was developed during the 1950s and 1960s by
Goldschmidt, Spencer, and others \cite{bcggg,spencer}. Generally, results in
these works are derived abstractly using jet bundle theory, and are severely
restricted in the sense that they apply most readily to differential operators
satisfying involutivity conditions. These features mean the theory can be
difficult to apply.

In the case that the given partial differential operator $D:E\to F$, has
surjective symbol there is an effective algorithmic approach to this problem.
The prolongations are constructed from the leading symbol
$\sigma(D):\bigodot^k\!\Lambda^1 \otimes E\to F$, where $\bigodot^k\!\Lambda^1$
is the bundle of symmetric covariant tensors on $M$ of rank~$k$. At a point of
$M$, denoting by $K$ the kernel of $\sigma(D)$, the spaces
$K^\ell=(\bigodot^{\ell}\!\Lambda^1\otimes K)\cap 
(\bigodot^{k+\ell}\!\Lambda^1 \otimes E)$, $\ell \geq 0$, capture spaces of new
variables to be introduced, and the system closes up if $K^\ell=0$ for
sufficiently large $\ell$. In this case the operator $D$ is said to be of 
{\em finite-type} (following \cite{spencer}). The equation is {\em regular\/}
if the spaces $K^\ell$ have constant rank over the manifold. The leading symbol
determines whether or not an equation is of finite-type and/or regular. 
If it is both, then the final prolonged system is a linear connection on 
$\bigoplus_{\ell=0}^{k-1}(\bigodot^k\!\Lambda^1\otimes E)
\oplus\bigoplus_{\ell=0}^\infty K^\ell$ with the property that its covariant 
constant sections are in $1$--$1$ correspondence with solutions of $D\sigma=0$.
In general, prolonged systems are complicated. In \cite{bceg} Kostant's
algebraic Hodge theory \cite{kostant} led to an explicit and uniform treatment
of prolongations for a large class of overdetermined partial differential
equations (in fact, semilinear equations are also treated in~\cite{bceg}).

On a connected manifold, a solution of a finite-type differential
operator is evidently determined by its finite jet at any point, that
is by a finite part of its Taylor series data. However on contact
manifolds a large class of differential operators that have the latter
property nevertheless fail to be of finite-type, in the sense
above. For example even the operation of taking
the differential of a function in contact
directions is not of finite-type. This signals that the general
prolongation theory is not adequate. If the underlying manifold has a
structure from the class of parabolic geometries \cite[\S4.2]{thebook}
(e.g.\ hypersurface type CR geometry) then, for a special class of
natural operators, the methods of the Bernstein-Gelfand-Gelfand
machinery \cite{cd,cssannals} may be applied. However, these methods are not
applicable in general.

Drawing on Tanaka's notion of a filtered manifold, Morimoto
initiated a programme for studying differential equations on contact manifolds
and their generalisations \cite{morimoto} via a notion of weighted jet bundles
that are adapted to the structure. This provides a formal framework for
treating these structures and, in particular, leads to a notion of weighted
finite-type. For example, using this notion of weighted jets,
Neusser~\cite{neusser} has recently and usefully adapted to the filtered
manifold setting, some tools of Goldschmidt~\cite{goldschmidt} sufficient to
show quite easily that the solution space of a weighted finite-type system is
finite-dimensional.

Despite this progress a significant gap remains. Ideally we would have a
uniform approach that, when applied to any specific equation from the class,
yields an explicit prolonged system from which obstructions to solution can be
calculated directly. In this article we provide a solution to this problem. In
particular we develop a new prolongation theory for contact structures which,
on the one hand, maintains a transparent and useable link with the weighted jet
picture of \cite{morimoto,neusser}, and which on the other hand is effective
and practically applicable. The main result is as follows. Corresponding to
weighted jets, on a contact manifold there is a notion of contact symbol. For
(suitably regular) partial differential operators $D:E\to F$ with surjective
contact symbol we describe an explicit iterative scheme for treating the
contact prolongation problem. The operator is said to be of {\em (contact)
finite-type\/} if the prolongations stabilise after a finite number of steps,
and in this case we obtain a partial connection on the prolonged system with
the property that its parallel sections correspond $1$--$1$ and explicitly to
solutions of $D$. This partial connection canonically promotes to a connection
on the same bundle. It follows that the dimension of the solution space for $D$
is bounded by the rank of the bundle supporting this partial connection and the
existence of solutions is equivalent to a rank reducing holonomy reduction of
the connection in the obvious way. Since the connection is constructed
concretely it is possible directly to use this to construct explicit curvature
obstructions to solutions of the $D$ equation.

For first order operators, our main result may be stated as follows. Let $H$
denote the contact distribution and $\Lambda_H^1$ its dual. There is a
canonical surjection $\Lambda^1\to\Lambda_H^1$. A first order differential
operator $E\to F$ is said to be {\em compatible\/} with the contact structure
if and only if its symbol $\Lambda^1\otimes E\to F$ factors through
$\Lambda^1\otimes E\to \Lambda_H^1\otimes E$. It means that the operator $D$
only differentiates in the contact directions. In this case the resulting
homomorphism $\Lambda_H^1\otimes E\to F$ is called the {\em partial symbol\/}
of~$D$. We shall suppose that it is surjective and write
$K_H\subseteq\Lambda_H^1\otimes E$ for its kernel. There are canonical
subbundles of $S_\perp^\ell\subseteq\bigotimes^\ell\!\Lambda_H^1$ defined via
the {\em Levi form\/}, as follows. In terms of a locally chosen contact
form~$\phi$, the Levi form may be regarded as $d\phi|_H$ and, from this point
is view, is well-defined up to scale. Adopting Penrose's {\em abstract index
notation\/}~\cite{OT} for sections of $H$ and its associated bundles, let us
write $L_{ab}$ for the Levi form. Then, it is clear that
$S_\perp^\ell\subseteq\bigotimes^\ell\!\Lambda_H^1$ defined as
\begin{equation}\label{intrinsicdefinitionofSperp}
S_\perp^\ell\equiv\left\{
\raisebox{7pt}{$X_{\underbrace{\scriptstyle abcde\cdots f}_{
\makebox[0pt]{\footnotesize$\ell$ indices}}}$}
\mbox{\enskip s.t.\ }\begin{array}{rcccl}
X_{[ab]cde\cdots f}&\!\!\!=\!\!\!&X_{c[ab]de\cdots f}
&\!\!\!=\!\!\!&X_{cd[ab]e\cdots f}\\
&\!\!\!=\!\!\!&\cdots&\!\!\!=\!\!\!&L_{ab}Y_{cde\cdots f}\end{array}
\mbox{ for some }Y_{cde\cdots f}\right\}\end{equation}
does not see the scale of $L_{ab}$ (enclosing a pair of
indices in square brackets means to take the skew part in those indices).
Certainly, $S_\perp^\ell\supseteq\bigodot^\ell\!\Lambda_H^1$ but, in fact, is
strictly bigger~(\ref{bundlesum}) for $\ell\geq 2$. Now we define
\begin{equation}\label{KHlfirstorder}
K_H^\ell\equiv(S_\perp^\ell\otimes K_H)\cap(S_\perp^{\ell+1}\otimes E),
\quad\mbox{for }\ell\geq 0.\end{equation}
\begin{theorem}\label{completefirstordertheorem}
Suppose that $K_H^\ell$ are vector bundles for all $\ell$ and that $K_H^\ell=0$
for $\ell$ sufficiently large. Then there is a connection on the bundle
$\textstyle{\mathbb{T}}\equiv E\oplus\bigoplus_{\ell\geq 0}\!K_H^\ell$ so that
the projection ${\mathbb{T}}\to E$ induces an isomorphism between the 
covariant constant sections of\/ ${\mathbb{T}}$ and the solution space 
$\{\sigma\in\Gamma(E)\mbox{\rm\ s.t.\ }D\sigma=0\}$.
\end{theorem}

Following a simplified treatment of the general prolongation theory for first
order operators in Section~\ref{general},
Theorem~\ref{completefirstordertheorem} is proved in Section~\ref{contact}
(cf.~Theorem~\ref{finalfirstordercontacttheorem}). Then, following a simplified
treatment of the general prolongation theory for higher order operators in
Section~\ref{highergeneral}, Theorem~\ref{completefirstordertheorem} is
generalised to higher order operators on contact manifolds in
Section~\ref{highercontact}. The construction is reasonably straightforward in
dimensions $2n+1$ for $n\geq 2$. Theorem~\ref{highercontacttheorem} is used to
replace the given operator with an equivalent contact compatible first order
prolonged system. It is used to construct a first order contact compatible
differential operator with surjective contact symbol, at which point we are
able to apply an iterative procedure developed for first order operators in
proving Theorem~\ref{finalfirstordercontacttheorem}. For $3$-dimensional
contact structures, however, one expects rather different phenomena to
occur~\cite{pansu,rumin}, and this is indeed the case. Nevertheless,
Proposition~\ref{iterum} provides a more general iterative scheme, and finally
the main result takes the same form in all dimensions. This is Theorem
\ref{finalcontacttheorem}. For these theorems to be useful, of course, one
needs to compute spaces of the form~(\ref{KHlfirstorder}) (and more
generally~(\ref{KHl})). Although this is, in principle, a simple matter of
multilinear algebra, in practise these spaces are difficult to identify. In
particular, it would be useful to know some {\em a priori\/} bounds on their
dimension so that the dimension of the solution space
$\{\sigma\in\Gamma(E)\mbox{ s.t.\ }D\sigma=0\}$ can thereby be bounded. For a
large class of geometrically arising linear differential operators on contact
manifolds, all this is possible and Section~\ref{geometric} is devoted to the
computation of the spaces (\ref{KHlfirstorder}) and (\ref{KHl}) for these
operators. It reduces to the computation of certain Lie algebra cohomologies for
the Heisenberg algebra. This cohomology is, in turn, already known as a special
case of Kostant's algebraic Bott-Borel-Weil Theorem~\cite{kostant} and the
resulting bounds on the dimension of the solution space are sharp.

\section{General prolongation for first order operators}\label{general}
Suppose $D:E\to F$ is a first order linear differential operator and suppose
that its symbol $\Lambda^1\otimes E\to F$ is surjective. Write $\pi$ for this
symbol and $K$ for its kernel. Define the vector bundle $E^\prime$ as the
kernel of $D:J^1E\to F$. We obtain a commutative diagram
\begin{equation}\label{commdiag}\begin{array}{ccccccccc}
&&0&&0\\
&&\downarrow&&\downarrow\\
0&\to&K&\to&E^\prime&\to&E&\to&0\\
&&\downarrow&&\downarrow&&\|\\
0&\to&\Lambda^1\otimes E&\to&J^1E&\to&E&\to&0\\
&&\makebox[0pt]{\scriptsize$\pi\,$}\downarrow\makebox[0pt]{}&&
\makebox[0pt]{\scriptsize$D\,$}\downarrow\makebox[0pt]{}\\
&&F&=&F\\
&&\downarrow&&\downarrow\\
&&0&&0
\end{array}\end{equation}
with exact rows and columns.
\begin{lemma}\label{choosesplitting}
We can find a connection $\nabla$ on $E$ so that $D$ is the composition
\begin{equation}\label{compo}
E\xrightarrow{\,\nabla\,}\Lambda^1\otimes E\xrightarrow{\,\pi\,}F.
\end{equation}
\end{lemma}
\begin{proof}
{From} diagram~(\ref{commdiag}), a splitting of
\begin{equation}\label{KTE}
0\to K\to E^\prime\to E\to 0\end{equation}
gives rise to a splitting of
$$0\to\Lambda^1\otimes E\to J^1E\to E\to 0.$$
Interpreted as a connection on~$E$, it has the required property. In fact, the
connections with this property correspond precisely to splittings
of~(\ref{KTE}).
\end{proof}
Let us fix a splitting of~(\ref{KTE}) and therefore a connection on $E$ in
accordance with Lemma~\ref{choosesplitting}. Having done this, the following
theorem and its proof describe the crucial step in classical prolongation.
\begin{theorem}\label{firsttheorem}There is a first order differential
operator
$$\widetilde{D}:
E^\prime=\begin{array}cE\\ \oplus\\ K\end{array}\longrightarrow
\begin{array}c\Lambda^1\otimes E\\ \oplus\\ \Lambda^2\otimes E\end{array}$$
so that the canonical projection $E^\prime\to E$ induces an isomorphism
\begin{equation}\label{basicisomorphism}
\{\Sigma\in\Gamma(E^\prime)\mbox{\rm\ s.t.\ }\widetilde{D}\Sigma=0\}\cong
\{\sigma\in\Gamma(E)\mbox{\rm\ s.t.\ }D\sigma=0\}.\end{equation}
\end{theorem}
\begin{proof}
Define $\widetilde{D}$ by
\begin{equation}\label{defofDtilde}
\left[\begin{array}c\sigma\\ \mu\end{array}\right]
\stackrel{\widetilde{D}}{\longmapsto}
\left[\begin{array}c\nabla\sigma-\mu\\
\nabla\mu-\kappa\sigma\end{array}\right],
\end{equation}
where $\nabla$ acting on $\mu$ denotes the differential operator
$\Lambda^1\otimes E\to\Lambda^2\otimes E$ induced by the connection
$\nabla:E\to\Lambda^1\otimes E$ and
$\kappa:E\to\Lambda^2\otimes E$ denotes the curvature of~$\nabla$. {From}
(\ref{compo}) it is clear that $D\sigma=0$ if and only if
$$\nabla\sigma=\mu\quad\mbox{for some }\mu\in\Gamma(K).$$
Having thus rewritten $D\sigma=0$, applying the differential operator
$\nabla:\Lambda^1\otimes E\to\Lambda^2\otimes E$ to both sides of this equation
implies that $\nabla\mu=\kappa\sigma$. In other words, this component of
$\widetilde{D}\Sigma$ is an optional extra arising as an obvious compatibility
requirement.
\end{proof}
\begin{remark}
Actually, there is no need to choose a connection in order to
define~$\widetilde{D}$. Following Goldschmidt~\cite[Proposition~3]{goldschmidt},
the target bundle can be invariantly defined as $J^1J^1E/J^2E$ and
$\widetilde{D}$ may then be obtained by restricting the tautological first order
differential operator $J^1E\to J^1J^1E/J^2E$ to $E^\prime\subseteq J^1E$. The
main reason for choosing $\nabla$ is that it makes prolongation into an
effective and computable procedure.
\end{remark}
Maintaining our chosen splitting of (\ref{KTE}) and induced connection, it is
evident that the symbol of $\widetilde{D}$ is
\begin{equation}\label{symbolofDtilde}
\begin{array}c\Lambda^1\otimes E\\ \oplus\\ \Lambda^1\otimes K\end{array}
\raisebox{-10pt}{$\xrightarrow{\mbox{\scriptsize
$\left[\begin{array}{cc}{\mathrm{Id}}&0\\ 0&\partial\end{array}\right]$}}$}
\begin{array}c\Lambda^1\otimes E\\ \oplus\\ \Lambda^2\otimes E\end{array},
\end{equation}
where $\partial$ is the composition
$$\Lambda^1\otimes K\hookrightarrow\Lambda^1\otimes\Lambda^1\otimes E
\xrightarrow{\,\wedge\,\otimes\,{\mathrm{Id}}\;}\Lambda^2\otimes E.$$
Let us suppose that $\partial$ has constant rank, write $F^\prime$ for the
subbundle
$$\begin{array}c\Lambda^1\otimes E\\ \oplus\\ \partial(\Lambda^1\otimes K)
\end{array}\subseteq
\begin{array}c\Lambda^1\otimes E\\ \oplus\\ \Lambda^2\otimes E\end{array},$$
and define $D^\prime:E^\prime\to F^\prime$ by
$$\left[\begin{array}c\sigma\\ \mu\end{array}\right]
\stackrel{D^\prime}{\longmapsto}
\left[\begin{array}c\nabla\sigma-\mu\\
\delta(\nabla\mu-\kappa\sigma)\end{array}\right],$$
where $\delta$ is an arbitrary splitting of
$\partial(\Lambda^1\otimes K)\hookrightarrow\Lambda^2\otimes E$, equivalently
an arbitrary choice of complementary bundle.
\begin{theorem}\label{secondtheorem}The canonical projection $E^\prime\to E$
induces an isomorphism
$$\{\Sigma\in\Gamma(E^\prime)\mbox{\rm\ s.t.\ }D^\prime\Sigma=0\}\cong
\{\sigma\in\Gamma(E)\mbox{\rm\ s.t.\ }D\sigma=0\}.$$
\end{theorem}
\begin{proof}
We follow exactly the same reasoning as for Theorem~\ref{firsttheorem}. The
only difference is that the $\delta(\nabla\mu-\kappa\sigma)$ records only some
part of the optional first order differential consequences of the equation
$\nabla\sigma=\mu$.
\end{proof}
\begin{remark}
Although the bundle $F^\prime$ is canonically defined just from
$\pi:\Lambda^1\otimes E\to F$, the construction of $D^\prime$ does involve a
choice of splitting~$\delta$. In practise, there is often a natural choice for
$\delta$ but, from the point of view adopted in this article, the main reason
for introducing $D^\prime$ is that its symbol is surjective by construction.
\end{remark}
From (\ref{symbolofDtilde}), the kernel of the symbol of $D^\prime$ is
precisely $\ker\partial\subseteq\Lambda^1\otimes K$. Equivalently, it is the
intersection
$$\textstyle K^\prime
\equiv(\Lambda^1\otimes K)\cap(\bigodot^2\!\Lambda^1\otimes E)$$
inside $\Lambda^1\otimes\Lambda^1\otimes E$. If $K^\prime$ is trivial,
then $D^\prime$ is a connection. If not, we can iterate this procedure, at
the next stage identifying
\begin{equation}\label{Kprimeprime}
\textstyle K^{\prime\prime}\equiv
(\bigodot^2\!\Lambda^1\otimes K)\cap(\bigodot^3\!\Lambda^1\otimes E)
\end{equation}
as the kernel of the symbol of $D^{\prime\prime}$. The details are left to the 
reader. Eventually, if
$$\textstyle (\bigodot^\ell\!\Lambda^1\otimes K)\cap
(\bigodot^{\ell+1}\!\Lambda^1\otimes E)=0\quad\mbox{for }\ell\mbox{
sufficiently large},$$
then $D$ is said to be of {\em finite-type\/} in the sense of
Spencer~\cite{spencer} and we have constructed a vector bundle with connection
whose covariant constant sections are in one-to-one correspondence with the
solutions of $D\sigma=0$.

\section{Contact prolongation for first order operators}\label{contact}
Let us firstly establish some notation. We shall denote
the contact distribution by~$H$ and its annihilator line-bundle
$H^\perp\hookrightarrow\Lambda^1$ by~$L$. We have a short exact sequence
$$0\to L\to\Lambda^1\to\Lambda_H^1\to 0,$$
which determines the contact structure. The de~Rham sequence begins
\begin{equation}\label{deRham}\begin{array}{ccccccccc}
&&&&0&&0\\
&&&&\downarrow&&\downarrow\\
&&&&L&&\makebox[0pt]{$\Lambda_H^1\otimes L$}\\
&&&&\downarrow&&\downarrow\\
0&\to&\Lambda^0&\xrightarrow{\,d\,}
     &\Lambda^1&\xrightarrow{\,d\,}
     &\Lambda^2&\xrightarrow{\,d\,}&\cdots\\
&&&&\downarrow&&\downarrow\\
&&&&\Lambda_H^1&&\Lambda_H^2\\
&&&&\downarrow&&\downarrow\\
&&&&0&&0
\end{array}\end{equation}
where $\Lambda_H^2$ denotes $\Lambda^2(\Lambda_H^1)$ and the columns are exact.
Let us denote by~${\mathcal{L}}$, the composition
$$L\to\Lambda^1\xrightarrow{\,d\,}\Lambda^2\to\Lambda_H^2.$$
It is a homomorphism of vector bundles. We shall refer to it as the {\em Levi
form\/} and the contact condition implies that ${\mathcal{L}}$ is injective.
For an arbitrary vector bundle $E$ with connection $\nabla$ is it easily
verified that the composition
\begin{equation}\label{coupledcomposition}
L\otimes E\to\Lambda^1\otimes E\xrightarrow{\,\nabla\,}
\Lambda^2\otimes E\to\Lambda_H^2\otimes E\end{equation}
is simply ${\mathcal{L}}\otimes{\mathrm{Id}}$. Let us denote by $d_H$ the
composition $\Lambda^0\xrightarrow{\,d\,}\Lambda^1\to\Lambda_H^1$ and, 
following Pansu~\cite{pansu},
say that a differential operator $\nabla_H:E\to\Lambda_H^1\otimes E$ is a
{\em partial connection\/} if and only if
\begin{equation}\label{leibniz}
\nabla_H(f\sigma)=f\nabla_H\sigma+d_H\!f\otimes\sigma\quad
\mbox{for any smooth function $f$ and }\sigma\in\Gamma(E).\end{equation}
If $\nabla$ is a connection on~$E$, then the composition
$E\stackrel{\nabla}{\longrightarrow}\Lambda^1\otimes E\to\Lambda_H^1\otimes E$
is a partial connection.

The operator $d_H:\Lambda^0\to \Lambda_H^1$ on a contact manifold is the
natural replacement for the exterior derivative $d:\Lambda^0\to\Lambda^1$, the
point being that, although $d_H$ sees only the contact directions, these
operators have the same kernel. With reference to the diagram (\ref{deRham}),
if $d_Hf=0$ then $df$ is actually a section of $L$. But then $d^2=0$ implies
that ${\mathcal{L}}df=0$ and then $df=0$ because ${\mathcal{L}}$ is supposed to
be injective. There is also a replacement for $d:\Lambda^1\to\Lambda^2$,
defined as follows. Again with reference to diagram~(\ref{deRham}), for
$\omega\in\Gamma(\Lambda_H^1)$, lift to $\tilde\omega\in\Gamma(\Lambda^1)$ and
project $d\tilde\omega\in\Gamma(\Lambda^2)$ into $\Gamma(\Lambda_H^2)$. Of
course, this is ill-defined owing to the choice of lift but the freedom
so entailed is precisely in the image of ${\mathcal{L}}$ in $\Lambda_H^2$.
Thus, we obtain a well-defined first order differential operator
$$d_H:\Lambda_H^1\to\Lambda_{H\perp}^2\equiv\Lambda_H^2/{\mathcal{L}}(L).$$
Furthermore, in dimension $5$ or more a little diagram chasing in
(\ref{deRham}) and injectivity of
${\mathrm{Id}}\wedge{\mathcal{L}}:\Lambda_H^1\otimes L\to\Lambda_H^3$ shows
that
\begin{equation}\label{rumin}
0\to{\mathbb{R}}\to\Lambda^0\xrightarrow{\,d_H\,}\Lambda_H^1
\xrightarrow{\,d_H\,}\Lambda_{H\perp}^2\end{equation}
is locally exact just as the de~Rham sequence is. It is the first part
of the Rumin complex~\cite{rumin}. In dimension $3$, however, the Levi form
${\mathcal{L}}:L\to\Lambda_H^2$ is an isomorphism so (\ref{rumin}) breaks down.
For the remainder of this section we shall suppose that our contact manifold
has dimension at least~$5$, postponing the $3$-dimensional case
until~\S\ref{three}.

The arguments in dimension $5$ or more closely follow the general procedure
outlined in~\S\ref{general}. Suppose $D:E\to F$ is a first order linear
differential operator and that $D$ only differentiates in the contact
directions. Precisely, we shall suppose that $D$ is {\em compatible\/} with the
contact structure meaning that its symbol factors as
\begin{equation}\label{piH}
\Lambda^1\otimes E\to\Lambda_H^1\otimes E\xrightarrow{\,\pi_H\,}F\end{equation}
and, in this case, refer to $\pi_H$ as the {\em partial\/} symbol of~$D$. As
in the general case, we shall suppose that $\pi_H$ is surjective and write 
$K_H$ for its kernel. Factoring (\ref{commdiag}) by
$$\begin{array}{ccccccccc} &&0&&0\\
&&\downarrow&&\downarrow\\
0&\to&L\otimes E&\to&L\otimes E&\to&0&\to&0\\
&&\downarrow&&\downarrow&&\|\\
0&\to&L\otimes E&\to&L\otimes E&\to&0&\to&0\\
&&\downarrow&&\downarrow\\
&&0&=&0\\
&&\downarrow&&\downarrow\\
&&0&&0
\end{array}$$
we obtain the commutative diagram
\begin{equation}\label{diagJH}\begin{array}{ccccccccc}
&&0&&0\\
&&\downarrow&&\downarrow\\
0&\to&K_H&\to&E_H^\prime&\to&E&\to&0\\
&&\downarrow&&\downarrow&&\|\\
0&\to&\Lambda_H^1\otimes E&\to&J_H^1E&\to&E&\to&0\\
&&\makebox[0pt]{\scriptsize$\pi_H\;\;$}\downarrow\makebox[0pt]{}&&
\makebox[0pt]{\scriptsize$D\,$}\downarrow\makebox[0pt]{}\\
&&F&=&F\\
&&\downarrow&&\downarrow\\
&&0&&0
\end{array}\end{equation}
with exact rows and columns and, in particular, hereby define $E_H^\prime$. A
splitting of
$$0\to K_H\to E_H^\prime\to E\to 0$$
gives rise to a partial connection $\nabla_H$ such that $D=\pi_H\circ\nabla_H$.
Any partial connection $\nabla_H:E\to\Lambda_H^1\otimes E$ gives rise to an
operator
$$\nabla_H:\Lambda_H^1\otimes E\to\Lambda_{H\perp}^2\otimes E\enskip
\mbox{characterised by }\nabla_H(\omega\otimes\sigma)=
d_H\omega\otimes\sigma-\omega\wedge\nabla_H\sigma\bmod{\mathcal{L}}(L),$$
mimicking the case of ordinary connections. {From} the partial Leibniz
rule~(\ref{leibniz}), it follows that the composition
$$E\xrightarrow{\,\nabla_H\,}\Lambda_H^1\otimes E\xrightarrow{\,\nabla_H\,}
\Lambda_{H\perp}^2\otimes E$$
is a homomorphism of vector bundles, which we shall denote by~$\kappa_H$ 
(being the natural {\em curvature} of a partial connection~\cite{pansu}).
Parallel to Theorem~\ref{firsttheorem} we have:--
\begin{theorem}\label{firstcontacttheorem}
There is a first order differential operator
$$\widetilde{D}_H:
E_H^\prime=\begin{array}cE\\ \oplus\\ K_H\end{array}\longrightarrow
\begin{array}c\Lambda_H^1\otimes E\\ \oplus\\
\Lambda_{H\perp}^2\otimes E\end{array}$$
so that the canonical projection $E_H^\prime\to E$ induces an isomorphism
$$\{\Sigma\in\Gamma(E_H^\prime)\mbox{\rm\ s.t.\ }\widetilde{D}_H\Sigma=0\}\cong
\{\sigma\in\Gamma(E)\mbox{\rm\ s.t.\ }D\sigma=0\}.$$
\end{theorem}
\begin{proof}
Define $\widetilde{D}_H$ by
\begin{equation}\label{DtildeH}
\left[\begin{array}c\sigma\\ \mu\end{array}\right]
\stackrel{\widetilde{D}_H}{\longmapsto}
\left[\begin{array}c\nabla_H\sigma-\mu\\
\nabla_H\mu-\kappa_H\sigma\end{array}\right]\end{equation}
and argue as before.
\end{proof}
Notice that $\widetilde{D}_H$ is again compatible with the contact structure.
Indeed, the symbol of $\widetilde{D}_H$ factors through
\begin{equation}\label{symbolDtilde}
\begin{array}c\Lambda_H^1\otimes E\\ \oplus\\
\Lambda_H^1\otimes K_H\end{array}
\raisebox{-10pt}{$\xrightarrow{\mbox{\scriptsize
$\left[\begin{array}{cc}{\mathrm{Id}}&0\\ 0&\partial_H\end{array}\right]$}}$}
\begin{array}c\Lambda_H^1\otimes E\\ \oplus\\
\Lambda_{H\perp}^2\otimes E\end{array},\end{equation}
where $\partial_H$ is the composition
$$\Lambda_H^1\otimes K_H\hookrightarrow\Lambda_H^1\otimes\Lambda_H^1\otimes E
\xrightarrow{\,\wedge\,\otimes\,{\mathrm{Id}}\;}\Lambda_H^2\otimes E
\to\Lambda_{H\perp}^2\otimes E,$$
which we shall suppose to be of constant rank. Again shadowing the general
case, let us write $F_H^\prime$ for the
subbundle
$$\begin{array}c\Lambda_H^1\otimes E\\ \oplus\\
\partial_H(\Lambda^1\otimes K_H)
\end{array}\subseteq
\begin{array}c\Lambda_H^1\otimes E\\ \oplus\\
\Lambda_{H\perp}^2\otimes E\end{array},$$
choose a splitting $\delta_H$ of
$\partial_H(\Lambda^1\otimes K_H)\hookrightarrow\Lambda_{H\perp}^2\otimes E$,
and define $D_H^\prime:E_H^\prime\to F_H^\prime$ by
$$\left[\begin{array}c\sigma\\ \mu\end{array}\right]
\stackrel{D_H^\prime}{\longmapsto}
\left[\begin{array}c\nabla_H\sigma-\mu\\
\delta_H(\nabla_H\mu-\kappa_H\sigma)\end{array}\right].$$
The counterpart to Theorem~\ref{secondtheorem} follows immediately:--
\begin{theorem}\label{secondcounterpart}
The canonical projection $E_H^\prime\to E$ induces an isomorphism
$$\{\Sigma\in\Gamma(E_H^\prime)\mbox{\rm\ s.t.\ }D_H^\prime\Sigma=0\}\cong
\{\sigma\in\Gamma(E)\mbox{\rm\ s.t.\ }D\sigma=0\}.$$
\end{theorem}
The operator $D_H^\prime$ is compatible with the contact structure and, by
design, has surjective symbol. Thus, we are in a position to iterate this
construction. We begin by observing from (\ref{symbolDtilde}) that the kernel
of the partial symbol of $D_H^\prime$ is
$$\ker\partial_H:\Lambda_H^1\otimes K_H\to\Lambda_{H\perp}^2\otimes E,$$
which may be viewed as the intersection
$$K_H^\prime\equiv
(\Lambda_H^1\otimes K_H)\cap(S_\perp^2\otimes E)\quad\mbox{inside }
\Lambda_H^1\otimes\Lambda_H^1\otimes E$$
where
\begin{equation}\label{Sperptwo}\textstyle
S_\perp^2=\bigodot^2\!\Lambda_H^1\oplus{\mathcal{L}}(L)=
\Big\{\phi_{ab}\in\bigotimes^2\!\Lambda_H^1\mbox{\rm\ s.t.\ }
\begin{array}l\phi_{ab}=P_{ab}+QL_{ab}\\
\mbox{\rm where }P_{ab}=P_{(ab)}\end{array}\Big\},\end{equation}
where $L_{ab}$ is (a representative of) the Levi form. That we are confined to
5 or more dimensions also shows up algebraically as follows. Let us write
$2n+1$ for the dimension of our contact manifold.
\begin{lemma}\label{Sperpthree}
If $n\geq 2$, then
$$\textstyle(\Lambda_H^1\otimes S_\perp^2)\cap(S_\perp^2\otimes\Lambda_H^1)=
\Big\{\phi_{abc}\in\bigotimes^3\!\Lambda_H^1\mbox{\rm\ s.t.\ }
\begin{array}l\phi_{abc}=P_{abc}+Q_aL_{bc}+Q_bL_{ac}+Q_cL_{ab}\\
              \mbox{\rm where }P_{abc}=P_{(abc)}\end{array}\Big\}.$$
\end{lemma}
\begin{proof}
According to elementary representation theory for
${\mathrm{Sp}}(2n,{\mathbb{R}})$, we can uniquely decompose
$\phi_{abc}\in\Lambda_H^1\otimes S_\perp^2$ as
\begin{equation}\label{decompose}
\phi_{abc}=P_{abc}+R_{abc}+T_bL_{ac}+T_cL_{ab}+Q_aL_{bc},\end{equation}
where
$$P_{abc}=P_{(abc)}\quad R_{abc}=R_{a(bc)}\quad R_{(abc)}=0\quad
L^{ab}R_{abc}=0,$$
and $L^{ab}L_{ac}=\delta^b{}_c$, the Kronecker delta.
The image
$$\textstyle\phi_{abc}-\phi_{bac}-\frac{1}{n}L^{de}\phi_{dec}L_{ab}$$
of this element in $\Lambda_{H\perp}^2\otimes\Lambda_H^1$ is
\begin{equation}\label{image}
\textstyle R_{abc}-R_{bac}+(Q_a-T_a)L_{bc}-(Q_b-T_b)L_{ac}
+\frac{1}{n}(Q_c-T_c)L_{ab}\end{equation}
and transvecting with $L^{bc}$ gives
$$\mbox{\large$\frac{(2n+1)(n-1)}{n}$}(Q_a-T_a).$$
Therefore, if $\phi_{abc}$ is also in $S_\perp^2\otimes\Lambda_H^1$, then it
follows that $T_a=Q_a$, from (\ref{image}) that $R_{abc}=0$, and
from (\ref{decompose}) the stated result.
\end{proof}
Similarly, if we inductively define
\begin{equation}\label{Sperpell}S_\perp^{\ell}=
(\Lambda_H^1\otimes S_\perp^{\ell-1})\cap
(S_\perp^{\ell-1}\otimes\Lambda_H^1),\quad\forall\,\ell\geq3,\end{equation}
or intrinsically as in (\ref{intrinsicdefinitionofSperp}),
then as ${\mathrm{Sp}}(2n,{\mathbb{R}})$-bundles
\begin{equation}\label{bundlesum}\textstyle S_\perp^\ell\cong
\bigodot^\ell\!\Lambda_H^1\oplus\bigodot^{\ell-2}\!\Lambda_H^1\oplus
\bigodot^{\ell-4}\!\Lambda_H^1\oplus\cdots\end{equation}
with explicit decompositions such as
\begin{equation}\label{exdecomp}
\textstyle S_\perp^4=
\left\{\begin{array}lP_{abcd}+Q_{ab}L_{cd}+Q_{ac}L_{bd}+Q_{ad}L_{bc}
+Q_{bc}L_{ad}+Q_{bd}L_{ac}+Q_{cd}L_{ab}\\
\quad{}+RL_{ab}L_{cd}+RL_{ac}L_{bd}+RL_{ad}L_{bc}\\
\qquad\quad
\mbox{where }P_{abcd}=P_{(abcd)}\mbox{ and } Q_{ab}=Q_{(ab)}\end{array}
\right\}.\end{equation}
The counterpart to (\ref{Kprimeprime}) is
\begin{equation}\label{KHprimeprime}K_H^{\prime\prime}\equiv
(S_\perp^2\otimes K_H)\cap(S_\perp^3\otimes E)\end{equation}
as the symbol of $D_H^{\prime\prime}$ and, more generally, if
$$(S_\perp^\ell\otimes K_H)\cap
(S_\perp^{\ell+1}\otimes E)=0\quad\mbox{for }\ell\mbox{
sufficiently large},$$
then, by iteration of the construction leading to
Theorem~\ref{secondcounterpart}, we may construct a vector bundle
${\mathbb{T}}$ with partial connection $\nabla_H$ such that
$$\{\Sigma\in\Gamma({\mathbb{T}})\mbox{ s.t.\ }\nabla_H\Sigma=0\}\cong
\{\sigma\in\Gamma(E)\mbox{ s.t.\ }D\sigma=0\}.$$
It particular, in this case it is clear that the solution space of $D$ is
finite-dimensional with dimension bounded by the rank of~${\mathbb{T}}$, namely
$$\textstyle\dim E +\dim K_H+\sum_{\ell\geq 1}
\dim\big((S_\perp^\ell\otimes K_H)\cap(S_\perp^{\ell+1}\otimes E)\big).$$
The details are left to the reader.

\subsection{The 3-dimensional case}\label{three}
On $3$-dimensional contact manifolds the Levi form
${\mathcal{L}}:L\to\Lambda_H^2$ is an isomorphism and so (\ref{rumin}) breaks
down. The Rumin complex~\cite{rumin} provides a perfectly satisfactory
replacement as follows.
\begin{lemma}\label{chase}
On a 3-dimensional contact manifold, there is a canonically
defined second order differential operator
$d_H^{(2)}:\Lambda_H^1\to \Lambda_H^1\otimes L$ so that
\begin{equation}\label{threeDrumin}
0\to{\mathbb{R}}\to\Lambda^0\xrightarrow{\,d_H\,}\Lambda_H^1
\xrightarrow{\,d_H^{(2)}\,}\Lambda_H^1\otimes L\end{equation}
is locally exact just as the de~Rham sequence is.
\end{lemma}
\begin{proof}
With reference to diagram~(\ref{deRham}), if $\omega$ is a local section of
$\Lambda_H^1$, choose an arbitrary lift $\tilde\omega$ to~$\Lambda^1$ and
consider $\tilde\omega-{\mathcal{L}}^{-1}\,q\,d\tilde\omega$, where $\,q\,$ is
the natural projection $\Lambda^2\to\Lambda_H^2$. By diagram chasing, this is
independent of choice of $\tilde\omega$ and canonically defines a differential
operator $\Lambda_H^1\to\Lambda^1$ splitting the natural projection
$\Lambda^1\to\Lambda_H^1$. By design, it also has the property that the
composition $d(\tilde\omega-{\mathcal{L}}^{-1}\,q\,d\tilde\omega)$ actually
takes values in $\Lambda_H^1\otimes L$. This defines $d_H^{(2)}$ and further
diagram chasing ensures that (\ref{threeDrumin}) is locally exact.
\end{proof}
Just as the de~Rham sequence couples with any connection on a vector bundle, so
(\ref{threeDrumin}) couples with any partial connection. To see this we can
proceed as follows. Firstly, some linear algebra. Not only is the Levi form 
${\mathcal{L}}:L\to\Lambda_H^2$ injective, but also its range consists of 
non-degenerate forms. If, as in (\ref{Sperptwo}), we choose $L_{ab}$ in the 
range of ${\mathcal{L}}$ and, as in the proof of Lemma~\ref{Sperpthree}, write 
$L^{ab}$ for its inverse, then we obtain a complement
$$\{\omega_{ab}\in\Lambda_H^2\mbox{ s.t.\ }L^{ab}\omega_{ab}=0\}$$
to the range of~${\mathcal{L}}$, independent of the choice of~$L_{ab}$. We may 
identify this complement with~$\Lambda_{H\perp}^2$. Let us write 
$s:\Lambda_H^2\to L$ for the canonical splitting of ${\mathcal{L}}$ so 
obtained. 
\begin{proposition}\label{modcons}
Suppose that $\nabla_H:E\to\Lambda_H^1\otimes E$ is a partial connection on a
contact manifold of arbitrary dimension.
Then $\nabla_H$ extends to a connection~$\nabla$, uniquely characterised by the
vanishing of the composition
\begin{equation}\label{splitcurvature}
E\stackrel{\kappa}{\longrightarrow}\Lambda^2\otimes E
\xrightarrow{\,q\otimes{\mathrm{Id}}\,}
\Lambda_H^2\otimes E
\xrightarrow{\,s\otimes{\mathrm{Id}}\,}L\otimes E,\end{equation}
where $\kappa$ is the curvature of\/~$\nabla$. Moreover, for this connection 
$\nabla_H\sigma=0\iff\nabla\sigma=0$.
\end{proposition}
\begin{proof}
Pick an arbitrary extension $\nabla$ of~$\nabla_H$. Any
homomorphism $\Phi:E\to\Lambda^1\otimes E$ gives rise another connection
$\hat\nabla=\nabla+\Phi$ with curvature
$\hat\kappa=\kappa+\nabla\Phi-\Phi\wedge\Phi$, where
$\nabla:\Lambda^1\otimes\End(E)\to\Lambda^2\otimes\End(E)$ is the natural
differential operator derived from the induced connection
$\nabla:\End(E)\to\Lambda^1\otimes\End(E)$ and $\Phi\wedge\Phi$ is the
composition
$$E\xrightarrow{\,\Phi\,}\Lambda^1\otimes
E\xrightarrow{\,{\mathrm{Id}}\otimes\Phi\,}\Lambda^1\otimes\Lambda^1\otimes E
\xrightarrow{\,\wedge\otimes{\mathrm{Id}}\,}\Lambda^2\otimes E.$$
If $\hat\nabla$ is to extend~$\nabla_H$, however, then $\Phi$ must have range
in $L\otimes E\subset\Lambda^1\otimes E$. In this case, the term
$\Phi\wedge\Phi$ does not arise in the formula for~$\hat\kappa$. Also recall
(\ref{coupledcomposition}) that the composition
$$L\otimes\End(E)\hookrightarrow\Lambda^1\otimes\End(E)
\xrightarrow{\,\nabla\,}\Lambda^2\otimes\End(E)
\xrightarrow{\,q\otimes{\mathrm{Id}}\,}\Lambda_H^2\otimes\End(E)$$
is always ${\mathcal{L}}\otimes{\mathrm{Id}}$. Hence, the curvature in the
contact directions
$$E\xrightarrow{\,\kappa\,}\Lambda^2\otimes E
\xrightarrow{\,q\otimes{\mathrm{Id}}\,}\Lambda_H^2\otimes E$$
of a connection $\nabla$ extending a given partial connection $\nabla_H$ is
determined up to
$$({\mathcal{L}}\otimes{\mathrm{Id}})\Phi,\enskip
\mbox{for }\Phi:E\to L\otimes E\mbox{ an arbitrary homomorphism}.$$
Thus, its further composition with $\Lambda_H^2\otimes
E\xrightarrow{\,s\otimes{\mathrm{Id}}\,}L\otimes E$, as
in~(\ref{splitcurvature}), is determined up to~$\Phi$ and may be
precisely eliminated. For the last statement, it is clear that 
$\nabla\sigma=0\implies\nabla_H\sigma=0$. Conversely, if $\nabla_H\sigma=0$ 
then $\nabla\sigma$ is a section of $L\otimes E$ whence 
$({\mathcal{L}}\otimes{\mathrm{Id}})\nabla\sigma=
(q\otimes{\mathrm{Id}})\kappa\sigma$. The vanishing of (\ref{splitcurvature}) 
now implies that 
$$\nabla\sigma=
(s\otimes{\mathrm{Id}})({\mathcal{L}}\otimes{\mathrm{Id}})\nabla\sigma=
(s\otimes{\mathrm{Id}})(q\otimes{\mathrm{Id}})\kappa\sigma=0,$$
as required.
\end{proof}

\begin{corollary}\label{extend}
Suppose $\nabla_H\!:\!E\to\Lambda_H^1\otimes E$ is a
partial connection on a $3$-dimensional contact manifold. Then $\nabla_H$
extends to a unique connection $\nabla$ characterised by being flat in the
contact directions, i.e.\
\begin{equation}\label{characterisation}
\nabla_X\nabla_Y\sigma-\nabla_Y\nabla_X\sigma=\nabla_{[X,Y]}\sigma
\end{equation}
for all $X,Y\in\Gamma(H)$ and $\sigma\in\Gamma(E)$.
\end{corollary}
\begin{proof} In $3$-dimensions ${\mathcal{L}}$ is an isomorphism and 
$s={\mathcal{L}}^{-1}$. Equation (\ref{characterisation}) is an explicit
rendering of the vanishing curvature~(\ref{splitcurvature}).
\end{proof}
Now, to couple (\ref{threeDrumin}) with $\nabla_H$ we simply extend to a full
connection~$\nabla$ on $E$ in accordance with Corollary~\ref{extend}. Then,
bearing in mind that the composition (\ref{coupledcomposition}) is simply
${\mathcal{L}}\otimes{\mathrm{Id}}$, the construction just given in the proof
of Lemma~\ref{chase} goes through almost unchanged. This can be seen by chasing
the following diagram
\begin{equation}\label{coupleddiagram}\begin{array}{ccccccccc}
&&&&0&&0\\
&&&&\downarrow&&\downarrow\\
&&&&L\otimes E&&\makebox[0pt]{$\Lambda_H^1\otimes L\otimes E$}\\
&&&&\downarrow&&\downarrow\\
0&\to&E&\xrightarrow{\,\nabla\,}
     &\Lambda^1\otimes E&\xrightarrow{\,\nabla\,}
     &\Lambda^2\otimes E&\xrightarrow{\,\nabla\,}&\cdots\\
&&&\begin{picture}(12,5)\put(-2,6){\vector(2,-1){14}}
\put(5,-5){\makebox(0,0){$\scriptstyle\nabla_H$}}\end{picture}
&\downarrow&&\downarrow\\
&&&&\Lambda_H^1\otimes E&&\Lambda_H^2\otimes E\\
&&&&\downarrow&&\downarrow\\
&&&&0&&0
\end{array}\end{equation}
obtained by coupling (\ref{deRham}) with the connection $\nabla$ provided by
Corollary~\ref{extend}. Let us write $\nabla_H^{(2)}$ for the resulting
operator. Of course, it is no longer the case that the composition
\begin{equation}\label{composition}
E\xrightarrow{\,\nabla_H\,}\Lambda_H^1\otimes E\xrightarrow{\,\nabla_H^{(2)}\,}
\Lambda_H^1\otimes L\otimes E\end{equation} 
vanishes. Instead, since the connection $\nabla$ is characterised by having its
curvature compose with $\Lambda^2\otimes E\to\Lambda_H^2\otimes E$ to give
zero, it follows immediately from (\ref{coupleddiagram}) that the composition
(\ref{composition}) is precisely this curvature, which we shall now write
as~$\kappa_H$. (The contrast between 3-dimensions and higher regarding the
notion of curvature of a partial connection is also noted and explored 
in~\cite{pansu}.)

We may now establish a counterpart to Theorem~\ref{firstcontacttheorem} in
the 3-dimensional setting.
\begin{theorem}\label{firstthreeDcontacttheorem}
There is a differential operator
$$\widetilde{D}_H:
E_H^\prime=\begin{array}cE\\ \oplus\\ K_H\end{array}\longrightarrow
\begin{array}c\Lambda_H^1\otimes E\\ \oplus\\
\Lambda_H^1\otimes L\otimes E\end{array}$$
so that the canonical projection $E_H^\prime\to E$ induces an isomorphism
$$\{\Sigma\in\Gamma(E_H^\prime)\mbox{\rm\ s.t.\ }\widetilde{D}_H\Sigma=0\}\cong
\{\sigma\in\Gamma(E)\mbox{\rm\ s.t.\ }D\sigma=0\}.$$
\end{theorem}
\begin{proof}Define $\widetilde{D}_H$ by
\begin{equation}\label{formulafortildeDH}
\left[\begin{array}c\sigma\\ \mu\end{array}\right]
\stackrel{\widetilde{D}_H}{\longmapsto}
\left[\begin{array}c\nabla_H\sigma-\mu\\
\nabla_H^{(2)}\mu-\kappa_H\sigma\end{array}\right],\end{equation}
noting that $\nabla_H\sigma=\mu\implies\nabla_H^{(2)}\mu=\kappa_H\sigma$ by
applying $\nabla_H^{(2)}$. Apart from this natural adjustment, the remainder of
the proof is as for Theorem~\ref{firsttheorem}.
\end{proof}
There is, of course, a significant difference between
Theorems~\ref{firstcontacttheorem} and~\ref{firstthreeDcontacttheorem} stemming
from the significantly different behaviour of the Rumin complex. The operator
$\widetilde{D}_H$ in dimension 5 and higher is again first order. But
$\widetilde{D}_H$ in Theorem~\ref{firstthreeDcontacttheorem} is second order.
Instead, we would like a first order prolonged operator and an analogue of
Theorem~\ref{secondcounterpart}.

To remedy this we may proceed as follows. Firstly, we shall present an argument
involving special local co\"ordinates and then we shall indicate how to remove
this choice to obtain a global result. For any contact distribution in 3
dimensions there are well-known local co\"ordinates $(x,y,z)$ due to Darboux
such that the contact distribution is spanned by $X\equiv\partial/\partial x$
and $Y\equiv\partial/\partial y+x\partial/\partial z$. Notice that
\begin{equation}\label{heisenberg}[X,Y]=Z\quad [X,Z]=0\quad [Y,Z]=0
\end{equation}
where $Z\equiv\partial/\partial z$. The vector fields $X,Y,Z$ span the tangent
vectors near the origin. Dually, the cotangent vectors are spanned by
$dx,dy,dz-x\,dy$ and we may split the projection $\Lambda^1\to\Lambda_H^1$ by
decreeing that $dx,dy$ span the lift of~$\Lambda_H^1$. It is then a simple
exercise to write out $d_H^{(2)}$ of Lemma~\ref{chase} using these local
co\"ordinates. Firstly, we compute the Levi form:--
$$L\ni dz-x\,dy\mapsto -dx\wedge dy\in\Lambda_H^2.$$
Following the recipe in the proof of Lemma~\ref{chase}, by writing
$\omega=g\,dx+h\,dy$ we have already lifted $\omega\in\Lambda_H^1$ to a
$1$-form $\tilde\omega$. Therefore,
$${\mathcal{L}}^{-1}\,q\,d\tilde\omega
={\mathcal{L}}^{-1}(Xh-Yg)\,dx\wedge dy=(Xh-Yg)(x\,dy-dz)$$
so
$$\tilde\omega-{\mathcal{L}}^{-1}\,q\,d\tilde\omega=
g\,dx+(h-xXh+xYg)\,dy+(Xh-Yg)\,dz.$$
Computing $d(\tilde\omega-{\mathcal{L}}^{-1}\,q\,d\tilde\omega)$ now
yields
$$\big((X^2h-XYg-Zg)\,dx+(YXh-Y^2g-Zh)\,dy\big)\wedge(dz-x\,dy).$$
Regarding this a section of $\Lambda_H^1\otimes L$ allows us to write out
(\ref{threeDrumin}) explicitly:--
$$f\stackrel{d_H}{\longmapsto}
\left[\begin{array}cXf\\ Yf\end{array}\right]\qquad
\left[\begin{array}cg\\ h\end{array}\right]\stackrel{d_H^{(2)}}{\longmapsto}
\left[\begin{array}cX^2h-XYg-Zg\\ YXh-Y^2g-Zh\end{array}\right],$$
where $\Lambda_H^1$ is trivialised using $dx,dy$ and $L$ is trivialised using
$dz-x\,dy$. As a check, notice that the composition is easily seen to be zero
by dint of~(\ref{heisenberg}). The coupled operators are given by essentially
the same formul{\ae}. Specifically, a partial connection on a vector bundle $E$
is determined by $\nabla_X$ and~$\nabla_Y$. Corollary~\ref{extend}
promotes this to a full connection by
$\nabla_Z\equiv\nabla_X\nabla_Y-\nabla_Y\nabla_X$ and (\ref{composition})
becomes
\begin{equation}\label{becomes}\sigma\stackrel{\nabla_H}{\longmapsto}
\left[\begin{array}c\nabla_X\sigma\\ \nabla_Y\sigma\end{array}\right]\qquad
\left[\begin{array}c\tau\\ \upsilon\end{array}\right]
\stackrel{\nabla_H^{(2)}}{\longmapsto} \left[\begin{array}c
\nabla_X\nabla_X\upsilon-\nabla_X\nabla_Y\tau-\nabla_Z\tau\\
\nabla_Y\nabla_X\upsilon-\nabla_Y\nabla_Y\tau-\nabla_Z\upsilon
\end{array}\right].\end{equation}
A second order linear differential operator $V\to W$ on a contact manifold is
said to be {\em compatible\/} with the contact structure if and only if its
symbol
$\bigodot^2\!\Lambda^1\otimes V\to W$ factors through the canonical projection
$\bigodot^2\!\Lambda^1\otimes V\to\bigodot^2\!\Lambda_H^1\otimes V$.
{From}~(\ref{becomes}), the operator $\nabla_H^{(2)}$ of (\ref{composition})
evidently has this property and hence so does its restriction to~$K_H$. This is
the key observation needed to re-express (\ref{formulafortildeDH}) as a first
order system. We proceed as follows. Pick any partial connection on $K_H$ and
extend to a full connection according to Corollary~\ref{extend}. Pick local
Darboux co\"ordinates $(x,y,z)$ as above and write
\begin{equation}\label{derivatives}
\nabla_1\equiv\nabla_X\qquad\nabla_2\equiv\nabla_Y\qquad
\nabla_0\equiv\nabla_Z=\nabla_1\nabla_2-\nabla_2\nabla_1.\end{equation}
To say that the second order operator $\nabla_H^{(2)}$ is compatible with the
contact structure means that we can write it uniquely as
$$\mu\stackrel{\,\nabla_H^{(2)}\,}{\longmapsto}S^{ab}\nabla_a\nabla_b\mu
+\Gamma^0\nabla_0\mu+\Gamma^a\nabla_a\mu+\Theta\mu,$$
where $S^{ab}$, $\Gamma^0$, $\Gamma^a$, $\Theta$ all take values in
$\Hom(K_H,\Lambda_H^1\otimes L\otimes E)$ and $S^{ab}$ is symmetric. Therefore,
we can write the equation $\nabla_H^{(2)}\mu=\kappa_H\sigma$ as
\begin{equation}\label{writeasfirstordersystem}
{\mathcal{P}}\rho+\Theta\mu=\kappa_H\sigma,\end{equation}
where
$${\mathcal{P}}\rho\equiv
S^{ab}\nabla_a\rho_b+\Gamma^0(\nabla_1\rho_2-\nabla_2\rho_1)+
\Gamma^a\rho_a\quad\mbox{and}\quad
\rho_a=\nabla_a\mu.$$
Overall, if we define a first order operator
\begin{equation}\label{defDhat}E_H^{\prime\prime}\equiv
\begin{array}cE\\ \oplus\\ K_H\\ \oplus\\
\makebox[0pt]{$\Lambda_H^1\otimes K_H$}\end{array}
\xrightarrow{\,\widehat{D}_H\,}
\begin{array}c\Lambda_H^1\otimes E\\ \oplus\\ \Lambda_H^1\otimes K_H\\
\oplus\\
\makebox[0pt]{$\Lambda_H^1\otimes L\otimes E$}\end{array}\quad\mbox{by }
\left[\begin{array}c\sigma\\ \mu\\ \rho\end{array}\right]\longmapsto
\left[\begin{array}c\nabla_H\sigma-\mu\\ \nabla_H\mu-\rho\\
{\mathcal{P}}\rho+\Theta\mu-\kappa_H\sigma\end{array}\right],\end{equation}
then we have proved
\begin{theorem}\label{improvedfirstthreeDcontacttheorem}
The projection $E_H^{\prime\prime}\to E$ induces an isomorphism
$$\{\Sigma\in\Gamma(E_H^{\prime\prime})\mbox{\rm\ s.t.\ }
\widehat{D}_H\Sigma=0\}\cong \{\sigma\in\Gamma(E)
\mbox{\rm\ s.t.\ }D\sigma=0\}.$$
\end{theorem}
This is the claimed remedy for Theorem~\ref{firstthreeDcontacttheorem}.
Certainly, the new operator $\widehat{D}_H$ is first order. It is compatible
with the contact structure because the same is true of~${\mathcal{P}}$. Indeed,
from the formula for~$\widehat{D}_H$, it is clear that on the first two
components of $E_H^{\prime\prime}$ the symbol is induced by the canonical
projection $\Lambda^1\to\Lambda_H^1$. Therefore, the symbol of
$\widehat{D}_H$ is carried by the symbol of ${\mathcal{P}}$
$$\Lambda^1\otimes\Lambda_H^1\otimes K_H\to\Lambda_H^1\otimes L\otimes E,$$
which is, in turn, carried by the tensors $S^{ab}$ and $\Gamma^0$, which we now
compute.
\begin{lemma}\label{contactsymbolthreeD}
For any second order operator $V\to W$ compatible with a three-dimensional
contact structure and written in Darboux local co\"ordinates as
$$\mu\longmapsto
S^{ab}\nabla_a\nabla_b\mu+\Gamma^0\nabla_0\mu+\Gamma^a\nabla_a\mu+\Theta\mu
\quad\mbox{with $S^{ab}=S^{ba}$}$$
for some partial connection $\nabla_a$ on~$V$, the vector bundle homomorphisms
$$\textstyle S^{ab}\in\Hom(V,W)\quad\mbox{for }a,b=1,2\quad\mbox{and}
\quad\Gamma^0\in\Hom(V,W)$$
are independent of choice of partial connection.
\end{lemma}
\begin{proof}
Any other partial connection has the form
\begin{equation}\label{changeonV}
\hat\nabla_a\mu=\nabla_a\mu-\Xi_a\mu\quad\mbox{for }\Xi_1,\Xi_2\in\End(V)
\end{equation}
and the required conclusion follows by substitution.
\end{proof}
\begin{lemma}\label{restriction}
Given the hypotheses of Lemma~\ref{contactsymbolthreeD} and a subbundle
$U\subseteq V$, the corresponding homomorphisms for the differential operator
restricted to $U$ are simply $S^{ab}|_U$ and $\Gamma^0|_U$.
\end{lemma}
\begin{proof} Now that we know by Lemma~\ref{contactsymbolthreeD} that these
homomorphisms are well-defined, we can start with a partial connection on $U$
and extend it to~$V$.
\end{proof}
\begin{remark} As far as the homomorphisms $S^{ab}$ are concerned,
Lemmata~\ref{contactsymbolthreeD} and~\ref{restriction} are merely saying that
the symbol $S:\bigodot^2\!\Lambda_H^1\otimes V\to W$ is invariantly
defined and behaves well when restricted to a subbundle. This is
completely standard. The new aspect is that, on a $3$-dimensional contact
manifold, the particular lower order coefficient $\Gamma^0$ behaves just as
well. This is a familiar feature of contact geometry whereby derivatives
transverse to the contact distribution should ``count double''. Later, in
Proposition~\ref{invarianceofsecondenhancedsymbol}, we shall
see this feature more precisely and find that there is an enhanced symbol in
all dimensions, best regarded as a homomorphism $S_\perp^2\otimes V\to W$.
\end{remark}
Recall that we wanted to compute the symbol of~${\mathcal{P}}$. {From}
Lemmata~\ref{contactsymbolthreeD} and~\ref{restriction} it follows that we may
do this by performing the analogous computation for the operator
$\nabla_H^{(2)}:\Lambda_H^1\otimes E\to\Lambda_H^1\otimes L\otimes E$, for
which we have a formula~(\ref{becomes}), and then restrict the result to
$K_H\subseteq\Lambda_H^1\otimes E$. Using (\ref{derivatives}) we may
re-write~(\ref{becomes}):--
$$\nabla_H^{(2)}\left[\begin{array}c\tau\\ \upsilon\end{array}\right]
=\left[\begin{array}c
\nabla_1\nabla_1\upsilon-\frac{1}{2}(\nabla_1\nabla_2+\nabla_2\nabla_1)\tau
-\frac{3}{2}\nabla_0\tau\\[4pt]
\frac{1}{2}(\nabla_2\nabla_1+\nabla_1\nabla_2)\upsilon-\nabla_2\nabla_2\tau
-\frac{3}{2}\nabla_0\upsilon
\end{array}\right].$$
This is of the form required in Lemma~\ref{contactsymbolthreeD} with $E$ as a
passenger and, otherwise,
$$\begin{array}c\begin{array}{rcl}S^{11}\,dx&=&0\\
S^{11}\,dy&=&dx\wedge(dz-x\,dy)\end{array}\quad
\begin{array}{rcl}S^{12}\,dx&=&-\frac{1}{2}dx\wedge(dz-x\,dy)\\[4pt]
S^{12}\,dy&=&\frac{1}{2}dy\wedge(dz-x\,dy)\end{array}\\[14pt]
\begin{array}{rcl}S^{22}\,dx&=&-dy\wedge(dz-x\,dy)\\
S^{22}\,dy&=&0\end{array}\quad
\begin{array}{rcl}\Gamma^0\,dx&=&-\frac{3}{2}dx\wedge(dz-x\,dy)\\[4pt]
\Gamma^0\,dy&=&-\frac{3}{2}dy\wedge(dz-x\,dy).\end{array}\end{array}$$
Therefore, the symbol of the corresponding first order operator
$$\Lambda_H^1\otimes\Lambda_H^1\otimes E\longrightarrow
\Lambda_H^1\otimes L\otimes E$$
has $E$ as a passenger and otherwise factors through the homomorphism
$$\Lambda_H^1\otimes\Lambda_H^1\otimes\Lambda_H^1\longrightarrow
\Lambda_H^1\otimes L$$
given by
\begin{equation}\label{hom}\begin{array}{rcl}
dx\otimes dx\otimes dx&\mapsto&0\\
dx\otimes dx\otimes dy&\mapsto&dx\wedge(dz-x\,dy)\\
dx\otimes dy\otimes dx&\mapsto&-2\,dx\wedge(dz-x\,dy)\\
dx\otimes dy\otimes dy&\mapsto&-dy\wedge(dz-x\,dy)\\
dy\otimes dx\otimes dx&\mapsto&dx\wedge(dz-x\,dy)\\
dy\otimes dx\otimes dy&\mapsto&2\,dy\wedge(dz-x\,dy)\\
dy\otimes dy\otimes dx&\mapsto&-dy\wedge(dz-x\,dy)\\
dy\otimes dy\otimes dy&\mapsto&0.
\end{array}\end{equation}
The most important attribute of this homomorphism is its kernel:--
\begin{proposition} The kernel of the homomorphism {\rm(\ref{hom})} is
$$\textstyle\bigodot^3\!\Lambda_H^1\oplus
{\mathrm{span}}\{dx\otimes dx\otimes dy-dy\otimes dx\otimes dx,
dx\otimes dy\otimes dy-dy\otimes dy\otimes dx\}.$$
\end{proposition}
\begin{proof}
Clearly (\ref{hom}) is surjective and it is easy to check that the given
elements are sent to zero.
\end{proof}
Evidently, there is another way of writing this kernel:--
\begin{equation}\label{anotherway}\textstyle\big\{
P_{abc}+Q_aL_{bc}+Q_bL_{ac}+Q_cL_{ab}\in\bigotimes^3\!\Lambda_H^1,
\mbox{ such that }P_{abc}=P_{(abc)}\big\},\end{equation}
where $L_{ab}$ is the Levi form. In all dimensions, we shall write this space
as $S_\perp^3$. Lemma~\ref{Sperpthree} shows that it extends our previous
definition~(\ref{Sperpell}) and that it coincides with the definition 
(\ref{intrinsicdefinitionofSperp}) given in the introduction. The main import
of Theorem~\ref{improvedfirstthreeDcontacttheorem} stems from the kernel of the
symbol of $\widehat{D}_H$, which we have now identified as
$$\textstyle(\bigotimes^2\!\Lambda_H^1\otimes K_H)\cap(S_\perp^3\otimes E)=
(S_\perp^2\otimes K_H)\cap(S_\perp^3\otimes E)$$
just as we found for the kernel of the second prolongation in higher
dimensions~(\ref{KHprimeprime}).

Before constructing yet higher prolongations, we pause to eliminate the use of
Darboux co\"ordinates. In Lemma~\ref{contactsymbolthreeD}, we can view
$\nabla_a\mu$ as employing abstract indices in the sense of Penrose~\cite{OT}.
Thus, a section of $\Lambda_H^1$ is written as $\omega_a$ with no implied
choice of frame. Darboux co\"ordinates were used, however, to define
$\nabla_a\nabla_b\mu$. The natural remedy is to choose a partial connection on
$\Lambda_H^1$ and view $\nabla_a\nabla_b\mu$ as applying the coupled connection
on $\Lambda_H^1\otimes V$ to $\nabla_b\mu$. There are now two checks that must
be performed in order to see that
$$\textstyle S^{ab}:\bigodot^2\!\Lambda_H^1\otimes V\to W\qquad\mbox{and}
\qquad\Gamma^0:L\otimes V\to W$$
are well-defined. Firstly, if we change the partial connection on $V$ by means
of (\ref{changeonV}) for $\Xi_a\in\Lambda_H^1\otimes\End(V)$, then the second
order terms in a compatible second order operator change according to
$$S^{ab}\nabla_a\nabla_b\mu= S^{ab}(\hat\nabla_a+\Xi_a)(\hat\nabla_b+\Xi_b)\mu=
S^{ab}\hat\nabla_a\hat\nabla_b+2(S^{ab}\Xi_{a})\hat\nabla_b\mu+
(S^{ab}\hat\nabla_a\Xi_b)\mu.$$
In particular, the induced change in the first order terms is to replace
$$\Gamma^a\nabla_a\mu\quad\mbox{by}\quad
\hat\Gamma^a\hat\nabla_a\mu,\quad\mbox{where }
\hat\Gamma^a=\Gamma^a+2S^{ab}\Xi_b$$
and, hence, if we interpret the first order coefficients as specifying a
homomorphism $\Lambda^1\otimes V\to W$ then the change in such a homomorphism
is only in $\Gamma^a:\Lambda_H^1\otimes V\to W$. In summary, the composition
$$L\otimes V\to\Lambda^1\otimes V\to W$$
does not depend on the choice of partial connection on $V$ and we denote it
by~$\Gamma^0$. Secondly, we must check that the same is true if we change the
partial connection on~$\Lambda_H^1$. The general such change is
$$\hat\nabla_a\omega_b=\nabla_a\omega_b-\Omega_{ab}{}^c\omega_c\quad
\mbox{for }\Omega_{ab}{}^c\in\Lambda_H^1\otimes\End(\Lambda_H^1)$$
and $\hat\Gamma^a=\Gamma^a+S^{bc}\Omega_{bc}{}^a$ is the only change in first
order coefficients. Again $\Gamma^0$ is unaffected. Also note that, with this
interpretation, it is unnecessary that the contact manifold be $3$-dimensional.
Thus, we have proved the following.
\begin{proposition}\label{invarianceofsecondenhancedsymbol}
On a contact manifold of arbitrary dimension, a second order linear differential
operator $V\to W$ compatible with the contact structure gives rise to
invariantly defined homomorphisms
$$\textstyle S:\bigodot^2\!\Lambda_H^1\otimes V\to W\quad\mbox{and}\quad
\Gamma^0:L\otimes V\to W,$$
in other words, an enhanced symbol\/ $S_\perp^2\otimes V\to W$.
\end{proposition}
For use in~\S\ref{highercontact}, it is worthwhile recording the reasoning
employed in deriving (\ref{writeasfirstordersystem}) and
Proposition~\ref{invarianceofsecondenhancedsymbol} as the following.
\begin{lemma}\label{firstordertrick}
Suppose $D:V\to W$ is a second order differential operator compatible with the 
contact structure on a $3$-dimensional contact manifold. Suppose $\nabla_H$ is 
a partial connection on~$V$. Then we can find a first order operator 
${\mathcal{P}}:\Lambda_H^1\otimes V\to W$ compatible with the contact structure
and a homomorphism $\Theta:V\to W$ such that
$$D={\mathcal{P}}\circ\nabla_H+\Theta.$$
Moreover, the restricted symbol of ${\mathcal{P}}$
$$\Lambda_H^1\otimes(\Lambda_H^1\otimes V)\to W$$
coincides with the enhanced symbol of $D$
$$\textstyle\bigotimes^2\!\Lambda_H^1\otimes V=S_\perp^2\otimes V\to W.$$
\end{lemma}
Now that we know by Proposition~\ref{invarianceofsecondenhancedsymbol} that the
homomorphisms $S$ and $\Gamma^0$ are invariantly defined, we may compute them
for the operator
\begin{equation}\label{delsquaredE}
\nabla_H^{(2)}:\Lambda_H^1\otimes E\to\Lambda_H^1\otimes L\otimes E
\end{equation}
on a $3$-dimensional contact manifold by using Darboux co\"ordinates. We
already did this in deriving (\ref{hom}) and the following proposition simply
writes the result in a globally well-defined manner.
\begin{proposition}\label{calculateenhancedsymbol} Let
$\Sigma:\bigotimes^3\!\Lambda_H^1\to\Lambda_H^1\otimes L$ denote the operator
\begin{equation}\label{globalhom}
\textstyle\bigotimes^3\!\Lambda_H^1\ni\phi_{abc}\longmapsto
L^{ab}\big(\phi_{abc}-\phi_{cab}\big)\in\Lambda_H^1\otimes L.\end{equation}
Then the enhanced symbol of the operator {\rm(\ref{delsquaredE})} is
$$\textstyle S_\perp^2\otimes\Lambda_H^1\otimes E=
\bigotimes^3\!\Lambda_H^1\otimes E
\xrightarrow{\,\Sigma\otimes{\mathrm{Id}}\,}\Lambda_H^1\otimes L\otimes E.$$
\end{proposition}
\begin{remark} One can readily verify that the kernel of (\ref{globalhom}) is
$S_\perp^3$, as expected. Indeed, if we regard the homomorphism $\Sigma$ as
$$\textstyle\Sigma^{ab}:\bigotimes^2\!\Lambda_H^1\to
\Hom(\Lambda_H^1,\Lambda_H^1\otimes L),$$
then Proposition~\ref{calculateenhancedsymbol} implies that we may write the
operator ${\mathcal{P}}$ as
\begin{equation}\label{writeP}
{\mathcal{P}}\rho=\Sigma^{ab}\nabla_a\rho_b+\Gamma^a\rho_a\end{equation}
and so the partial symbol of
${\mathcal{P}}$ is
$(\Sigma\otimes{\mathrm{Id}})|_{\Lambda_H^1\otimes\Lambda_H^1\otimes K_H}$
with kernel
$$(\Lambda_H^1\otimes\Lambda_H^1\otimes K_H)\cap
\ker(\Sigma\otimes{\mathrm{Id}})=
(\Lambda_H^1\otimes\Lambda_H^1\otimes K_H)\cap(S_\perp^3\otimes E).$$
\end{remark}
\begin{remark} Contact geometry is often developed by supposing that the
bundle $L$ is trivial. A trivialising section $\alpha$ is then referred to as a
{\em contact form\/}. Such a contact form gives rise to a preferred vector
field $Z$ transverse to the contact distribution and characterised by
$Z\intprod\alpha=1$ and $Z\intprod d\alpha=0$. It is called the {\em Reeb\/}
vector field. In Darboux co\"ordinates on a $3$-dimensional contact manifold
$Z=\partial/\partial z$. We obtain an alternative global point of view in
which~$\nabla_0=Z\intprod\nabla$.
\end{remark}
\begin{remark} Although an unnecessary restriction in choosing a partial
connection on $\Lambda_H^1$ above, it is interesting to note that there are
preferred connections having a convenient relationship with the Levi form as follows
(cf.~\cite[Proposition~4.2.1]{thebook}). We work in arbitrary dimension. Let us
say a linear connection $\nabla$ on the tangent bundle $TM$ to our contact
manifold $M$ is {\em adapted\/} if it preserves the distribution $H$. Adapted
connections can be constructed by splitting the sequence
\begin{equation}\label{split}
 0\to H\to TM\to L^*\to 0
\end{equation}
and choosing separate connections on $H$ and~$L$. An adapted connection
$\nabla$ on $TM$ restricts to a partial connection~$\nabla_H$. We shall use 
the same notation for the induced connections and partial connections on $H$ 
and~$L$.

Let $T^\nabla$ be the torsion of an adapted connection $\nabla$ on~$TM$. {From}
the formula
$$T^\nabla(X,Y)=\nabla_XY-\nabla_YX-[X,Y],\enskip\mbox{for }X,Y,\in\Gamma(TM)$$
it follows that if $X,Y\in\Gamma(H)$, then $T(X,Y)\bmod H$ is precisely
$-{\mathcal{L}}(X,Y)$, where ${\mathcal{L}}$ is the Levi form. In particular,
adapted connections cannot be torsion free. On the other hand, there is a
related and well-suited condition available in the presence of a splitting of
the sequence~(\ref{split}), equivalently a splitting of the first column of the
diagram (\ref{deRham}). Then we may take the Levi-form ${\mathcal{L}}(~,~)$ to
be $TM$-valued and define
$$T_H^\nabla\equiv T^\nabla+{\mathcal{L}},$$  
as a section of $\Lambda^2 \otimes TM$. For this 
{\em adapted torsion}, if $X,Y\in \Gamma (H)$, then we have 
$T_H^{\nabla}(X,Y)\in \Gamma(H)$. In fact we shall be mainly interested in this
part of $T_H$ so let us write $\tau^\nabla$ for the restriction of $T_H^\nabla$
to $H\wedge H$ and call it {\em partial torsion}.

Note that if we modify $\nabla$ to $\nabla'$, so that the difference
$\nabla-\nabla'$ is $\frac{1}{2}T_H^\nabla$, then $\nabla'$ is again adapted
but is also adapted torsion free, i.e.\ $T_H^{\nabla'}=0$. The full torsion of
$\nabla'$ is then $-{\mathcal{L}}$.  In particular, $\nabla'$ is partially
torsion-free, i.e.~ $\tau^{\nabla'}=0$.

Let $\nabla'$ be any adapted connection such that $\tau^{\nabla'}=0$
and let us write $R^{\nabla'}$ for its curvature, and $\nabla_H'$ for
the associated partial connection on $TM$. By Proposition
\ref{modcons} (with $TM$ as $E$ there) there is modification
$\nabla''$ of $\nabla$, so that $\nabla'_H$ remains fixed, and so that
$(q\otimes Id) \circ R^{\nabla''}$ is a section of
$\Lambda_{H\perp}^2\otimes\End (TM) $. From the proof of that
proposition (in particular that the $\Phi$ involved takes values in
$L\otimes TM$) it follows at once that $\tau^{\nabla'}$ is a property
of $\nabla'_H$, whence $\tau^{\nabla''}=\tau^{\nabla'}=0$. Using also
how the uniqueness statement is established in the proposition, it
follows easily that $\nabla''$ preserves $H$. Using this and the
uniqueness statement itself we have, in summary, the following.
\begin{proposition} \label{Tconn}
Given a contact distribution $H$ and a splitting of the short exact 
sequence~{\rm(\ref{split})}, there is a partial connection
$\nabla_H$ on $TM$ with partial torsion zero,
i.e.\ $\tau^{\nabla_H}=0$. The connection $\nabla_H$ admits a unique
extension to a connection $\nabla$ on $TM$ characterised by the property 
that 
$$
(q\otimes Id) \circ R^{\nabla} 
$$ is a section of $\Lambda_{H\perp}^2\otimes\End (TM) $. The
connection $\nabla$ preserves $H$ and, viewed as a connection on $H$,
is the (unique) extension of the partial connection $\nabla_H$ on $H$,
as given by Proposition \ref{modcons}.
\end{proposition}

Finally concerning adapted connections, let us note that if $L$ is trivial (in
which case a splitting of (\ref{split}) can be obtained from the Reeb field
associated to any trivialising section) then one can also arrange that the
induced connection on $L$ is flat. If we work locally, then the best we can do
for an adapted connection is to construct a flat connection $\partial$ with
torsion from Darboux local co\"ordinates. Using abstract indices $a,b,\ldots$ 
to adorn sections of $\Lambda_H^1$ and the index $0$ to indicate a section of
the line bundle~$L$ (now trivialised), we have a flat connection with
\begin{equation}\label{Darbouxconnection}
\nabla_{[a}\nabla_{b]}f=-L_{ab}\nabla_0f,\qquad
\mbox{for all smooth functions~$f$}\end{equation}
and we shall refer to it as a {\em Darboux connection}.
\end{remark}

The main remaining task in this subsection is to construct higher prolongations
of a first order operator $D:E\to F$ compatible with a $3$-dimensional contact
structure. Unfortunately, for this task it is not sufficient to modify and
iterate Theorem~\ref{improvedfirstthreeDcontacttheorem} as one might expect.
The problem is that
$$\textstyle(\bigotimes^2\!\Lambda_H^1\otimes S_\perp^3)\cap
(S_\perp^3\otimes\bigotimes^2\!\Lambda_H^1)\nsupseteq S_\perp^5$$
whereas we shall soon see that $S_\perp^5$ is attained with a more efficient
prolongation. But firstly, we shall encounter $S_\perp^4$ as follows. Write
$K_H^\prime\equiv\Lambda_H^1\otimes K_H$ and $\partial_H$ for the composition
$$\textstyle\Lambda_H^1\otimes K_H^\prime=
\bigotimes^2\!\Lambda_H^1\otimes K_H\hookrightarrow
\bigotimes^3\!\Lambda_H^1\otimes E\xrightarrow{\,\Sigma\otimes{\mathrm{Id}}\,}
\Lambda_H^1\otimes L\otimes E.$$
It is the partial symbol of
${\mathcal{P}}:K_H^\prime\to\Lambda_H^1\otimes L\otimes E$ with kernel
$$K_H^{\prime\prime}\equiv
(\Lambda_H^1\otimes\Lambda_H^1\otimes K_H)\cap(S_\perp^3\otimes E).$$
Write $F_H^{\prime\prime}$ for the subbundle
$$\begin{array}c\Lambda_H^1\otimes E\\ \oplus\\ \Lambda_H^1\otimes K_H\\
\oplus\\ \partial_H(\Lambda_H^1\otimes K_H^\prime)
\end{array}\subseteq
\begin{array}c\Lambda_H^1\otimes E\\ \oplus\\ \Lambda_H^1\otimes K_H\\
\oplus\\ \Lambda_H^1\otimes L\otimes E\end{array},$$
recall the definition of $E_H^{\prime\prime}$ in (\ref{defDhat}), choose a
splitting $\delta_H$ of
$$\partial_H(\Lambda_H^1\otimes K_H^\prime)
\hookrightarrow\Lambda_H^1\otimes L\otimes E,$$
and consider the operator obtained from $\widehat{D}_H$ by using this
splitting to delete some of the consistency equations, namely
$$D_H^{\prime\prime}:E_H^{\prime\prime}\to F_H^{\prime\prime}
\quad\mbox{defined by}\quad
\left[\begin{array}c\sigma\\ \mu\\ \rho\end{array}\right]\longmapsto
\left[\begin{array}c\nabla_H\sigma-\mu\\ \nabla_H\mu-\rho\\
\delta_H({\mathcal{P}}\rho+\Theta\mu-\kappa_H\sigma)\end{array}\right].$$
The following theorem is a true analogue of Theorems~\ref{secondtheorem}
and~\ref{secondcounterpart} (note that $D_H^{\prime\prime}$ is evidently
compatible with the contact structure and has surjective symbol).
\begin{theorem}
The projection $E_H^{\prime\prime}\to E$ induces an isomorphism
$$\{\Sigma\in\Gamma(E_H^{\prime\prime})\mbox{\rm\ s.t.\ }
D_H^{\prime\prime}\Sigma=0\}\cong \{\sigma\in\Gamma(E)
\mbox{\rm\ s.t.\ }D\sigma=0\}.$$
\end{theorem}
\begin{proof} Immediate from Theorem~\ref{improvedfirstthreeDcontacttheorem}.
\end{proof}
Let us now lift the splitting $\delta_H$ so that it maps to~$\Lambda^1\otimes
K_H^\prime$. In other words, let us consider the exact sequence
$$0\to K_H^{\prime\prime}\to\Lambda_H^1\otimes K_H^\prime
\xrightarrow{\,\partial_H\,}\Lambda_H^1\otimes L\otimes E$$
and choose $\delta_H:\Lambda_H^1\otimes K_H^\prime\longleftarrow
\Lambda_H^1\otimes L\otimes E$ such that
$\partial_H\delta_H\partial_H=\partial_H$ and
$\delta_H\partial_H\delta_H=\delta_H$. Then we can rewrite the kernel of
$D_H^{\prime\prime}$ as the system of equations
\begin{equation}\label{rewritten}
\begin{array}{rcl}\nabla_H\sigma&=&\mu\\ \nabla_H\mu&=&\rho\\
\delta_H{\mathcal{P}}\rho&=&
\delta_H\kappa_H\sigma-\delta_H\Theta\mu\enskip\bmod K_H^{\prime\prime}.
\end{array}\end{equation}
Consider the operator
$$K_H^\prime\ni\rho\longmapsto
\delta_H{\mathcal{P}}\rho\enskip\bmod K_H^{\prime\prime}\in
\frac{\Lambda_H^1\otimes K_H^\prime}{K_H^{\prime\prime}}.$$
Its symbol is surjective by design. Therefore, there is a partial connection
$\nabla_H$ on $K_H^\prime$ such that
\begin{equation}\label{newpartialconnection}
\delta_H{\mathcal{P}}\rho=\nabla_H\rho\enskip\bmod K_H^{\prime\prime}
\end{equation}
and we may rewrite the last equation of (\ref{rewritten}) as
\begin{equation}\label{lastrewritten}
\nabla_H\rho=\delta_H\kappa_H\sigma-\delta_H\Theta\mu+\tau\quad
\mbox{for some }\tau\in\Gamma(K_H^{\prime\prime}).\end{equation}
Note that previously, in order to define~${\mathcal{P}}$, we already chose a
partial connection on $\Lambda_H^1$ and hence on
$K_H^\prime=\Lambda_H^1\otimes K_H$ but henceforth we shall always prefer to
use our new partial connection chosen so that (\ref{newpartialconnection})
holds. More concretely, we are choosing the partial connection on $K_H^\prime$
so as to eliminate the effect of the $\Gamma^a$ terms in~(\ref{writeP}).

Now consider the equation $\nabla_H\mu=\rho$ from (\ref{rewritten}). As usual,
Corollary~\ref{extend} extends the partial connection $\nabla_H$
to a full connection on $K_H$ whose curvature $\kappa_H^\prime$ appears as the
composition
\begin{equation}\label{kappaHprime}
K_H\xrightarrow{\,\nabla_H\,}\Lambda_H^1\otimes K_H
\xrightarrow{\,\nabla_H^{(2)}\,}\Lambda_H^1\otimes L\otimes K_H.\end{equation}
Therefore, we may add another equation
$$\nabla_H^{(2)}\rho=\kappa_H^\prime\mu$$
to the system~(\ref{rewritten}) without disturbing its solutions. As we did
with the equation $\nabla_H^{(2)}\mu=\kappa_H\sigma$
in~(\ref{writeasfirstordersystem}), we may write this second order equation as
a first order system
\begin{equation}\label{systemforrho}
\begin{array}{rcl}\nabla_H\rho&=&\nu\\
{\mathcal{Q}}\nu+\Omega\rho&=&\kappa_H^\prime\mu\end{array}\end{equation}
where
\begin{itemize}
\item $\nabla_H$ is our preferred partial connection on $K_H^\prime$;
\item
${\mathcal{Q}}:\Lambda_H^1\otimes K_H^\prime\to\Lambda_H^1\otimes L\otimes K_H$
is a first order differential operator compatible with the contact structure
and whose partial symbol is
$$\textstyle\Lambda_H^1\otimes\Lambda_H^1\otimes K_H^\prime
=\bigotimes^3\!\Lambda^1\otimes K_H
\xrightarrow{\,\Sigma\otimes{\mathrm{Id}}\,}
\Lambda_H^1\otimes L\otimes K_H$$
with kernel $S_\perp^3\otimes K_H$;
\item $\Omega:K_H^\prime\to\Lambda_H^1\otimes L\otimes K_H$ is some
homomorphism.
\end{itemize}
We may write ${\mathcal{Q}}$ explicitly as
\begin{equation}\label{thisiscalQ}
{\mathcal{Q}}\nu=(\Sigma\otimes{\mathrm{Id}})\nabla_H\nu+\Gamma\nu
\end{equation}
where the partial connection $\nabla_H$ on $\Lambda_H^1\otimes K_H^\prime$ is
induced by choosing any partial connection on $\Lambda_H^1$ and
$\Gamma:\Lambda_H^1\otimes K_H^\prime\to\Lambda_H^1\otimes L\otimes K_H$ is
some homomorphism. But instead of the system (\ref{systemforrho}) we may
substitute from (\ref{lastrewritten}) to eliminate $\nu$ and obtain
\begin{equation}\label{twoequations}
\begin{array}{rcl}
\nabla_H\rho&=&\delta_H\kappa_H\sigma-\delta_H\Theta\mu+\tau\\
{\mathcal{Q}}(\tau+\delta_H\kappa_H\sigma-\delta_H\Theta\mu)
&=&\kappa_H^\prime\mu-\Omega\rho.
\end{array}\end{equation}
If we regard $\delta_H\kappa_H\sigma$ as $(\delta_H\kappa_H)\intprod\sigma$
obtained by pairing
$$\delta_H\kappa_H\in\Hom(E,\Lambda_H^1\otimes K_H^\prime)
=\Gamma(\Lambda_H^1\otimes K_H^\prime\otimes E^*)\quad
\mbox{with}\quad\sigma\in\Gamma(E),$$
then we may use the Leibniz rule to write
$$\nabla_H(\delta_H\kappa_H\sigma)=
(\nabla_H(\delta_H\kappa_H))\sigma+\nabla_H\sigma\intprod(\delta_H\kappa_H)$$
and, furthermore, substitute from (\ref{rewritten}) to obtain
$$\nabla_H(\delta_H\kappa_H\sigma)=
(\nabla_H(\delta_H\kappa_H))\sigma+\mu\intprod(\delta_H\kappa_H).$$
Similarly,
$$\nabla_H(\delta_H\Theta\mu)=
(\nabla_H(\delta_H\Theta))\mu+\rho\intprod(\delta_H\Theta).$$
In other words, these are known linear expressions in $\sigma,\mu,\rho$.
Bearing in mind the formula (\ref{thisiscalQ}) for ${\mathcal{Q}}$, the same
conclusion applies to ${\mathcal{Q}}(\delta_H\kappa_H\sigma)$ and
${\mathcal{Q}}(\delta_H\Theta\mu)$. Hence, we may combine (\ref{rewritten})
with (\ref{twoequations}) to conclude that our original equation $D\sigma=0$ is
equivalent to the prolonged system
\begin{equation}\label{finalprolongation}
\begin{array}{rcl}\nabla_H\sigma&=&\mu\\ \nabla_H\mu&=&\rho\\
\nabla_H\rho&=&\delta_H\kappa_H\sigma-\delta_H\Theta\mu+\tau\\
{\mathcal{Q}}\tau&=&L(\sigma,\mu,\rho)
\end{array}\end{equation}
for some explicit linear function $L(\sigma,\mu,\rho)$ defined in terms of the
chosen connections on the three vector bundles $E,K_H,K_H^\prime$. Recall that
$\sigma,\mu,\rho,\tau$ are sections of the bundles
$E,K_H,K_H^\prime,K_H^{\prime\prime}$, respectively. The operator
${\mathcal{Q}}$ is initially defined on $\Lambda_H^1\otimes K_H^\prime$
but in (\ref{finalprolongation}) we see that its action is confined to
$K_H^{\prime\prime}\subset\Lambda^1\otimes K_H^\prime$. As such, its partial
symbol therefore has as its kernel
$$K_H^{\prime\prime\prime}\equiv
(S_\perp^3\otimes K_H)\cap(\Lambda_H^1\otimes K_H^{\prime\prime})=
(S_\perp^3\otimes K_H)\cap(S_\perp^4\otimes E),$$
where
$$S_\perp^4=(\Lambda_H^1\otimes S_\perp^3)\cap(S_\perp^3\otimes\Lambda_H^1)$$
as is the case in dimension~$\geq 5$. Furthermore, is straightforward to verify
that (\ref{exdecomp}) also holds in dimension~$3$. Hence, although the method
of building prolongations on a contact manifold is quite different in
dimension~$3$, the criteria for being of finite-type are identical so far. This
phenomenon continues for higher prolongations. Although rather complicated in
practise, it is clear enough how to continue with higher order prolongations in
principle. With reference to the prolonged system~(\ref{finalprolongation}), at
the next stage one rewrites ${\mathcal{Q}}\tau$ at the expense of introducing a
partial connection on $K_H^{\prime\prime}$, a suitable splitting~$\delta_H$,
and a new variable taking values in~$K_H^{\prime\prime\prime}$. There is a
second order constraint
$$\nabla_H^{(2)}(\delta_H\kappa_H\sigma-\delta_H\Theta\mu+\tau)=
\kappa_H^{\prime\prime}\rho$$
obtained from the third equation in (\ref{finalprolongation}) where
$\kappa_H^{\prime\prime}$ is the curvature arising from our chosen partial
connection on~$K_H^\prime$. As usual, one rewrites this as a first order system
using the various partial connections and the Leibniz rule, organising the
result in terms of a first order operator ${\mathcal{R}}$
on~$K_H^{\prime\prime\prime}$. The details are left to the reader. In fact,
this scheme is obtained by taking the deceptively simple iterative scheme from
\S\ref{general}, writing it out in detail, and then making adjustments to
account for the relevant integrability conditions in the Rumin complex being of
second order in dimension~$3$. An iterative scheme in dimension $3$ is
presented in~\S\ref{highercontact}. For convenience we record the final
conclusion in all dimensions as follows.

\begin{theorem}\label{finalfirstordercontacttheorem} Suppose that $D:E\to F$ is
a first order linear differential operator between smooth vector bundles on a
contact manifold. Suppose its symbol $\Lambda^1\otimes E\to F$ is surjective
and descends to a homomorphism $\Lambda_H^1\otimes E\to F$ whose kernel we
shall denote by~$K_H$. Let
$$\textstyle S_\perp^\ell\cong
\bigodot^\ell\!\Lambda_H^1\oplus\bigodot^{\ell-2}\!\Lambda_H^1\oplus
\bigodot^{\ell-4}\!\Lambda_H^1\oplus\cdots\subset\bigotimes^\ell\!\Lambda_H^1$$
be defined in terms of the Levi form $L_{ab}$ by\/ {\rm(\ref{Sperptwo}),
(\ref{anotherway}), (\ref{exdecomp})}, and generally 
by\/~{\rm(\ref{intrinsicdefinitionofSperp})}. Suppose
$$K_H^\ell\equiv(S_\perp^\ell\otimes K_H)\cap(S_\perp^{\ell+1}\otimes E)$$
are vector bundles for all~$\ell$ (we say that $D$ is `regular') and that
$K_H^\ell=0$ for $\ell$ sufficiently large (we say that $D$ is `finite-type').
Then there is a partial connection $\nabla_H$ on the bundle
$$\textstyle{\mathbb{T}}\equiv E\oplus K_H\oplus
\bigoplus_{\ell\geq 1}\!K_H^\ell$$
such that taking the first component ${\mathbb{T}}\to E$ induces an isomorphism
$$\{\Sigma\in\Gamma({\mathbb{T}})\mbox{\rm\ s.t.\ }\nabla_H\Sigma=0\}\cong
\{\sigma\in\Gamma(E)\mbox{\rm\ s.t.\ }D\sigma=0\}.$$
In particular, the solution space of $D$ is finite-dimensional with dimension
bounded by the rank of\/~${\mathbb{T}}$.
\end{theorem}
\begin{remark} A uniform approach to contact prolongation in all dimensions is
provided by the theory of {\em weighted jets\/} developed by
Morimoto~\cite{morimoto} in the much more general context of {\em filtered
manifolds}. As far as contact manifolds are concerned, the bundle $J_H^1E$
appearing in (\ref{diagJH}) is the first {\em weighted jet bundle\/} and, more
generally and in all dimensions, there are higher weighted jet bundles and
weighted jet exact sequences
\begin{equation}\label{weightedjets}
0\to S_\perp^\ell\otimes E\to J_H^\ell E\to J_H^{\ell-1}\to 0.\end{equation}
We shall return to these sequences in \S\ref{highercontact} but here we just
remark that one can modify, without too much trouble, the usual theory of
prolongation and finite-type linear differential operators due to
Goldschmidt~\cite{goldschmidt}, Spencer~\cite{spencer}, et alia, and usually
expressed in terms of ordinary jet bundles, so as to apply to filtered
manifolds simply by systematically replacing ordinary jets by weighted jets.
This is the spirit of~\cite{morimoto}. Although the partial connection in
Theorem~\ref{finalfirstordercontacttheorem} seems to be out of reach from this
point of view, Neusser~\cite{neusser} has used weighted jets to obtain the same
final bound on the dimension of the solution space of~$D$.
\end{remark}

\begin{example} As a simple example of
Theorem~\ref{finalfirstordercontacttheorem} in action, let us consider the
system of partial differential equations on~${\mathbb{R}}^3$ given by
\begin{equation}\label{pdes}Xf=0,\qquad Xg+Yf=0,\qquad Yg=0\end{equation}
in Darboux co\"ordinates (with $X=\partial/\partial x$ and
$Y=\partial/\partial y+x\partial/\partial z$, as before). Recall that we may
take $\Lambda_H^1={\mathrm{span}}\{dx,dy\}$. As a differential operator
$$E={\mathbb{R}}^2\ni\left[\begin{array}cf\\ g\end{array}\right]\longmapsto
\left[\begin{array}cXf\\ Xg+Yf\\ Yg\end{array}\right]\in{\mathbb{R}}^3=F$$
with partial symbol given by
$$dx\otimes\left[\begin{array}cf\\ g\end{array}\right]\longmapsto
\left[\begin{array}cf\\ g\\ 0\end{array}\right]\qquad
dy\otimes\left[\begin{array}cf\\ g\end{array}\right]\longmapsto
\left[\begin{array}c0\\ f\\ g\end{array}\right].$$
We see that
$$K_H={\mathrm{span}}\left\{
dx\otimes\left[\begin{array}c0\\ 1\end{array}\right]-
dy\otimes\left[\begin{array}c1\\ 0\end{array}\right]\right\}$$
has rank $1$ and so $K_H^\prime=\Lambda_H^1\otimes K_H$ (because we are in $3$
dimensions) has rank~$2$. However, it is easy to use the description of
$S_\perp^3$ in Darboux co\"ordinates given as the kernel of (\ref{hom}) to
check that $K_H^{\prime\prime}=0$. Theorem~\ref{finalfirstordercontacttheorem}
implies that the dimension of the solution space is bounded by
$$\rank E+\rank K_H+\rank K_H^\prime=2+1+2=5.$$
In fact, taking
\begin{equation}\label{solution}f=2pz+qy^2+ry+s\qquad
g=2q(z-xy)-px^2-rx+t\end{equation}
for arbitrary constants $p,q,r,s,t$ solves (\ref{pdes}) and so this bound is
sharp with (\ref{solution}) the general solution. We shall return to this
example in \S\ref{geometric}. For the system
$$Xf=0,\qquad Xg+Yf=0,\qquad Xh+Yg=0,\qquad Yh=0,$$
we find that $K_H^{\ell}=0$ for $\ell\geq 4$ and that
$$\rank E+\rank K_H+\rank K_H^\prime+\rank K_H^{\prime\prime}+
\rank K_H^{\prime\prime\prime}=3+2+4+2+3=14.$$
In fact, with the machinery of \S\ref{geometric} we shall be able to see
that the pattern of bounds for systems of this type in dimension $3$ continues
as
$$5,14,30,55,91,140,\cdots,\ffrac{(k+1)(k+2)(2k+3)}6,\cdots.$$
\end{example}

\section{General prolongation for higher order operators}\label{highergeneral}
The initial steps in prolonging a higher order operator closely follow the
first order case detailed in~\S\ref{general}. Suppose $D:E\to F$ is a
$k^{\mathrm{th}}$ order linear differential operator and suppose that its
symbol $\bigodot^k\!\Lambda^1\otimes E\to F$ is surjective. Write $\pi$ for
this symbol and $K$ for its kernel. Define the vector bundle $E^\prime$ as the
kernel of $D:J^kE\to F$. We obtain, generalising~(\ref{commdiag}), a
commutative diagram
\begin{equation}\label{commdiag_k}\begin{array}{ccccccccc} &&0&&0\\
&&\downarrow&&\downarrow\\
0&\to&K&\to&E^\prime&\to&J^{k-1}E&\to&0\\
&&\downarrow&&\downarrow&&\|\\
0&\to&\bigodot^k\!\Lambda^1\otimes E&\to&J^kE&\to&J^{k-1}E&\to&0\\
&&\makebox[0pt]{\scriptsize$\pi\,$}\downarrow\makebox[0pt]{}&&
\makebox[0pt]{\scriptsize$D\,$}\downarrow\makebox[0pt]{}\\
&&F&=&F\\
&&\downarrow&&\downarrow\\
&&0&&0
\end{array}\end{equation}
with exact rows and columns. To replace Lemma~\ref{choosesplitting}, we need
the following notion.
\begin{definition}
A {\em $k^{\mathrm{th}}$ order connection\/} on a smooth vector bundle $E$ is a
linear differential operator $\nabla^k:E\to\bigodot^k\!\Lambda^1\otimes E$
whose symbol $\bigodot^k\!\Lambda^1\otimes E\to\bigodot^k\!\Lambda^1\otimes E$
is the identity. Equivalently, such a higher order connection is a splitting of
the jet exact sequence
\begin{equation}\label{jes}
\textstyle 0\to\bigodot^k\!\Lambda^1\otimes E\to J^kE\xrightarrow{\,p\,}
J^{k-1}E\to 0.
\end{equation}
\end{definition}
\begin{lemma}\label{rewritingD}
There is a $k^{\mathrm{th}}$ order connection $\nabla^k$ on $E$ so that $D$
is the composition
$$\textstyle E\xrightarrow{\,\nabla^k\,}\bigodot^k\!\Lambda^1\otimes E
\xrightarrow{\,\pi\,}F.$$
\end{lemma}
\begin{proof}
Choose a splitting of the short exact sequence
\begin{equation}\label{higherKTE}
0\to K\to E^\prime\to J^{k-1}E\to 0
\end{equation}
and then mimic the proof of Lemma~\ref{choosesplitting}.
\end{proof}
A connection $\nabla:E\to\Lambda^1\otimes E$ induces an operator
$\nabla:\Lambda^1\otimes E\to\Lambda^2\otimes E$ and allows us to define the
its curvature. To extend this to higher order connections, let us denote by
$\YY^{k+1}\Lambda^1$ the bundle of covariant tensors $\phi_{abcd\cdots e}$
with $k+1$ indices satisfying
$$\phi_{abcd\cdots e}=\phi_{[ab](cd\cdots e)}\quad\mbox{and}\quad
\phi_{[abc]d\cdots e}=0.$$
Notice that there is a canonical projection
$$\textstyle\Lambda^1\otimes\bigodot^k\!\Lambda^1\ni\phi_{abcd\cdots e}
\stackrel{Y}{\longmapsto}\phi_{[ab]cd\cdots e}\ni\YY^{k+1}\Lambda^1$$
corresponding to the decomposition of irreducible tensor bundles
\begin{equation}\label{decomposition}
\raisebox{8pt}{$\Lambda^1\otimes\bigodot^k\!\Lambda^1\enskip=\quad$}
\begin{picture}(8,20) \put(0,8){\line(1,0){8}}
\put(0,16){\line(1,0){8}} \put(0,8){\line(0,1){8}} \put(8,8){\line(0,1){8}}
\end{picture}
\raisebox{8pt}{\enskip$\otimes$\enskip}
\begin{picture}(48,20) \put(0,8){\line(1,0){48}}
\put(0,16){\line(1,0){48}} \put(0,8){\line(0,1){8}} \put(8,8){\line(0,1){8}}
\put(16,8){\line(0,1){8}} \put(28,12){\makebox(0,0){$\cdots$}}
\put(40,8){\line(0,1){8}} \put(48,8){\line(0,1){8}}
\end{picture}
\raisebox{8pt}{\quad$=$\quad}
\begin{array}[t]{ccc}\begin{picture}(56,20) \put(0,8){\line(1,0){56}}
\put(0,16){\line(1,0){56}} \put(0,8){\line(0,1){8}} \put(8,8){\line(0,1){8}}
\put(16,8){\line(0,1){8}} \put(36,12){\makebox(0,0){$\cdots$}}
\put(24,8){\line(0,1){8}} \put(48,8){\line(0,1){8}} \put(56,8){\line(0,1){8}}
\end{picture}
&\raisebox{8pt}{$\oplus$}
&\begin{picture}(48,20) \put(0,0){\line(1,0){8}} \put(0,8){\line(1,0){48}}
\put(0,16){\line(1,0){48}} \put(0,0){\line(0,1){16}} \put(8,0){\line(0,1){16}}
\put(16,8){\line(0,1){8}} \put(28,12){\makebox(0,0){$\cdots$}}
\put(40,8){\line(0,1){8}} \put(48,8){\line(0,1){8}}
\end{picture}\\ &\|\\
\bigodot^{k+1}\!\Lambda^1&\oplus&\YY^{k+1}\Lambda^1.\end{array}\end{equation}
\begin{proposition}\label{nextBGG}
A $k^{\mathrm{th}}$ order connection
$\nabla^k:E\to\bigodot^k\!\Lambda^1\otimes E$ canonically induces a first
order operator
$\nabla:\bigodot^k\!\Lambda^1\otimes E\to\YY^{k+1}\Lambda^1\otimes E$ such that
\begin{itemize}
\item its symbol $\Lambda^1\otimes\bigodot^k\!\Lambda^1\otimes E\to
\YY^{k+1}\Lambda^1\otimes E$ is $Y\otimes{\mathrm{Id}}$
\item the composition $E\to\bigodot^k\!\Lambda^1\otimes
E\to\YY^{k+1}\Lambda^1\otimes E$ is a differential operator of order $k-1$,
which we shall denote by~$\kappa$.
\end{itemize}\end{proposition}
\begin{proof}
Choose an arbitrary connection on $E$ and an arbitrary torsion-free connection
on the tangent bundle and hence on all tensor bundles coupled with~$E$.
Denoting all resulting connections by~$\partial_a$, the operator $\nabla^k$ has
the form $$\sigma\stackrel{\nabla^k}{\longmapsto}
\overbrace{\partial_{(b}\partial_c\partial_d\cdots\partial_{e)}}^k\sigma+
\Gamma_{bcd\cdots e}{}^{fg\cdots h}
\overbrace{\partial_f\partial_g\cdots\partial_h}^{k-1}\sigma
+\mbox{lower order terms}$$
for a uniquely defined tensor $\Gamma_{bcd\cdots e}{}^{fg\cdots h}$ symmetric
in both its lower and upper indices and having values in~$\End(E)$. But then
$$\sigma_{bcd\cdots e}\stackrel{\nabla}{\longmapsto}
\partial_{[a}\sigma_{b]cd\cdots e}
+\Gamma_{cd\cdots e[a}{}^{fg\cdots h}\sigma_{b]fg\cdots h}$$
is forced by the two characterising properties of~$\nabla$ (and, in 
particular, does not depend on choice of~$\partial_a$).
\end{proof}
\begin{remark}
The operator $\nabla$ can also be constructed in a rather tautological but less
explicit fashion as follows. As a special case of
\cite[Proposition~3]{goldschmidt}, there is a canonically defined first order
differential operator
$$J^kE\xrightarrow{\,{\mathcal{G}}\,}
\frac{\Lambda^1\otimes J^kE}{\bigodot^{k+1}\!\Lambda^1\otimes E},$$
where $\bigodot^{k+1}\!\Lambda^1\otimes E$ is regarded as a subbundle of
$\Lambda^1\otimes J^kE$ by means of the inclusion
$\bigodot^{k+1}\!\Lambda^1\otimes E\hookrightarrow
\Lambda^1\otimes\bigodot^k\!\Lambda^1\otimes E$ and the jet exact
sequence~(\ref{jes}). But a $k^{\mathrm{th}}$ order connection splits
(\ref{jes}) whence there is a homomorphism of vector bundles
$$\frac{\Lambda^1\otimes J^kE}{\bigodot^{k+1}\!\Lambda^1\otimes E}
\longrightarrow\frac{\Lambda^1\otimes\bigodot^k\!\Lambda^1\otimes E}
{\bigodot^{k+1}\!\Lambda^1\otimes E}
=\YY^{k+1}\Lambda^1\otimes E,$$
where the last identification comes from~(\ref{decomposition}).
The operator $\nabla$ is the composition
$${\textstyle\bigodot^k\!\Lambda^1\otimes E}\to J^kE
\xrightarrow{\,{\mathcal{G}}\,}
\frac{\Lambda^1\otimes J^kE}{\bigodot^{k+1}\!\Lambda^1\otimes E}
\to\YY^{k+1}\Lambda^1\otimes E.$$
\end{remark}
The {\em Spencer operator\/}~\cite{spencer} is a canonically defined first
order linear differential operator 
${\mathcal{S}}:J^kE\to\Lambda^1\otimes J^{k-1}E$ uniquely characterised by
\begin{itemize}
\item its symbol $\Lambda^1\otimes J^kE\to\Lambda^1\otimes J^{k-1}E$ is induced
by the projection $J^kE\xrightarrow{\,p\,}J^{k-1}E$
\item the sequence $E\xrightarrow{\,j^k\,}J^kE\xrightarrow{\,{\mathcal{S}}\,}
\Lambda^1\otimes J^{k-1}E$ is exact,
\end{itemize}
where $j^k$ is the universal $k^{\mathrm{th}}$ order differential operator. As
done in the proof of Proposition~\ref{nextBGG}, it is straightforward to write
down a formula for ${\mathcal{S}}$ in terms of arbitrarily chosen
connections~$\partial_a$. If~$k=2$, for example, then
\begin{equation}\label{S}\sigma\xrightarrow{\,j^2\,}
\left[\begin{array}c\sigma\\
\partial_b\sigma\\
\partial_{(b}\partial_{c)}\sigma\end{array}\right]\quad\mbox{forces}\quad
\left[\begin{array}c\sigma\\
\sigma_b\\
\sigma_{bc}\end{array}\right]\xrightarrow{\,{\mathcal{S}}\,}
\left[\begin{array}c\partial_a\sigma-\sigma_a\\
\partial_a\sigma_b-K_{ab}\sigma-\sigma_{ab}\end{array}\right]\end{equation}
where $K_{ab}$ is the curvature of~$\partial_a$. Formul{\ae} for higher $k$
have more complicated terms involving higher covariant derivatives of curvature
but, clearly, the result is forced and when the connections are flat, as can
always be supposed locally, the general component of ${\mathcal{S}}$ is simply
$\partial_a\sigma_{bc\cdots d}-\sigma_{abc\cdots d}.$

The Spencer operator can be combined with a $k^{\mathrm{th}}$ order
connection $\nabla^k$
on $E$ to yield an ordinary connection on~$J^{k-1}E$. Specifically, we view
$\nabla^k$ as a splitting of (\ref{jes}) and compose
$$J^{k-1}E\leftrightarrows J^kE\xrightarrow{\,{\mathcal{S}}\,}
\Lambda^1\otimes J^{k-1}E,$$
noting that the result is a connection because its symbol is the identity. Also
denoting this connection by~$\nabla$, it is clear that the composition
$$J^{k-1}E\xrightarrow{\,\nabla\,}\Lambda^1\otimes J^{k-1}E
\xrightarrow{\,{\mathrm{Id}}\otimes p\,}\Lambda^1\otimes J^{k-2}E$$
is simply the Spencer operator
$J^{k-1}E\xrightarrow{\,{\mathcal{S}}\,}\Lambda^1\otimes J^{k-2}E$ one degree
lower down. A detailed investigation into the relationship between 
$k^{\mathrm{th}}$ order connections on $E$ and ordinary connections on 
$J^{k-1}E$ is undertaken in~\cite{eastwood}.
\begin{proposition}\label{secondproposition}
Let $\mu\in\Gamma(\bigodot^k\!\Lambda^1\otimes E)$ and
also regard $\mu$ as a section of $\Lambda^1\otimes J^{k-1}E$ by means of the
inclusions
$$\textstyle\bigodot^k\!\Lambda^1\otimes E\hookrightarrow
\Lambda^1\otimes\bigodot^{k-1}\!\Lambda^1\otimes E\hookrightarrow
\Lambda^1\otimes J^{k-1}E.$$
Then, the canonical projection $J^{k-1}E\to E$ induces an
isomorphism
$$\{\tilde\sigma\in\Gamma(J^{k-1}E)\mbox{\rm\ s.t.\ }\nabla\tilde\sigma=\mu\}
\cong\{\sigma\in\Gamma(E)\mbox{\rm\ s.t.\ }\nabla^k\sigma=\mu\}.$$
\end{proposition}
\begin{proof}
The crucial observation is that
$$\nabla\tilde\sigma=\mu\implies{\mathcal{S}}\tilde\sigma=0\iff
\tilde\sigma=j^{k-1}\sigma,\quad\mbox{for some }\sigma\in\Gamma(E).$$
The remainder of the proof is just a matter of untangling a couple of
definitions.
\end{proof}
\begin{remark} It is illuminating to view Proposition~\ref{secondproposition}
in terms of arbitrarily chosen connections $\partial_a$ as above. Suppose, for
example, that $k=2$. Then  we can write
$$\sigma\xrightarrow{\,\nabla^2\,}
\partial_{(a}\partial_{b)}\sigma+\Gamma_{ab}{}^c\partial_c\sigma
+\Theta_{ab}\sigma$$
for certain uniquely determined tensors $\Gamma_{ab}{}^c=\Gamma_{(ab)}{}^c$ and
$\Theta_{ab}=\Theta_{(ab)}$ having values in $\End(E)$, in which case
$$\nabla_a\left[\begin{array}c\sigma\\
\sigma_b\end{array}\right]={\mathcal{S}}
\left[\begin{array}c\sigma\\
\sigma_b\\
-\Gamma_{bc}{}^d\sigma_d-\Theta_{bc}\sigma\end{array}\right]=
\left[\begin{array}c\partial_a\sigma-\sigma_a\\
\partial_a\sigma_b-K_{ab}\sigma+\Gamma_{bc}{}^d\sigma_d+\Theta_{bc}\sigma
\end{array}\right]$$
in accordance with~(\ref{S}).
Therefore,
$$\nabla\tilde\sigma=\mu\iff\left\{
\begin{array}{rcl}\partial_a\sigma-\sigma_a&=&0\\
\partial_a\sigma_b-K_{ab}\sigma+\Gamma_{bc}{}^d\sigma_d+\Theta_{bc}\sigma&=&
\mu_{ab}.\end{array}\right.$$
But the first of these equations implies that
$$\begin{array}{rcl}
\partial_a\sigma_b-K_{ab}\sigma+\Gamma_{bc}{}^d\sigma_d+\Theta_{bc}\sigma&=&
\partial_a\partial_b\sigma-K_{ab}\sigma
+\Gamma_{bc}{}^d\partial_d\sigma+\Theta_{bc}\sigma\\
&=&
\partial_{(a}\partial_{b)}\sigma+\Gamma_{bc}{}^d\partial_d\sigma
+\Theta_{bc}\sigma=\nabla_{ab}^2\sigma
\end{array}$$
and so the second equation maybe rewritten as $\nabla^2\sigma=\mu$.
\end{remark}
\begin{remark}
The abstract approach and results expressed in terms of jets are due to
Goldschmidt and Spencer, e.g.~\cite{goldschmidt,spencer}. It is often the
case, however, that the operator $D:E\to F$ in question has a geometric origin,
in which case there are associated connections that one is almost obliged to
use in writing down an effective prolongation scheme. This is the approach
adopted, for example, in~\cite{bceg}.
\end{remark}

The following result generalises Theorem~\ref{firsttheorem}.
\begin{theorem}\label{highertheorem} There is a first order differential
operator
$$\widetilde{D}:
E^\prime\equiv\begin{array}cJ^{k-1}E\\ \oplus\\ K\end{array}\longrightarrow
\begin{array}c\Lambda^1\otimes J^{k-1}E\\ \oplus\\
\YY^{k+1}\Lambda^1\otimes E\end{array}$$
so that the canonical projection $E^\prime\to J^{k-1}E\to E$ induces an
isomorphism
$$\{\Sigma\in\Gamma(E^\prime)\mbox{\rm\ s.t.\ }\widetilde{D}\Sigma=0\}\cong
\{\sigma\in\Gamma(E)\mbox{\rm\ s.t.\ }D\sigma=0\}.$$
\end{theorem}
\begin{proof}
Choose a $k^{\mathrm{th}}$ order connection $\nabla^k$ in accordance with
Lemma~\ref{rewritingD} so that $D\sigma=0$ if and only if $\nabla^k\sigma=\mu$
for some $\mu\in\Gamma(K)$. Define $\widetilde{D}$ by
$$\left[\begin{array}c\tilde\sigma\\ \mu\end{array}\right]
\stackrel{\widetilde{D}}{\longmapsto}
\left[\begin{array}c\nabla\tilde\sigma-\mu\\
\nabla\mu-\kappa\tilde\sigma\end{array}\right].$$
It is the same formula as used in the proof of Theorem~\ref{firsttheorem} but
the meaning of the terms have been generalised:--
\begin{itemize}
\item $\tilde\sigma\mapsto\nabla\tilde\sigma$ is the connection on $J^{k-1}E$
associated to~$\nabla^k$
\item $\mu\mapsto\nabla\mu$ is the restriction to $K$ of the operator provided
by Proposition~\ref{nextBGG}
\item $\tilde\sigma\mapsto\kappa\tilde\sigma$ is the homomorphism of vector
bundles
$$J^{k-1}E\longrightarrow\YY^{k+1}\Lambda^1\otimes E$$
induced by the ${(k-1)}^{\mathrm{st}}$ order
$\kappa:E\to\YY^{k+1}\Lambda^1\otimes E$ in Proposition~\ref{nextBGG}.
\end{itemize}
We have already seen in Proposition~\ref{secondproposition} that
$$\nabla^k\sigma=\mu\iff\nabla\tilde\sigma=\mu\quad\mbox{for some uniquely
determined }\tilde\sigma,$$
namely $\tilde\sigma=j^{k-1}\sigma$. The equation
$\nabla\mu=\kappa\tilde\sigma$ is an optional differential consequence obtained
by applying
$\nabla:\bigodot^k\!\Lambda^1\otimes E\to\YY^{k+1}\Lambda^1\otimes E$ to the
equation $\nabla^k\sigma=\mu$.
\end{proof}
To obtain a suitable generalisation of Theorem~\ref{secondtheorem} we consider
the homomorphism $\partial$ obtained as the composition
\begin{equation}\label{itspartial}\textstyle\Lambda^1\otimes K\hookrightarrow
\Lambda^1\otimes\bigodot^k\!\Lambda^1\otimes E
\xrightarrow{\,Y\otimes{\mathrm{Id}}\,}\YY^{k+1}\Lambda^1\otimes E
\end{equation}
and choose a splitting $\delta$ of
$\partial(\Lambda^1\otimes E)\hookrightarrow\YY^{k+1}\Lambda^1\otimes E$. Then,
if we define
$$D^\prime:E^\prime\equiv\hspace*{-5pt}
\begin{array}cJ^{k-1}E\\ \oplus\\ K\end{array}\longrightarrow
\begin{array}c\Lambda^1\otimes J^{k-1}E\\ \oplus\\
\partial(\Lambda^1\otimes K)\end{array}\hspace*{-5pt}\equiv F^\prime\quad
\mbox{by}\quad\left[\begin{array}c\tilde\sigma\\ \mu\end{array}\right]
\stackrel{D^\prime}{\longmapsto}
\left[\begin{array}c\nabla\tilde\sigma-\mu\\
\delta(\nabla\mu-\kappa\tilde\sigma)\end{array}\right],$$
then, from Theorem~\ref{highertheorem}, we evidently obtain
\begin{theorem}\label{aftersplitting}The canonical projection $E^\prime\to E$
induces an isomorphism
$$\{\Sigma\in\Gamma(E^\prime)\mbox{\rm\ s.t.\ }D^\prime\Sigma=0\}\cong
\{\sigma\in\Gamma(E)\mbox{\rm\ s.t.\ }D\sigma=0\}.$$
\end{theorem}
\noindent The operator $D^\prime$ is manifestly first order with symbol
$$\Lambda^1\otimes E^\prime=
\begin{array}c\Lambda^1\otimes J^{k-1}E\\ \oplus\\
\Lambda^1\otimes K\end{array}
\raisebox{-10pt}{$\xrightarrow{\mbox{\scriptsize
$\left[\begin{array}{cc}{\mathrm{Id}}&0\\ 0&\partial\end{array}\right]$}}$}
\begin{array}c\Lambda^1\otimes J^{k-1}E\\ \oplus\\
\partial(\Lambda^1\otimes K)\end{array}.$$
In particular, the symbol is surjective and its kernel is carried by the
kernel of~$\partial$, which we shall write as~$K^\prime$. In fact, from
(\ref{itspartial}) and (\ref{decomposition}) we see that
$$\textstyle
K^\prime=(\Lambda^1\otimes K)\cap(\bigodot^{k+1}\!\Lambda^1\otimes E).$$
We conclude that if $K^\prime=0$ then $D^\prime$ is a connection. Otherwise we
are now in the realm of first order operators and may construct higher
prolongations as \S\ref{general}. We have proved the following
prolongation theorem.
\begin{theorem}\label{improvedspencer} Suppose
that $D:E\to F$ is a $k^{\mathrm{th}}$ order linear differential operator
between smooth vector bundles. Suppose its symbol
$\bigodot^k\!\Lambda^1\otimes E\to F$ is surjective and write $K$ for its
kernel. Suppose that
$$\textstyle K^\ell\equiv(\bigodot^\ell\!\Lambda^1\otimes K)\cap
(\bigodot^{k+\ell}\!\Lambda^1\otimes E)$$
are vector bundles for all~$\ell$ (we say that $D$ is `regular') and that
$K^\ell=0$ for $\ell$ sufficiently large (we say that $D$ is `finite-type').
Then there is a connection $\nabla$ on the bundle
$$\textstyle{\mathbb{T}}\equiv J^{k-1}E\oplus K\oplus
\bigoplus_{\ell\geq 1}\!K^\ell$$
such that taking the first component ${\mathbb{T}}\to J^{k-1}E\to E$ induces an
isomorphism
$$\{\Sigma\in\Gamma({\mathbb{T}})\mbox{\rm\ s.t.\ }\nabla\Sigma=0\}\cong
\{\sigma\in\Gamma(E)\mbox{\rm\ s.t.\ }D\sigma=0\}.$$
In particular, the solution space of $D$ is finite-dimensional with dimension
bounded by the rank of\/~${\mathbb{T}}$.
\end{theorem}
It is shown in \cite{bceg} that there is an extensive class of geometrically
defined symbols both on manifolds with no further structure and on Riemannian
manifolds, which belong to operators necessarily of finite-type and for which
the bundles $K^\ell$ can be computed and their dimensions determined. In
\S\ref{geometric} we shall derive corresponding results for our modified
prolongation procedure on contact manifolds.

Although the classical approach by means of jets \cite{spencer} does not reach 
Theorem~\ref{improvedspencer}, it is useful to see how far it goes.
Firstly, there is a canonical
inclusion $J^{k+1}E\hookrightarrow J^1J^kE$ for any smooth vector bundle $E$
corresponding to the composition of differential operators
$$E\xrightarrow{\,j^k\,}J^kE\xrightarrow{\,j^1\,}J^1J^kE.$$
Secondly, as we already observed following Proposition~\ref{nextBGG},
the jet exact sequence (\ref{jes}) induces a canonical inclusion 
$$\textstyle\bigodot^{k+1}\!\Lambda^1\otimes E\hookrightarrow
\Lambda^1\otimes\bigodot^k\!\Lambda^1\otimes E\hookrightarrow
\Lambda^1\otimes J^kE.$$
Goldschmidt~\cite[Proposition~3]{goldschmidt} shows that there is a canonical 
isomorphism 
\begin{equation}\label{canonicalisomorphism}
\frac{J^1J^kE}{J^{k+1}E}\cong
\frac{\Lambda^1\otimes J^kE}{\bigodot^{k+1}\!\Lambda^1\otimes E}.
\end{equation}
Let us write $W^kE$ for the vector bundle defined by either side of this
isomorphism. The operator ${\mathcal{G}}:J^kE\to W^kE$ is then induced by the
universal differential operator $j^1:J^kE\to J^1J^kE$ and the differential
operator $\widetilde{D}$ in Theorem~\ref{highertheorem} invariantly defined as
the restriction of ${\mathcal{G}}$ to $E^\prime$ where $E^\prime\subset J^kE$
in accordance with~(\ref{commdiag_k}). To proceed further, the classical
approach is either to assume that the range of $\widetilde{D}$ is the same as
the range of its symbol (this is the first criterion for the system $D$ to be
{\em compatible\/} or {\em formally integrable\/} in the sense of
Goldschmidt~\cite{goldschmidt}) in which case there is no need to choose a
splitting $\delta$ in order to obtain Theorem~\ref{aftersplitting} or, instead,
to prolong the original operator $D:E\to F$ of order $k$ to an operator
$D^\ell:E\to J^\ell F$ of order $k+\ell$ and then use
(\ref{canonicalisomorphism}) to construct a first order operator with injective
symbol in the finite-type case for $\ell$ sufficiently large. This latter
approach is carried out by Neusser~\cite{neusser} on general filtered
manifolds, including contact manifolds as a special case.

It is illuminating to write out the Goldschmidt operator ${\mathcal{G}}$ using
a connection on $E$ coupled with a flat torsion-free connection on~$\Lambda^1$ 
as can be arranged locally (whilst maintaining a preferred connection 
on~$E$). Writing $\nabla_a$ for all these connections and using them to 
trivialise the jet bundles $J^kE$, the second Spencer operator (\ref{S}) yields
\begin{equation}\label{firstgoldschmidt}
J^1E\ni\left[\begin{array}c\sigma\\ \mu_b\end{array}\right]
\stackrel{{\mathcal{G}}}{\longmapsto}
\left[\begin{array}c\nabla_a\sigma-\mu_a\\
\nabla_{[a}\mu_{b]}-\kappa_{ab}\sigma\end{array}\right]\in W^1E,\end{equation}
as is familiar~(\ref{defofDtilde}), whilst the third Spencer operator is 
straightforwardly computed to be
$$J^3E\ni
\left[\begin{array}c\sigma\\ \mu_b\\ \rho_{bc}\\ \tau_{bcd}\end{array}\right]
\stackrel{{\mathcal{S}}}{\longmapsto}
\left[\begin{array}c\nabla_a-\mu_a\\
\nabla_a\mu_b-\kappa_{ab}\sigma-\rho_{ab}\\
\nabla_a\rho_{bc}
-\kappa_{ab}\mu_c
-\kappa_{ac}\mu_b
-\frac{1}{3}(\nabla_c\kappa_{ab})\sigma
-\frac{1}{3}(\nabla_b\kappa_{ac})\sigma
-\tau_{abc}
\end{array}\right]$$
and yields
\begin{equation}\label{secondgoldschmidt}J^2E\ni
\left[\begin{array}c\sigma\\ \mu_b\\ \rho_{bc}\end{array}\right]
\stackrel{{\mathcal{G}}}{\longmapsto}
\left[\begin{array}c\nabla_a-\mu_a\\
\nabla_a\mu_b-\kappa_{ab}\sigma-\rho_{ab}\\
\nabla_{[a}\rho_{b]c}
-\kappa_{ab}\mu_c
+\kappa_{a[b}\mu_{c]}
-\frac{1}{3}(\nabla_c\kappa_{ab})\sigma
+\frac{1}{3}(\nabla_{[b}\kappa_{c]a})\sigma
\end{array}\right].\end{equation}

\section{Contact prolongation for higher order operators}\label{highercontact}
Our first task is to explain what it means for a higher order differential
operator to be {\em compatible\/} with a contact structure. For
$1^{\mathrm{st}}$ or $2^{\mathrm{nd}}$ order operators, compatibility was 
defined in \S\ref{contact} as a restriction on the symbol, namely that it factor 
through
$$\textstyle\Lambda^1\otimes E\to\Lambda_H^1\otimes E\quad\mbox{or}\quad
\bigodot^2\!\Lambda^1\otimes E\to\bigodot^2\!\Lambda_H^1\otimes E,$$
respectively. For a $k^{\mathrm{th}}$ order operator, having the symbol factor 
through 
$$\textstyle
\bigodot^k\!\Lambda^1\otimes E\to\bigodot^k\!\Lambda_H^1\otimes E,$$
is a necessary but not sufficient condition for compatibility. To proceed, let
us recall \cite{spencer} the definition of the fibre of the $k^{\mathrm{th}}$
order jet bundle $J^kE$ at a point $x$ as the space of germs of smooth sections
of $E$ at $x$ modulo those that vanish to order~$k+1$. Also recall that the
notion of vanishing to a certain order is defined componentwise with respect to
any local trivialisation of $E$ and that a function $f$ vanishes to order $k+1$
at $x$ if and only if $X_1X_2\cdots X_\ell f|_x=0$ for any vector fields
$X_1,\ldots,X_\ell$ defined near $x$ and for any $\ell\leq k+1$. Following
Morimoto~\cite{morimoto}, we define the weighted jet bundles $J_H^kE$ in
exactly the same manner except that we require the vector fields
$X_1,\ldots,X_\ell$ to lie in the contact distribution. As a less stringent
requirement, this defines a larger subspace of the germs and so there is a
natural surjection of bundles $J^kE\to J_H^kE$. We now define compatibility of
a $k^{\mathrm{th}}$ order linear differential operator $D:E\to F$ with the
contact structure to mean that the corresponding homomorphism of vector bundles
$D:J^kE\to F$ factor through $J^kE\to J_H^kE$. The usual jet exact sequence
$$\textstyle 0\to\bigodot^k\!\Lambda^1\otimes E\to J^kE\to J^{k-1}E\to 0$$
is derived from the canonical isomorphisms
$$({\textstyle\bigodot^k\!\Lambda^1})_x\cong
\frac{\{f\mbox{ s.t.\ }X_1\cdots X_\ell f|_x=0,\,\forall
\mbox{ vector fields and }\forall\,\ell\leq k\}}
{\{f\mbox{ s.t.\ }X_1\cdots X_\ell f|_x=0,\,\forall
\mbox{ vector fields and }\forall\,\ell\leq k+1\}}$$
induced by $\phi_{ab\cdots d}\mapsto\phi_{ab\cdots d}x^ax^b\cdots x^d$ for any
local co\"ordinates $(x_1,x_2,\cdots,x_m)$ centred on~$x$. We have already
mentioned (\ref{weightedjets}) that there is a corresponding exact sequence
$$0\to S_\perp^k\otimes E\to J_H^k E\to J_H^{k-1}\to 0.$$
for weighted jets, where $S_\perp^k$ is defined 
by~(\ref{intrinsicdefinitionofSperp}) and has the form given by 
(\ref{Sperptwo}), (\ref{anotherway}), (\ref{exdecomp}), and so on. 
It is derived from the canonical isomorphisms
$$(S_\perp^k)_x\cong
\frac{\{f\mbox{ s.t.\ }X_1\cdots X_\ell f|_x=0,\,\forall
\mbox{ contact fields and }\forall\,\ell\leq k\}}
{\{f\mbox{ s.t.\ }X_1\cdots X_\ell f|_x=0,\,\forall
\mbox{ contact fields and }\forall\,\ell\leq k+1\}}$$
induced by using Darboux local co\"ordinates $(x^1,x^2,\cdots,x^{2n},z)$
instead (for example, 
$(P_{abcd},Q_{ab},R)\mapsto P_{abcd}x^ax^bx^cx^d+Q_{ab}x^ax^bz+Rz^2$ for 
$S_\perp^4$ as in~(\ref{exdecomp})).

The commutative diagram 
$$\begin{array}{ccccccccc}
0&\to&\Lambda^1\otimes E&\to&J^1E&\to&E&\to&0\\
&&\downarrow&&\downarrow&&\|\\
0&\to&\Lambda_H^1\otimes E&\to&J_H^1E&\to&E&\to&0\\
&&\downarrow&&\downarrow\\
&&0&&0\end{array}$$
with exact rows and columns shows that a first order differential operator
$D:E\to F$ is compatible with the contact structure as defined above if and
only if its symbol factors through $\Lambda_H^1\otimes E$ as defined 
in~\S\ref{contact}. Similarly, the commutative diagram
$$\begin{array}{ccccccccc}
0&\to&\bigodot^2\!\Lambda^1\otimes E&\to&J^2E&\to&J^1E&\to&0\\
&&\downarrow&&\downarrow&&\|\\
0&\to&\bigodot^2\!\Lambda_H^1\otimes E&\to&J_H^2E&\to&J^1E&\to&0\\
&&\downarrow&&\downarrow\\
&&0&&0\end{array}$$
with exact rows and columns shows that a second order $D:E\to F$ is compatible
with the contact structure as defined above if and only if its symbol factors
through $\bigodot^2\!\Lambda_H^1\otimes E$ as defined in~\S\ref{contact}. For
higher order operators there is no such equivalence because $J^kE\to J^{k-1}E$
does not factor through $J_H^kE$ for $k\geq 3$. 

For a $k^{\mathrm{th}}$~order operator $D:E\to F$ compatible with the contact 
structure, the {\em enhanced symbol\/} of $D$ is defined to be the composition
$$S_\perp^k\otimes E\to J_H^kE\xrightarrow{\,D\,}F.$$
Its invariance extends Proposition~\ref{invarianceofsecondenhancedsymbol} for
second order operators. In line with~(\ref{piH}), we shall write $\pi_H$ for
the enhanced symbol of $D$ and suppose that it is surjective.

Our next task is to generalise Theorems~\ref{firstcontacttheorem}
and~\ref{firstthreeDcontacttheorem} for contact compatible higher order
operators in the same way that Theorem~\ref{highertheorem} generalises
Theorem~\ref{firsttheorem} for higher order operators in the absence of extra
structure. The following commutative diagram with exact rows and columns
extends (\ref{diagJH}) and defines the bundle~$E_H^\prime$ parallel to the
definition of $E^\prime$ via (\ref{commdiag_k}) in the general case.
$$\begin{array}{ccccccccc} &&0&&0\\
&&\downarrow&&\downarrow\\
0&\to&K_H&\to&E_H^\prime&\to&J_H^{k-1}E&\to&0\\
&&\downarrow&&\downarrow&&\|\\
0&\to&S_\perp^k\otimes E&\to&J_H^kE&\to&J_H^{k-1}E&\to&0\\
&&\makebox[0pt]{\scriptsize$\pi_H\;\;$}\downarrow\makebox[0pt]{}&&
\makebox[0pt]{\scriptsize$D\,$}\downarrow\makebox[0pt]{}\\
&&F&=&F\\
&&\downarrow&&\downarrow\\
&&0&&0
\end{array}$$
Let us first approach the contact version of Theorem~\ref{highertheorem} via 
weighted jet constructions and then make these constructions more explicit by 
means of partial connections.

Proposition~1 of~\cite{neusser} may be interpreted as the existence of a 
canonical isomorphism
$$\frac{J_H^1J_H^kE}{J_H^{k+1}E}\cong
\frac{\Lambda_H^1\otimes J_H^kE}{S_\perp^{k+1}\otimes E}$$
parallel to (\ref{canonicalisomorphism}) in the general case and we shall
denote by $W_H^kE$ the vector bundle defined by either side of this
isomorphism. It follows that there is a canonically defined compatible first
order linear differential operator ${\mathcal{G}}_H:J_H^kE\to W_H^kE$ induced
by the universal first order contact compatible operator 
$j_H^1:J_H^kE\to J_H^1J_H^kE$. If we define an operator 
$\widetilde{D}_H:E_H^\prime\to W_H^kE$ as the restriction of~${\mathcal{G}}_H$,
then we might expect the following result analogous to 
Theorem~\ref{highertheorem}.
\begin{theorem}\label{highercontacttheorem}
There is a contact compatible first order linear differential operator
$$\widetilde{D}_H:E_H^\prime\longrightarrow W_H^kE$$
so that the canonical projection $E_H^\prime\to J_H^{k-1}E\to E$ induces an
isomorphism
$$\{\Sigma\in\Gamma(E_H^\prime)\mbox{\rm\ s.t.\ }\widetilde{D}_H\Sigma=0\}\cong
\{\sigma\in\Gamma(E)\mbox{\rm\ s.t.\ }D\sigma=0\}.$$
\end{theorem}
This theorem is essentially proved in~\cite{neusser} by reasoning with jets. In
order to make the definition of $\widetilde{D}_H$ and this reasoning more
explicit, we might expect to use higher partial connections by analogy with the
constructions of~\S\ref{highergeneral}. However, it turns out that this is not
quite sufficient and we shall also need the adapted connections of
Proposition~\ref{Tconn} or, locally, the Darboux
connections~(\ref{Darbouxconnection}).

Before carrying this out, let us observe that 
Theorem~\ref{highercontacttheorem} is sufficient to start an inductive contact 
prolongation in dimension $\geq 5$. The restricted symbol of ${\mathcal{G}}_H$
$$\Lambda_H^1\otimes J_H^kE\longrightarrow W_H^kE=
\frac{\Lambda_H^1\otimes J_H^kE}{S_\perp^{k+1}\otimes E}$$
is the canonical projection with $S_\perp^{k+1}\otimes E$ as kernel. 
Therefore, the restricted symbol of $\widetilde{D}_H$ is the composition
\begin{equation}\label{restrictedsymbol}
\Lambda_H^1\otimes E_H^\prime\hookrightarrow\Lambda_H^1\otimes J_H^kE
\longrightarrow\frac{\Lambda_H^1\otimes J_H^kE}{S_\perp^{k+1}\otimes E}=W_H^kE
\end{equation}
with $(\Lambda_H^1\otimes K_H)\cap(S_\perp^{k+1}\otimes E)$ as kernel. But 
this is what we should define to be $K_H^\prime$, suppose that it has constant 
rank, and, following the method of \S\ref{contact}, 
\begin{itemize}
\item write $F_H^\prime$ for range of the composition~(\ref{restrictedsymbol});
\item choose a complementary subbundle to $F_H^\prime\subseteq W_H^kE$;
\item define $D_H^\prime:E_H^\prime\to F_H^\prime$ as the resulting projection 
of $\widetilde{D}_H$. 
\end{itemize}
Then $D_H^\prime$ is a first order operator compatible with the contact 
structure, having
\begin{equation}\label{KHprime}
K_H^\prime\equiv(\Lambda_H^1\otimes K_H)\cap(S_\perp^{k+1}\otimes E)
\end{equation}
as the kernel of its restricted symbol. As in \S\ref{contact}, further
prolongation of this first order operator gives first order operators 
$D_H^{\ell}:E_H^{\ell}\to F_H^{\ell},\,\forall\ell\geq 1$ such that
\begin{equation}\label{KHl}
K_H^\ell\equiv(S_\perp^\ell\otimes K_H)\cap(S_\perp^{k+\ell}\otimes E)
\end{equation}
is realised as the kernel of the restricted symbol of~$D_H^\ell$. 

Clearly, we are heading for contact version of Theorem~\ref{improvedspencer}
where the vanishing of (\ref{KHl}) for $\ell$ sufficiently large is the contact
version of finite-type. The final statement is
Theorem~\ref{finalcontacttheorem}. It remains to sort out the $3$-dimensional
case but, before we do so, let us make more explicit the constructions given
above with weighted jets.

We shall content ourselves with a formula for ${\mathcal{G}}_H:J_H^2E\to
W_H^2E$ written in terms of a general partial connection on $E$ and a local
Darboux connection (\ref{Darbouxconnection}) on~$\Lambda^1$. The partial
connection on $E$ canonically lifts to a full connection by
Proposition~\ref{modcons} and the remaining commutation relation for the tensor
connection on $\Lambda^1\otimes E$ is that
\begin{equation}\label{contactcommutation}
\nabla_{[a}\nabla_{b]}\sigma=\kappa_{ab}\sigma-L_{ab}\nabla_0\sigma
\qquad\mbox{for all $\sigma\in\Gamma(E)$}.\end{equation}
The first Goldschmidt operator ${\mathcal{G}}_H:J_H^1E\to W_H^1E$ is then given
by
$$\left[\begin{array}c\sigma\\ \mu_b\end{array}\right]
\longmapsto
\left[\begin{array}c\nabla_a\sigma-\mu_a\\
\nabla_{[a}\mu_{b]}-\kappa_{ab}\sigma\end{array}\right]$$
and its restriction to $E_H^\prime$ coincides with~(\ref{DtildeH}) from the
proof of Theorem~\ref{firstcontacttheorem}. The only difference from
(\ref{firstgoldschmidt}) is that the second entry on the right hand side has 
values in $\Lambda^2_{H\perp}$ and therefore does not see $L_{ab}$
from~(\ref{contactcommutation}). Similarly, the second Goldschmidt operator 
${\mathcal{G}}_H:J_H^2E\to W_H^2E$ is a simple variation on 
(\ref{secondgoldschmidt}) given by
$$\left[\begin{array}c\sigma\\ \mu_b\\ \rho_{bc}+L_{bc}\nu\end{array}\right]
\longmapsto
\left[\begin{array}c\nabla_a-\mu_a\\
\nabla_a\mu_b-\kappa_{ab}\sigma+L_{ab}\nu-\rho_{ab}\\
\nabla_{[a}\rho_{b]c}
-\kappa_{ab}\mu_c
+\kappa_{a[b}\mu_{c]}
-\frac{1}{3}(\nabla_c\kappa_{ab})\sigma
+\frac{1}{3}(\nabla_{[b}\kappa_{c]a})\sigma\\
{}+L_{c[a}\nabla_{b]}\nu-L_{[ab}\nabla_{c]}\nu
\phantom{-\frac{1}{3}(\nabla_c\kappa_{ab})\sigma
+\frac{1}{3}(\nabla_{[b}\kappa_{c]a})\sigma}
\end{array}\right].$$

In the three-dimensional case, we have already seen in \S\ref{three} how to
prolong a first order operator compatible with the contact structure. If
$D:E\to F$ is a compatible operator of order $k\geq 2$, let us still consider
the first order compatible operator $D_H^\prime:E_H^\prime\to F_H^\prime$
constructed above. For this construction, there is nothing special about the
three-dimensional case and the kernel of its restricted symbol is $K_H^\prime$
as in~(\ref{KHprime}). Now if we choose partial connections on $E$ and 
$\Lambda_H^1$, then we have
\begin{equation}\label{DHprime}
D_H^\prime:\begin{array}cJ_H^{k-2}E\\ \oplus\\ S_\perp^{k-1}\otimes E\\ 
\oplus\\ K_H\end{array}
\longrightarrow
\begin{array}c\Lambda_H^1\otimes J_H^{k-2}E\\ \oplus\\ 
\Lambda_H^1\otimes S_\perp^{k-1}\otimes E\;,\\ \oplus\\
(\Lambda_H^1\otimes K_H)/K_H^\prime
\end{array}\end{equation}
a first order differential operator of the form
$$\left[\begin{array}c\sigma\\ \mu\\ 
\rho\end{array}\right]\stackrel{D_H^\prime}{\longmapsto}
\left[\begin{array}c\nabla_H\sigma-\mu\\ \nabla_H\mu+L(\sigma)-\rho\\ 
\nabla_H\rho+M(\sigma,\mu)\enskip\bmod K_H^{\prime\prime}\end{array}\right]$$
for some smooth homomorphisms 
$$L:J_H^{k-2}E\to \Lambda_H^1\otimes S_\perp^{k-1}\otimes E\quad\mbox{and}\quad 
M:J_H^{k-2}E\oplus(S_\perp^{k-1}\otimes E)\to\Lambda_H^1\otimes K_H.$$
But, since $K_H\subseteq S_\perp^k\otimes E$ and 
$$(\Lambda_H^1\otimes S_\perp^k)\cap(S_\perp^3\otimes 
S_\perp^{k-2})=S_\perp^{k+1},\enskip\mbox{for }k\geq2$$
we see that we can rewrite $K_H^\prime$ as
$$(\Lambda_H^1\otimes K_H)\cap(S_\perp^3\otimes J_H^{k-2}E).$$
Therefore, viewed as in~(\ref{DHprime}), the operator $D_H^\prime$ is precisely
of the type to which Proposition~\ref{iterum} below may be applied. Thus starts
an iterative process for building contact prolongations of~$D$, crucially
employing that $D$ is of order at least~$2$.

The following lemma is convenient for the proof of Proposition~\ref{iterum}.

\begin{lemma}\label{tautology}
Suppose ${\mathcal{P}}:V\to W$ is a first order operator on a contact manifold
compatible with the contact structure. Suppose its partial symbol
$\pi:\Lambda_H^1\otimes V\to W$ has constant rank. Then we can find a partial
connection $\nabla_H$ on $V$ and a homomorphism $\theta:V\to W$ such that
$${\mathcal{P}}=\pi\circ\nabla_H+\theta.$$
\end{lemma}
\begin{proof} Choose a complement $C$ to the range $R$ of $\pi$ in $W$ and take
$\theta:V\to W$ to be the composition of ${\mathcal{P}}$ with projection 
to~$C$. It is a homomorphism of vector bundles and ${\mathcal{P}}-\theta$ is a 
compatible first order operator $V\to R$, which therefore may be written as 
$\pi\circ\nabla_H$ for an appropriate choice of partial connection~$\nabla_H$.
\end{proof}
As presaged in~\S\ref{three}, we now formulate and prove a result that can be 
iterated in $3$-dimensions to give the prolongations we require. The method of 
proof is a variation on the discussion in \S\ref{three} but there are some 
extra points to be borne in mind, namely 
\begin{itemize}
\item that equality $K_H^\prime=\Lambda_H^1\otimes K_H$ is
weakened to inclusion $K_H^\prime\subseteq\Lambda_H^1\otimes K_H$;
\item that an extra linear term $L(\sigma)$ is allowed in the definition 
of~$D_H$.
\end{itemize}
\begin{proposition}\label{iterum}
Suppose $E$ is a smooth vector bundle on a $3$-dimensional contact manifold. 
Suppose we are given subbundles $K_H\subseteq\Lambda_H^1\otimes E$ and 
$K_H^\prime\subseteq\Lambda_H^1\otimes K_H$ such that 
$K_H^{\prime\prime}
\equiv(\Lambda_H^1\otimes K_H^\prime)\cap(S_\perp^3\otimes E)\subset
\bigotimes^3\!\Lambda_H^1\otimes E$ is also a subbundle.
Suppose
$$D_H:\begin{array}cE\\ \oplus\\ K_H\\ \oplus\\ K_H^\prime\end{array}
\longrightarrow
\begin{array}c\Lambda_H^1\otimes E\\ \oplus\\ \Lambda_H^1\otimes K_H\\ \oplus\\
(\Lambda_H^1\otimes K_H^\prime)/K_H^{\prime\prime}
\end{array}$$
is a first order differential operator of the form
$$\left[\begin{array}c\sigma\\ \mu\\ 
\rho\end{array}\right]\stackrel{D_H}{\longmapsto}
\left[\begin{array}c\nabla_H\sigma-\mu\\ \nabla_H\mu+L(\sigma)-\rho\\ 
\nabla_H\rho+M(\sigma,\mu)\enskip\bmod K_H^{\prime\prime}\end{array}\right]$$
for some partial connections on $E$, $K_H$, and $K_H^\prime$ and smooth 
homomorphisms 
$$L:E\to \Lambda_H^1\otimes K_H\quad\mbox{and}\quad 
M:E\oplus K_H\to\Lambda_H^1\otimes K_H^\prime.$$
Finally suppose that $K_H^{\prime\prime\prime}
\equiv(\Lambda_H^1\otimes K_H^{\prime\prime})\cap(S_\perp^4\otimes E)\subset
\bigotimes^4\!\Lambda_H^1\otimes E$ is a subbundle.
Then we can find a differential operator
$$D_H^\prime:\begin{array}cE\\ \oplus\\ K_H\\ \oplus\\ K_H^\prime\\ \oplus\\
K_H^{\prime\prime}\end{array}
\longrightarrow
\begin{array}c\Lambda_H^1\otimes E\\ \oplus\\ \Lambda_H^1\otimes K_H\\ \oplus\\
\Lambda_H^1\otimes K_H^\prime\\ \oplus\\
(\Lambda_H^1\otimes K_H^{\prime\prime})/K_H^{\prime\prime\prime}
\end{array}$$
of the form
$$\left[\begin{array}c\sigma\\ \mu\\ \rho\\ \tau\end{array}\right]
\stackrel{D_H^\prime}{\longmapsto}
\left[\begin{array}c\nabla_H\sigma-\mu\\ \nabla_H\mu+L(\sigma)-\rho\\ 
\nabla_H\rho+M(\sigma,\mu)-\tau\\
\nabla_H\tau+N(\sigma,\mu,\rho)\enskip\bmod K_H^{\prime\prime\prime}
\end{array}\right]$$
with the same kernel as~$D_H$. 
\end{proposition}
\begin{proof}
Certainly, we may rewrite the last line of $D_H\Phi=0$ as
\begin{equation}\label{bottomline}
\nabla_H\rho+M(\sigma,\mu)-\tau=0,\enskip\mbox{for some }
\tau\in\Gamma(K_H^{\prime\prime})\end{equation}
and it suffices to derive a differential equation on $\tau$ of the given form. 
A suitable equation arises from the second line of $D_H\Phi=0$, namely
$$\nabla_H\mu=\rho-L(\sigma),$$
by an application of the coupled Rumin operator 
$\nabla_H^{(2)}:\Lambda_H^1\otimes
K_H\to\Lambda_H^1\otimes L\otimes K_H$ discussed in~\S\ref{three}. We conclude
that 
\begin{equation}\label{whatweconclude}
\nabla_H^{(2)}\rho=\kappa_H^\prime\mu+\nabla_H^{(2)}(L(\sigma)).\end{equation}
We would like to rewrite both $\nabla_H^{(2)}\rho$ and 
$\nabla_H^{(2)}(L(\sigma))$, noting that $\rho\in\Gamma(K_H^\prime)$ whereas 
$L(\sigma)\in\Gamma(\Lambda_H^1\otimes K_H)$. Dealing with $\rho$ first, we 
may use Lemma~\ref{firstordertrick} to write
$$\nabla_H^{(2)}\rho={\mathcal{P}}\nabla_H\rho+\Theta\rho$$
where $\nabla_H$ is the given partial connection on~$K_H^\prime$. {From}
(\ref{bottomline}) we may substitute for $\nabla_H\rho$ and conclude that 
$$\nabla_H^{(2)}\rho={\mathcal{P}}(\tau-M(\sigma,\mu))+\Theta\rho.$$
If we now use Lemma~\ref{tautology} to rewrite ${\mathcal{P}}:
\Lambda_H^1\otimes K_H^\prime\to\Lambda_H^1\otimes L\otimes K_H$, then 
\begin{equation}\label{ruminrho}
\nabla_H^{(2)}\rho=\pi\nabla_H\tau-\pi\nabla_H(M(\sigma,\mu))
+\theta(\tau-M(\sigma,\mu))+\Theta\rho\end{equation}
where $\nabla_H$ is a partial connection on $\Lambda_H^1\otimes K_H^\prime$ and
$\pi$ is the composition
\begin{equation}\label{thisispi}
\Lambda_H^1\otimes\Lambda_H^1\otimes K_H^\prime\hookrightarrow
\Lambda_H^1\otimes\Lambda_H^1\otimes\Lambda_H^1\otimes K_H
\xrightarrow{\,\Sigma\otimes{\mathrm{Id}}\,}\Lambda_H^1\otimes L\otimes K_H
\end{equation}
because $\Sigma\otimes{\mathrm{Id}}$ is the enhanced symbol 
of~$\nabla_H^{(2)}$, as noted in Proposition~\ref{calculateenhancedsymbol}. 
Recall that $M:E\oplus K_H\to\Lambda_H^1\otimes K_H^\prime$ is a smooth
homomorphism between bundles on which we already have defined various partial
connections. We may therefore use the partial Leibniz rule to expand
$$\nabla_H(M(\sigma,\mu))=
(\nabla_HM)(\sigma,\mu)+M(\nabla_H\sigma,\nabla_H\mu)$$
and substitute from the first and second lines of $D_H\Phi=0$ for
$\nabla_H\sigma$ and~$\nabla_H\mu$ to rewrite (\ref{ruminrho}) as
\begin{equation}\label{ruminrhorewritten}
\nabla_H^{(2)}\rho=\pi\nabla_H\tau+\theta\tau+S(\sigma,\mu,\rho),\end{equation}
where $\theta:K_H^{\prime\prime}\to\Lambda_H^1\otimes L\otimes K_H$ and
$S:E\oplus K_H\oplus K_H^\prime\to\Lambda_H^1\otimes L\otimes K_H$ are
smooth homomorphisms. Unravelling $\nabla_H^{(2)}(L(\sigma))$ is a similar
exercise except that in using Lemma~\ref{firstordertrick} we are obliged to
choose a partial connection on $\Lambda_H^1\otimes K_H$. Schematically, the
calculation reads
$$\begin{array}{rcl}\nabla_H^{(2)}(L(\sigma))
&=&{\mathcal{P}}(\nabla_H(L(\sigma)))+\Theta(L(\sigma))\\
&=&{\mathcal{P}}((\nabla_HL)(\sigma)+L(\mu))+\Theta(L(\sigma))\\
&=&\pi\nabla_H((\nabla_HL)(\sigma)+L(\mu))
+\theta((\nabla_HL)(\sigma)+L(\mu))+\Theta(L(\sigma))\\
&=&\pi(\nabla_H(\nabla_HL))(\sigma)+2\pi(\nabla_HL)(\mu)
+\pi L(\rho-L(\sigma))\\
&&\phantom{\pi\nabla_H((\nabla_HL)(\sigma)+L(\mu))}
+\theta((\nabla_HL)(\sigma)+L(\mu))+\Theta(L(\sigma))
\end{array}$$
and the result is that
\begin{equation}\label{theresult}
\nabla_H^{(2)}(L(\sigma))=T(\sigma,\mu,\rho)\end{equation}
for some smooth homomorphism 
$T:E\oplus K_H\oplus K_H^\prime\to\Lambda_H^1\otimes L\otimes K_H$.
Substituting (\ref{ruminrhorewritten}) and (\ref{theresult}) into
(\ref{whatweconclude}), we find that
\begin{equation}\label{whatwefind}
\pi\nabla_H\tau+\theta\tau+(S-T)(\sigma,\mu,\rho)
-\kappa_H^\prime\mu=0.\end{equation}
The left hand side is a section of $\Lambda_H^1\otimes L\otimes K_H$. Let us 
consider the differential operator
$$K_H^{\prime\prime}\ni\tau\longmapsto\pi\nabla_H\tau+\theta\tau\in
\Lambda_H^1\otimes L\otimes K_H.$$
According to~(\ref{thisispi}), the partial symbol of this operator is 
$(\Sigma\otimes{\mathrm{Id}})|_{\Lambda_H^1\otimes K_H^{\prime\prime}}$, where
$K_H^{\prime\prime}\subseteq\Lambda_H^1\otimes K_H^\prime
\subseteq\Lambda_H^1\otimes\Lambda_H^1\otimes K_H$ and,
according to~(\ref{globalhom}), its kernel is 
$$(\Lambda_H^1\otimes K_H^{\prime\prime})\cap(S_\perp^3\otimes K_H).$$
However, recalling that 
\begin{itemize}
\item 
$K_H^{\prime\prime}=(\Lambda_H^1\otimes K_H^\prime)\cap(S_\perp^3\otimes E)$;
\item 
$K_H\subseteq\Lambda_H^1\otimes E$;
\item
$(\Lambda_H^1\otimes S_\perp^3)\cap(S_\perp^3\otimes\Lambda_H^1)=S_\perp^4$,
\end{itemize}
we see that 
$$(\Lambda_H^1\otimes K_H^{\prime\prime})\cap(S_\perp^3\otimes K_H)=
(\Lambda_H^1\otimes K_H^{\prime\prime})\cap(S_\perp^4\otimes E),$$
which is precisely how $K_H^{\prime\prime\prime}$ is defined in the 
statement of the proposition. In summary, the kernel of the symbol of 
$\tau\mapsto\pi\nabla_H\tau+\theta\tau$ is~$K_H^{\prime\prime\prime}$. We 
shall use $\Sigma\otimes{\mathrm{Id}}$ to identify 
$$\Lambda_H^1\otimes L\otimes K_H\supseteq{\mathrm{range}}
((\Sigma\otimes{\mathrm{Id}})|_{\Lambda_H^1\otimes K_H^{\prime\prime}})=
(\Lambda_H^1\otimes K_H^{\prime\prime})/K_H^{\prime\prime\prime},$$
choose a smooth splitting
$$\delta:\Lambda_H^1\otimes L\otimes K_H\to
(\Lambda_H^1\otimes K_H^{\prime\prime})/K_H^{\prime\prime\prime},$$
and consider 
\begin{equation}\label{whatweconsider}
\delta(\pi\nabla_H\tau+\theta\tau)+\delta((S-T)(\sigma,\mu,\rho)
-\kappa_H^\prime\mu)\end{equation}
whose vanishing is a consequence of~(\ref{whatwefind}). We have just arranged 
that the symbol of the operator 
$\tau\mapsto\delta(\pi\nabla_H\tau+\theta\tau)$ is the canonical projection
$$\Lambda_H^1\otimes K_H^{\prime\prime}\to
(\Lambda_H^1\otimes K_H^{\prime\prime})/K_H^{\prime\prime\prime}$$
and the remaining part of (\ref{whatweconsider}) is some homomorphism
$K_H^{\prime\prime}\to
(\Lambda_H^1\otimes K_H^{\prime\prime})/K_H^{\prime\prime\prime}$.
It remains to choose a partial connection on $K_H^{\prime\prime}$ and a lift 
of this homomorphism to $\Lambda_H^1\otimes K_H^{\prime\prime}$ to write
(\ref{whatweconsider}) in the form in the form
$$\nabla_H\tau+N(\sigma,\mu,\rho)\enskip\bmod K_H^{\prime\prime\prime},$$
as required.\end{proof}
It remains to see why Proposition~\ref{iterum} can be iterated. The point is 
that we may regroup the output from this Proposition as follows.
$$\begin{array}c
\makebox[0pt][r]{$\widetilde{E}=\left\{\rule{0pt}{20pt}\right.\!\!\!$}
\begin{array}cE\\ \oplus\\ K_H\end{array}\\ 
\oplus\\ 
\makebox[0pt][r]{$\widetilde{K}_H={}$}K_H^\prime\\ \oplus\\
\makebox[0pt][r]{$\widetilde{K}_H^\prime={}$}K_H^{\prime\prime}\end{array}
\ni\left[\begin{array}c\sigma\\ \mu\\ \rho\\ \tau\end{array}\right]
\stackrel{D_H^\prime}{\longmapsto}
\left[\begin{array}c\nabla_H\sigma-\mu\\ \nabla_H\mu+L(\sigma)-\rho\\ 
\vdots\\[-7pt] \vdots
\end{array}\right]\in\!\!\!
\begin{array}c\begin{array}c\Lambda_H^1\otimes E\\ \oplus\\ 
\Lambda_H^1\otimes K_H\end{array}
\makebox[0pt][l]{$\!\!\!\left.\rule{0pt}{20pt}\right\}
=\Lambda_H^1\otimes\widetilde{E}$}\\ \oplus\\
\Lambda_H^1\otimes K_H^\prime\\ \oplus\\
(\Lambda_H^1\otimes K_H^{\prime\prime})/K_H^{\prime\prime\prime}
\end{array}$$
In order to substitute back into Proposition~\ref{iterum}, the crux is to note
that
\begin{itemize}
\item $\left[\begin{array}c\sigma\\ \mu\end{array}\right]\mapsto
\left[\begin{array}c\nabla_H\sigma-\mu\\ \nabla_H\mu+L(\sigma)
\end{array}\right]$ is a partial connection on~$\widetilde{E}$;
\item \rule{0pt}{15pt}$\widetilde{K}_H^{\prime\prime}\equiv
(\Lambda_H^1\otimes\widetilde{K}_H^\prime)\cap(S_\perp^3\otimes\widetilde{E})=
(\Lambda_H^1\otimes K_H^{\prime\prime})\cap(S_\perp^4\otimes E)
\equiv K_H^{\prime\prime\prime}$.
\end{itemize}
In order to interpret the new output, the crux is to note that
$$\widetilde{K}_H^{\prime\prime\prime}\equiv
(\Lambda_H^1\otimes\widetilde{K}_H^{\prime\prime})\cap
(S_\perp^4\otimes\widetilde{E})=
(\Lambda_H^1\otimes K_H^{\prime\prime\prime})\cap(S_\perp^5\otimes E)
\equiv K_H^{\prime\prime\prime\prime}.$$
For completeness, here is the final conclusion in all dimensions and for
operators of arbitrary order. 
\begin{theorem}\label{finalcontacttheorem}
Suppose $D:E\to F$ is a $k^{\mbox{\scriptsize th}}$ order linear differential
operator between smooth vector bundles on a contact manifold. Suppose that it
is compatible with the contact structure and that its enhanced symbol
$S_\perp^k\otimes E\to F$ is surjective with kernel~$K_H$. Suppose that
$$K_H^\ell\equiv(S_\perp^\ell\otimes K_H)\cap(S_\perp^{k+\ell}\otimes E)$$
are vector bundles for all~$\ell$ (we say that $D$ is `regular') and that
$K_H^\ell=0$ for $\ell$ sufficiently large (we say that $D$ is `finite-type').
Then there is a computable partial connection $\nabla_H$ on the
bundle
$$\textstyle{\mathbb{T}}\equiv J_H^{k-1}E\oplus K_H\oplus
\bigoplus_{\ell\geq 1}\!K_H^\ell$$
such that taking the first component ${\mathbb{T}}\to J_H^{k-1}E\to E$ 
induces an isomorphism
$$\{\Sigma\in\Gamma({\mathbb{T}})\mbox{\rm\ s.t.\ }\nabla_H\Sigma=0\}\cong
\{\sigma\in\Gamma(E)\mbox{\rm\ s.t.\ }D\sigma=0\}.$$
In particular, the solution space of $D$ is finite-dimensional with dimension
bounded by the rank of\/~${\mathbb{T}}$.
\end{theorem}

\section{Geometric operators}\label{geometric}
Although their definition is simple enough and determined purely in terms
of the given subbundle $K_H\subset S_\perp^k\otimes E$, the spaces
\begin{equation}\label{KHell}
K_H^\ell\equiv(S_\perp^\ell\otimes K_H)\cap(S_\perp^{k+\ell}\otimes E)
\end{equation}
can be hard to understand. For a wide class of geometrically natural examples,
however, these bundles can be sensibly computed. The corresponding operators
are seen to be finitely determined and we obtain sharp bounds on the dimension
of their solution spaces. The key ingredient is Kostant's
computation of certain Lie algebra cohomologies~\cite{kostant} and our approach
follows \cite{bceg} where similar reasoning was used in the case of classical
prolongation.

To proceed, let us recall from \S\ref{contact} that we are writing the
dimension of our contact manifold as $2n+1$ and let us realise the Lie algebras
${\mathfrak{sp}}(2n,{\mathbb{R}})$ and ${\mathfrak{sp}}(2(n+1),{\mathbb{R}})$
as matrices of the form
$$\left[\begin{array}{cc}\mbox{\Large$A$}&\mbox{\Large$B$}\\
\mbox{\Large$C$}&\raisebox{2pt}{$-$}\mbox{\Large$A^t$}\end{array}\right]
\quad\mbox{and}\quad
\left[\begin{array}{cccc}\lambda&-q^t&p^t&\alpha\\[4pt]
\raisebox{3pt}{$r$}&\mbox{\Large$A$}&\mbox{\Large$B$}&\raisebox{3pt}{$p$}\\
\raisebox{3pt}{$s$}&\mbox{\Large$C$}&
\raisebox{2pt}{$-$}\mbox{\Large$A^t$}&\raisebox{3pt}{$q$}\\[2pt]
\beta&s^t&-r^t&-\lambda
\end{array}\right]$$
respectively, where $B$ and $C$ are symmetric $n\times n$ real matrices, $A$ is
an arbitrary $n\times n$ real matrix, $\lambda,\alpha,\beta$ are real numbers,
and $p,q,r,s$ are real $n$-vectors. As the notation suggests, we have
${\mathfrak{sp}}(2n,{\mathbb{R}})\hookrightarrow
{\mathfrak{sp}}(2(n+1),{\mathbb{R}})$ an embedding of Lie algebras. The adjoint
action of the `grading element'
\begin{equation}\label{ge}\left[
\begin{array}{cccc}1&0&0&0\\ 0&0&0&0\\ 0&0&0&0\\ 0&0&0&-1\end{array}
\right]\end{equation}
splits ${\mathfrak{g}}\equiv{\mathfrak{sp}}(2(n+2),{\mathbb{R}})$ into
eigenspaces
\begin{equation}\label{Cdeomposition}
{\mathfrak{g}}={\mathfrak{g}}_{-2}\oplus{\mathfrak{g}}_{-1}\oplus
{\mathfrak{g}}_{0}\oplus{\mathfrak{g}}_{1}\oplus{\mathfrak{g}}_{2}
\end{equation}
containing elements with the form
$$\mbox{\footnotesize$\left[
\begin{array}{cccc}0&0&0&0\\ 0&0&0&0\\ 0&0&0&0\\ *&0&0&0\end{array}
\right]\quad
\left[
\begin{array}{cccc}0&0&0&0\\ *&0&0&0\\ *&0&0&0\\ 0&*&*&0\end{array}
\right]\quad
\left[
\begin{array}{cccc}*&0&0&0\\ 0&*&*&0\\ 0&*&*&0\\ 0&0&0&*\end{array}
\right]\quad
\left[
\begin{array}{cccc}0&*&*&0\\ 0&0&0&*\\ 0&0&0&*\\ 0&0&0&0\end{array}
\right]\quad
\left[
\begin{array}{cccc}0&0&0&*\\ 0&0&0&0\\ 0&0&0&0\\ 0&0&0&0\end{array}
\right],$}$$
respectively. Let us denote by ${\mathfrak{g}}_-$ the Lie subalgebra
${\mathfrak{g}}_{-2}\oplus{\mathfrak{g}}_{-1}$ of~${\mathfrak{g}}$. As a Lie
algebra is its own right, it is usually referred to as the {\em Heisenberg
algebra}. Suppose ${\mathbb{V}}$ is a finite-dimensional representation
of~${\mathfrak{g}}_-$. Then we may define linear transformations
\begin{equation}\label{Liecomplex}0\to
{\mathbb{V}}\xrightarrow{\,\partial\,}\Hom({\mathfrak{g}}_-,{\mathbb{V}})
\xrightarrow{\,\partial\,}\Hom(\Lambda^2{\mathfrak{g}}_-,{\mathbb{V}})
\end{equation}
by $(\partial v)(X)\equiv Xv$ and
$(\partial\phi)(X\wedge Y)\equiv\phi([X,Y])-X\phi(Y)+Y\phi(X)$,
respectively. It is easily verified that $\partial^2=0$ and we define
the {\em Lie algebra cohomologies}
$$H^0({\mathfrak{g}}_-,{\mathbb{V}})\equiv
\ker\partial:{\mathbb{V}}\to\Hom({\mathfrak{g}}_-,{\mathbb{V}})$$
and
$$H^1({\mathfrak{g}}_-,{\mathbb{V}})\equiv
\frac{\ker\partial:\Hom({\mathfrak{g}}_-,{\mathbb{V}})\to
\Hom({\Lambda^2\mathfrak{g}}_-,{\mathbb{V}})}
{\im\partial:{\mathbb{V}}\to\Hom({\mathfrak{g}}_-,{\mathbb{V}})}.$$
\begin{lemma}\label{en_is_one}
Suppose $n=1$. Then $H^0({\mathfrak{g}}_-,{\mathbb{V}})$ and
$H^1({\mathfrak{g}}_-,{\mathbb{V}})$ may be computed using the complex
\begin{equation}\label{replaceLiecomplexthree}0\to
{\mathbb{V}}\xrightarrow{\,\partial_H\,}\Hom({\mathfrak{g}}_{-1},{\mathbb{V}})
\xrightarrow{\,\partial_H\,}
\Hom({\mathfrak{g}}_{-1}\otimes{\mathfrak{g}}_{-2},{\mathbb{V}})
\end{equation}
instead of {\rm(\ref{Liecomplex})}, where $(\partial_Hv)(X)\equiv Xv$ and
$$(\partial_H\phi)(X\otimes[Y,Z])\equiv
[Y,Z]\phi(X)-X\big(Y\phi(Z)-Z\phi(Y)\big)$$
for $X,Y,Z\in{\mathfrak{g}}_{-1}$.
\end{lemma}
\begin{proof}
The composition
\begin{equation}\label{zigzag}
\Hom({\mathfrak{g}}_{-2},{\mathbb{V}})\to\Hom({\mathfrak{g}}_-,{\mathbb{V}})
\xrightarrow{\,\partial\,}\Hom(\Lambda^2{\mathfrak{g}}_-,{\mathbb{V}})
\to\Hom(\Lambda^2{\mathfrak{g}}_{-1},{\mathbb{V}})\end{equation}
is given by $\phi\longmapsto\big(X\wedge Y\mapsto\phi([X,Y])\big)$. In other
words, we may use the isomorphism
\begin{equation}\label{bracket}
\Lambda^2{\mathfrak{g}}_{-1}\ni X\wedge Y\longmapsto
[X,Y]\in{\mathfrak{g}}_{-2}\end{equation}
to eliminate the composition (\ref{zigzag}) from
$$\begin{array}c\Hom({\mathfrak{g}}_{-2},{\mathbb{V}})\\
\oplus\\
\Hom({\mathfrak{g}}_{-1},{\mathbb{V}})\end{array}
\!=\Hom({\mathfrak{g}}_-,{\mathbb{V}})\xrightarrow{\enskip\partial\enskip}
\Hom(\Lambda^2{\mathfrak{g}}_-,{\mathbb{V}})=\!
\begin{array}c
\Hom({\mathfrak{g}}_{-1}\otimes{\mathfrak{g}}_{-2},{\mathbb{V}})\\ \oplus\\
\Hom(\Lambda^2{\mathfrak{g}}_{-1},{\mathbb{V}})\end{array}$$
to leave the complex (\ref{replaceLiecomplexthree}), as required.
\end{proof}
\begin{lemma} Suppose $n\geq 2$.  Then $H^0({\mathfrak{g}}_-,{\mathbb{V}})$ and
$H^1({\mathfrak{g}}_-,{\mathbb{V}})$ may be computed using the complex
\begin{equation}\label{replaceLiecomplex}0\to
{\mathbb{V}}\xrightarrow{\,\partial_H\,}\Hom({\mathfrak{g}}_{-1},{\mathbb{V}})
\xrightarrow{\,\partial_H\,}
\Hom(\Lambda_\perp^2{\mathfrak{g}}_{-1},{\mathbb{V}})
\end{equation}
induced from {\rm(\ref{Liecomplex})}, where
$\Lambda_\perp^2{\mathfrak{g}}_{-1}$ is the kernel of the Lie bracket
homomorphism~{\rm(\ref{bracket})}.
\end{lemma}
\begin{proof} A crucial difference for $n\geq 2$ is that the complex
(\ref{Liecomplex}) may be replaced by
\begin{equation}\label{replacement}0\to
{\mathbb{V}}\xrightarrow{\,\partial\,}\Hom({\mathfrak{g}}_-,{\mathbb{V}})
\xrightarrow{\,\partial\,}\Hom(\Lambda^2{\mathfrak{g}}_{-1},{\mathbb{V}})
\end{equation}
without changing the cohomology. In other words, the kernel of the homomorphism
$\partial:\Hom({\mathfrak{g}}_-,{\mathbb{V}})\to
\Hom(\Lambda^2{\mathfrak{g}}_-,{\mathbb{V}})$ is the same as the kernel of
the composition
$$\Hom({\mathfrak{g}}_-,{\mathbb{V}})
\xrightarrow{\,\partial\,}\Hom(\Lambda^2{\mathfrak{g}}_-,{\mathbb{V}})\to
\Hom(\Lambda^2{\mathfrak{g}}_{-1},{\mathbb{V}}).$$
To see this, suppose that $\phi\in\Hom({\mathfrak{g}}_-,{\mathbb{V}})$ and
$(\partial\phi)(\omega)=0$ for all $\omega\in\Lambda^2{\mathfrak{g}}_{-1}$.
Then, according to the decomposition
$$\Lambda^2{\mathfrak{g}}_-=({\mathfrak{g}}_{-2}\wedge{\mathfrak{g}}_{-1})
\oplus\Lambda^2{\mathfrak{g}}_{-1},$$
we must show that $(\partial\phi)(\omega)=0$ for all
$\omega\in{\mathfrak{g}}_{-2}\wedge{\mathfrak{g}}_{-1}$.
As $n\geq 2$, the homomorphism
$$\Lambda^3{\mathfrak{g}}_{-1}\ni X\wedge Y\wedge Z\mapsto
[X,Y]\otimes Z+[Y,Z]\otimes X+[Z,X]\otimes Y\in
{\mathfrak{g}}_{-2}\otimes{\mathfrak{g}}_{-1}$$
is surjective and we find
$$\begin{array}l
(\partial\phi)([X,Y]\wedge Z+[Y,Z]\wedge X+[Z,X]\wedge Y)=\\[4pt]
Z\phi([X,Y])-[X,Y]\phi(Z)+X\phi([Y,Z])-[Y,Z]\phi(X)+Y\phi([Z,X])
-[Z,X]\phi(Y)=\\[4pt]
X(\partial\phi)(Y\wedge Z)+Y(\partial\phi)(Z\wedge X)
+Z(\partial\phi)(X\wedge Y)=0,\end{array}$$
as required. Having replaced (\ref{Liecomplex}) by (\ref{replacement}), we may
now argue as in the proof of Lemma~\ref{en_is_one}.
Specifically, we may cancel
$\Hom({\mathfrak{g}}_{-2},{\mathbb{V}})$ inside
$\Hom({\mathfrak{g}}_-,{\mathbb{V}})$ with its image in
$\Hom(\Lambda^2{\mathfrak{g}}_{-1},{\mathbb{V}})$, leaving the
complex~(\ref{replaceLiecomplex}). This completes our proof.
\end{proof}
\begin{remark}The Killing form on
${\mathfrak{g}}={\mathfrak{sp}}(2(n+1),{\mathbb{R}})$ canonically identifies
the duals of ${\mathfrak{g}}_{-1}$ and ${\mathfrak{g}}_{-2}$ with
${\mathfrak{g}}_1$ and ${\mathfrak{g}}_2$, respectively.
It is sometimes convenient to rewrite the complexes
(\ref{replaceLiecomplexthree}) and (\ref{replaceLiecomplex}) as
\begin{equation}\label{bestLiecomplexthree}0\to
{\mathbb{V}}\xrightarrow{\,\partial_H\,}{\mathfrak{g}}_1\otimes{\mathbb{V}}
\xrightarrow{\,\partial_H\,}
{\mathfrak{g}}_1\otimes{\mathfrak{g}}_2\otimes{\mathbb{V}}\end{equation}
and
\begin{equation}\label{bestLiecomplex}0\to
{\mathbb{V}}\xrightarrow{\,\partial_H\,}{\mathfrak{g}}_1\otimes{\mathbb{V}}
\xrightarrow{\,\partial_H\,}
\Lambda_\perp^2{\mathfrak{g}}_1\otimes{\mathbb{V}}\end{equation}
respectively, where $\Lambda_\perp^2{\mathfrak{g}}_1$ is the kernel of the Lie
bracket $\Lambda^2{\mathfrak{g}}_1\to{\mathfrak{g}}_2$.
\end{remark}
\begin{remark} The reader will have noticed that the distinction between the
cases $n=1$ and $n\geq 2$ resembles the distinction found earlier in contact
geometry, namely the algebraic complex (\ref{bestLiecomplex}) closely resembles
the complex (\ref{rumin}) of differential operators whilst
(\ref{bestLiecomplexthree}) follows~(\ref{threeDrumin}). This observation
continues into higher cohomology: the Lie algebra cohomology
$H^q({\mathfrak{g}}_-,{\mathbb{V}})$ for any representation ${\mathbb{V}}$ of
the Heisenberg Lie algebra ${\mathfrak{g}}_-$ is defined by an algebraic
complex resembling the de~Rham complex but may be computed by an alternative
algebraic complex following the Rumin complex \cite{rumin} in contact geometry.
A precise explanation for this observation may be obtained from the usual
interpretation \cite{knapp} of $H^q({\mathfrak{g}}_-,{\mathbb{V}})$ as
equivariant de~Rham cohomology on the Heisenberg group.\end{remark}

When the representation ${\mathbb{V}}$, of ${\mathfrak{g}}_-$, is obtained from
an irreducible finite-dimensional representation of ${\mathfrak{g}}$ by
restriction to~${\mathfrak{g}}_-$, the Lie algebra cohomology
$H^q({\mathfrak{g}}_-,{\mathbb{V}})$ is\linebreak 
computed by a theorem of Kostant~\cite{kostant}. If we characterise such
representations of ${\mathfrak{g}}$ by\linebreak 
means of their highest weight written as an integral combination of the
fundamental weights in the usual way, then for $k\geq 1$ and non-negative
integers $a,b,c,\cdots d,e$,
\begin{equation}\label{Hzero}
H^0({\mathfrak{g}}_-,\begin{picture}(85,20)(-20,5)
\put(-13,8){\makebox(0,0){$\bullet$}}
\put(-13,13){\makebox(0,0)[b]{$\scriptstyle k-1$}}
\put(2,13){\makebox(0,0)[b]{$\scriptstyle a$}}
\put(12,13){\makebox(0,0)[b]{$\scriptstyle b$}}
\put(22,13){\makebox(0,0)[b]{$\scriptstyle c$}}
\put(50,13){\makebox(0,0)[b]{$\scriptstyle d$}}
\put(60,13){\makebox(0,0)[b]{$\scriptstyle e$}}
\put(2,8){\makebox(0,0){$\bullet$}}
\put(12,8){\makebox(0,0){$\bullet$}}
\put(22,8){\makebox(0,0){$\bullet$}}
\put(-13,8){\line(1,0){40}}
\put(50,8){\line(-1,0){5}}
\put(37,8){\makebox(0,0){$\cdots$}}
\put(55,8){\makebox(0,0){$\scriptstyle\langle$}}
\put(50,8){\makebox(0,0){$\bullet$}}
\put(60,8){\makebox(0,0){$\bullet$}}
\put(50,9){\line(1,0){10}}
\put(50,7){\line(1,0){10}}
\end{picture})=\begin{picture}(65,20)(0,5)
\put(2,13){\makebox(0,0)[b]{$\scriptstyle a$}}
\put(12,13){\makebox(0,0)[b]{$\scriptstyle b$}}
\put(22,13){\makebox(0,0)[b]{$\scriptstyle c$}}
\put(50,13){\makebox(0,0)[b]{$\scriptstyle d$}}
\put(60,13){\makebox(0,0)[b]{$\scriptstyle e$}}
\put(2,8){\makebox(0,0){$\bullet$}}
\put(12,8){\makebox(0,0){$\bullet$}}
\put(22,8){\makebox(0,0){$\bullet$}}
\put(2,8){\line(1,0){25}}
\put(50,8){\line(-1,0){5}}
\put(37,8){\makebox(0,0){$\cdots$}}
\put(55,8){\makebox(0,0){$\scriptstyle\langle$}}
\put(50,8){\makebox(0,0){$\bullet$}}
\put(60,8){\makebox(0,0){$\bullet$}}
\put(50,9){\line(1,0){10}}
\put(50,7){\line(1,0){10}}
\end{picture}\end{equation}
and
\begin{equation}\label{Hone}
H^1({\mathfrak{g}}_-,\begin{picture}(85,20)(-20,5)
\put(-13,8){\makebox(0,0){$\bullet$}}
\put(-13,13){\makebox(0,0)[b]{$\scriptstyle k-1$}}
\put(2,13){\makebox(0,0)[b]{$\scriptstyle a$}}
\put(12,13){\makebox(0,0)[b]{$\scriptstyle b$}}
\put(22,13){\makebox(0,0)[b]{$\scriptstyle c$}}
\put(50,13){\makebox(0,0)[b]{$\scriptstyle d$}}
\put(60,13){\makebox(0,0)[b]{$\scriptstyle e$}}
\put(2,8){\makebox(0,0){$\bullet$}}
\put(12,8){\makebox(0,0){$\bullet$}}
\put(22,8){\makebox(0,0){$\bullet$}}
\put(-13,8){\line(1,0){40}}
\put(50,8){\line(-1,0){5}}
\put(37,8){\makebox(0,0){$\cdots$}}
\put(55,8){\makebox(0,0){$\scriptstyle\langle$}}
\put(50,8){\makebox(0,0){$\bullet$}}
\put(60,8){\makebox(0,0){$\bullet$}}
\put(50,9){\line(1,0){10}}
\put(50,7){\line(1,0){10}}
\end{picture})=\begin{picture}(75,20)(-10,5)
\put(-3,13){\makebox(0,0)[b]{$\scriptstyle a+k$}}
\put(12,13){\makebox(0,0)[b]{$\scriptstyle b$}}
\put(22,13){\makebox(0,0)[b]{$\scriptstyle c$}}
\put(50,13){\makebox(0,0)[b]{$\scriptstyle d$}}
\put(60,13){\makebox(0,0)[b]{$\scriptstyle e$}}
\put(-3,8){\makebox(0,0){$\bullet$}}
\put(12,8){\makebox(0,0){$\bullet$}}
\put(22,8){\makebox(0,0){$\bullet$}}
\put(-3,8){\line(1,0){30}}
\put(50,8){\line(-1,0){5}}
\put(37,8){\makebox(0,0){$\cdots$}}
\put(55,8){\makebox(0,0){$\scriptstyle\langle$}}
\put(50,8){\makebox(0,0){$\bullet$}}
\put(60,8){\makebox(0,0){$\bullet$}}
\put(50,9){\line(1,0){10}}
\put(50,7){\line(1,0){10}}
\end{picture},\end{equation}
where these equalities are interpreted as isomorphisms of
${\mathfrak{sp}}(2n,{\mathbb{R}})$-modules. More precisely, it is easy to check
that the various complexes used to define and compute the Lie algebra
cohomology are complexes of ${\mathfrak{g}}_0$-modules. Hence, the
cohomologies are ${\mathfrak{g}}_0$-modules and, in particular,
${\mathfrak{sp}}(2n,{\mathbb{R}})$-modules under restriction
${\mathfrak{sp}}(2n,{\mathbb{R}})\subset{\mathfrak{g}}_0$.
\begin{remark}In fact, with more care, the complexes used in defining and
computing the Lie algebra cohomology of a representation of ${\mathfrak{g}}$
restricted to ${\mathfrak{g}}_-$ are all complexes of ${\mathfrak{p}}$-modules,
where
${\mathfrak{p}}={\mathfrak{g}}_0\oplus{\mathfrak{g}}_1\oplus{\mathfrak{g}}_2$.
\end{remark}

The point of these considerations is that (\ref{Hzero}) and (\ref{Hone}) allow
us to compute the spaces $K_H^\ell$ for a large class of geometrically natural
operators on contact manifolds. Recall that $K_H^\ell$ are defined as
intersections (\ref{KHell}) but we shall see that they also occur in the Lie
algebra cohomology that we have been discussing. To proceed, let us consider
the action of the grading element (\ref{ge}) from ${\mathfrak{g}}_0$ on
${\mathbb{V}}$. The representation ${\mathbb{V}}$ thereby splits as a direct
sum of eigenspaces, each of which is a ${\mathfrak{g}}_0$-module and,
following~\cite{bceg}, it is convenient for our purposes to write this
decomposition as
\begin{equation}\label{gradedV}
{\mathbb{V}}={\mathbb{V}}_0\oplus{\mathbb{V}}_1\oplus
{\mathbb{V}}_2\oplus\cdots\oplus{\mathbb{V}}_N,\quad
\mbox{in which }{\mathfrak{g}}_i{\mathbb{V}}_j\subseteq{\mathbb{V}}_{i+j}.
\end{equation}
Standard representation theory~\cite{humphreys} allows us to conclude that
$${\mathbb{V}}=\begin{picture}(85,20)(-20,5)
\put(-13,8){\makebox(0,0){$\bullet$}}
\put(-13,13){\makebox(0,0)[b]{$\scriptstyle k-1$}}
\put(2,13){\makebox(0,0)[b]{$\scriptstyle a$}}
\put(12,13){\makebox(0,0)[b]{$\scriptstyle b$}}
\put(22,13){\makebox(0,0)[b]{$\scriptstyle c$}}
\put(50,13){\makebox(0,0)[b]{$\scriptstyle d$}}
\put(60,13){\makebox(0,0)[b]{$\scriptstyle e$}}
\put(2,8){\makebox(0,0){$\bullet$}}
\put(12,8){\makebox(0,0){$\bullet$}}
\put(22,8){\makebox(0,0){$\bullet$}}
\put(-13,8){\line(1,0){40}}
\put(50,8){\line(-1,0){5}}
\put(37,8){\makebox(0,0){$\cdots$}}
\put(55,8){\makebox(0,0){$\scriptstyle\langle$}}
\put(50,8){\makebox(0,0){$\bullet$}}
\put(60,8){\makebox(0,0){$\bullet$}}
\put(50,9){\line(1,0){10}}
\put(50,7){\line(1,0){10}}
\end{picture}
\Rightarrow{\mathbb{V}}_0=\begin{picture}(65,20)(0,5)
\put(2,13){\makebox(0,0)[b]{$\scriptstyle a$}}
\put(12,13){\makebox(0,0)[b]{$\scriptstyle b$}}
\put(22,13){\makebox(0,0)[b]{$\scriptstyle c$}}
\put(50,13){\makebox(0,0)[b]{$\scriptstyle d$}}
\put(60,13){\makebox(0,0)[b]{$\scriptstyle e$}}
\put(2,8){\makebox(0,0){$\bullet$}}
\put(12,8){\makebox(0,0){$\bullet$}}
\put(22,8){\makebox(0,0){$\bullet$}}
\put(2,8){\line(1,0){25}}
\put(50,8){\line(-1,0){5}}
\put(37,8){\makebox(0,0){$\cdots$}}
\put(55,8){\makebox(0,0){$\scriptstyle\langle$}}
\put(50,8){\makebox(0,0){$\bullet$}}
\put(60,8){\makebox(0,0){$\bullet$}}
\put(50,9){\line(1,0){10}}
\put(50,7){\line(1,0){10}}
\end{picture}\mbox{ and }N=2(k-1+a+b+c+\cdots+d+e).$$
Now let us suppose that $n\geq 2$ and consider the complex
(\ref{bestLiecomplex}) with a view to computing the Lie algebra cohomologies
$H^0({\mathfrak{g}}_{-1},{\mathbb{V}})$ and
$H^1({\mathfrak{g}}_{-1},{\mathbb{V}})$. Because the differentials
$\partial_H$ are compatible with the action of the grading element, it follows
that $\partial_H$ must be compatible with the grading~(\ref{gradedV}), whence
(\ref{bestLiecomplex}) splits into a series of complexes
\begin{equation}\label{series}0\to{\mathbb{V}}_j
\xrightarrow{\,\partial_H\,}{\mathfrak{g}}_1\otimes{\mathbb{V}}_{j-1}
\xrightarrow{\,\partial_H\,}
\Lambda_\perp^2{\mathfrak{g}}_1\otimes{\mathbb{V}}_{j-2}.\end{equation} Let
$${\mathbb{E}}\equiv{\mathbb{V}}_0=\begin{picture}(65,20)(0,5)
\put(2,13){\makebox(0,0)[b]{$\scriptstyle a$}}
\put(12,13){\makebox(0,0)[b]{$\scriptstyle b$}}
\put(22,13){\makebox(0,0)[b]{$\scriptstyle c$}}
\put(50,13){\makebox(0,0)[b]{$\scriptstyle d$}}
\put(60,13){\makebox(0,0)[b]{$\scriptstyle e$}}
\put(2,8){\makebox(0,0){$\bullet$}} \put(12,8){\makebox(0,0){$\bullet$}}
\put(22,8){\makebox(0,0){$\bullet$}} \put(2,8){\line(1,0){25}}
\put(50,8){\line(-1,0){5}} \put(37,8){\makebox(0,0){$\cdots$}}
\put(55,8){\makebox(0,0){$\scriptstyle\langle$}}
\put(50,8){\makebox(0,0){$\bullet$}} \put(60,8){\makebox(0,0){$\bullet$}}
\put(50,9){\line(1,0){10}} \put(50,7){\line(1,0){10}}
\end{picture}\quad\mbox{and}\quad
{\mathbb{F}}\equiv\begin{picture}(75,20)(-10,5)
\put(-3,13){\makebox(0,0)[b]{$\scriptstyle a+k$}}
\put(12,13){\makebox(0,0)[b]{$\scriptstyle b$}}
\put(22,13){\makebox(0,0)[b]{$\scriptstyle c$}}
\put(50,13){\makebox(0,0)[b]{$\scriptstyle d$}}
\put(60,13){\makebox(0,0)[b]{$\scriptstyle e$}}
\put(-3,8){\makebox(0,0){$\bullet$}} \put(12,8){\makebox(0,0){$\bullet$}}
\put(22,8){\makebox(0,0){$\bullet$}} \put(-3,8){\line(1,0){30}}
\put(50,8){\line(-1,0){5}} \put(37,8){\makebox(0,0){$\cdots$}}
\put(55,8){\makebox(0,0){$\scriptstyle\langle$}}
\put(50,8){\makebox(0,0){$\bullet$}} \put(60,8){\makebox(0,0){$\bullet$}}
\put(50,9){\line(1,0){10}} \put(50,7){\line(1,0){10}}
\end{picture}.$$
Kostant's Theorem~\cite{kostant} implies that
$H^0({\mathfrak{g}}_{-1},{\mathbb{V}})={\mathbb{E}}$ and
$H^1({\mathfrak{g}}_{-1},{\mathbb{V}})={\mathbb{F}}$. The action of
the grading element forces
the first cohomology ${\mathbb{F}}$ to arise from the complex (\ref{series})
when $j=k$. The identification ${\mathbb{V}}_0={\mathbb{E}}$ is also built into
(\ref{series}) as the trivial case $j=0$. Otherwise, we conclude that for
$j=1,2,\cdots,\hat{k},\cdots,N$ the complexes (\ref{series}) are exact.
To state and interpret the algebraic consequences of these statements we make
some preliminary observations and introduce some suggestive notation.
Notice that the action of
${\mathfrak{sp}}(2n,{\mathbb{R}})\subset{\mathfrak{g}}_0$ on ${\mathfrak{g}}_2$
is trivial. Therefore, we may identify ${\mathfrak{g}}_2={\mathbb{R}}$, view
the Lie bracket $\Lambda^2{\mathfrak{g}}_1\to{\mathfrak{g}}_2$ as a
non-degenerate skew form, and thereby ${\mathfrak{g}}_1$ as the defining
representation of ${\mathfrak{sp}}(2n,{\mathbb{R}})$, in which case we shall
denote it by~${\mathbb{S}}_\perp$. As representations of
${\mathfrak{sp}}(2n,{\mathbb{R}})$, we have the irreducible decomposition
\begin{equation}\label{symplecticdecomposition}
\textstyle\bigotimes^2{\mathbb{S}}_\perp=
\underbrace{\textstyle\bigodot^2{\mathbb{S}}_\perp\oplus{\mathbb{R}}}_{\mbox{
\small$\equiv{\mathbb{S}}_\perp^2$}}
\oplus\Lambda_\perp^2{\mathbb{S}}_\perp\end{equation}
defining ${\mathbb{S}}_\perp^2$, where
$\Lambda_\perp^2{\mathbb{S}}_\perp\subset\Lambda^2{\mathbb{S}}_\perp$ is the
kernel of the symplectic form $\Lambda^2{\mathbb{S}}_\perp\to{\mathbb{R}}$. Let
\begin{equation}\label{defofSperpell}{\mathbb{S}}_\perp^\ell\equiv
({\mathbb{S}}_\perp\otimes{\mathbb{S}}_\perp^{\ell-1})\cap
({\mathbb{S}}_\perp^{\ell-1}\otimes{\mathbb{S}}_\perp)\quad\forall\,
\ell\geq 3.\end{equation}
\begin{proposition}\label{algebraicconsequences} For $n\geq 2$,
the algebraic consequences alluded to above are
$$\begin{array}{rcll}{\mathbb{V}}_j&=&{\mathbb{S}}_\perp^j\otimes{\mathbb{E}},&
\forall\,j<k\\
{\mathbb{V}}_k&=&{\mathbb{K}}_H\\
{\mathbb{V}}_j&=&({\mathbb{S}}_\perp^{j-k}\otimes{\mathbb{K}}_H)
\cap({\mathbb{S}}_\perp^j\otimes{\mathbb{E}}),&\forall\,j>k,
\end{array}$$
where ${\mathbb{K}}_H$ is the kernel of the natural projection
${\mathbb{S}}_\perp^k\otimes{\mathbb{E}}\to{\mathbb{F}}$.
\end{proposition}
\begin{proof}If $j=1<k$, then the exact sequence (\ref{series}) reduces to
$$0\to{\mathbb{V}}_1\xrightarrow{\,\partial_H\,}
{\mathfrak{g}}_1\otimes{\mathbb{V}}_{j-1}\to 0,\quad\mbox{equivalently }
{\mathbb{V}}_1={\mathbb{S}}_\perp\otimes{\mathbb{E}}.$$
For $j=2,\cdots,k-1$, suppose by induction that
${\mathbb{V}}_{j-1}={\mathbb{S}}_\perp^{j-1}\otimes{\mathbb{E}}$. Then
(\ref{series}) reads
\begin{equation}\label{nowreads}
0\to{\mathbb{V}}_j\to
{\mathbb{S}}_\perp\otimes{\mathbb{S}}_\perp^{j-1}\otimes{\mathbb{E}}
\xrightarrow{\,\partial_H\,}
\Lambda_\perp^2{\mathbb{S}}_\perp\otimes
{\mathbb{S}}_\perp^{j-2}\otimes{\mathbb{E}}.\end{equation}
Tracing back through the definitions, it may be verified that the
homomorphism $\partial_H$ in (\ref{nowreads}) does not see~${\mathbb{E}}$,
i.e.~there is a homomorphism
$$\eth:{\mathbb{S}}_\perp\otimes{\mathbb{S}}_\perp^{j-1}\to
\Lambda_\perp^2{\mathbb{S}}_\perp\otimes{\mathbb{S}}_\perp^{j-2}$$
so that $\partial_H=\eth\otimes{\mathrm{Id}}$. Explicitly, $\eth$ is the
composition
$${\mathbb{S}}_\perp\otimes{\mathbb{S}}_\perp^{j-1}\hookrightarrow
{\mathbb{S}}_\perp\otimes{\mathbb{S}}_\perp\otimes{\mathbb{S}}_\perp^{j-2}
\xrightarrow{\,\wedge_\perp\otimes{\mathrm{Id}}\,}
\Lambda_\perp^2{\mathbb{S}}_\perp\otimes{\mathbb{S}}_\perp^{j-2}$$
where $\wedge_\perp$ is the projection onto $\Lambda_\perp^2{\mathbb{S}}_\perp$
visible in the symplectic decomposition~(\ref{symplecticdecomposition}). More
explicitly, these statements easily follow from the observation that the
operator $\partial_H:{\mathfrak{g}}_1\otimes{\mathbb{V}}\to
\Lambda_\perp^2{\mathfrak{g}}_1\otimes{\mathbb{V}}$ in (\ref{bestLiecomplex})
may itself be written as the composition
\begin{equation}\label{writepartialH}
{\mathfrak{g}}_1\otimes{\mathbb{V}}
\xrightarrow{\,{\mathrm{Id}}\otimes\partial_H\,}
{\mathfrak{g}}_1\otimes{\mathfrak{g}}_1\otimes{\mathbb{V}}
\xrightarrow{\,\wedge_\perp^2\otimes{\mathrm{Id}}\,}
\Lambda_\perp^2{\mathfrak{g}}_1\otimes{\mathbb{V}}.\end{equation}
When $j=2$ it follows immediately from (\ref{symplecticdecomposition}) that the
kernel of $\eth$ is~${\mathbb{S}}_\perp^2$. More generally, it follows from the
definition (\ref{defofSperpell}) that
$$0\to{\mathbb{S}}_\perp^j\to
{\mathbb{S}}_\perp\otimes{\mathbb{S}}_\perp^{j-1}\xrightarrow{\,\eth\,}
\Lambda_\perp^2{\mathbb{S}}_\perp\otimes{\mathbb{S}}_\perp^{j-2}$$
is exact. Therefore ${\mathbb{V}}_j={\mathbb{S}}_\perp^j\otimes{\mathbb{E}}$,
for $j<k$. Now we encounter the complex
$$0\to{\mathbb{V}}_k
\to{\mathbb{S}}_\perp\otimes{\mathbb{S}}_\perp^{k-1}\otimes{\mathbb{E}}
\xrightarrow{\,\partial_H=\eth\otimes{\mathrm{Id}}\,}
\Lambda_\perp^2{\mathbb{S}}_\perp\otimes
{\mathbb{S}}_\perp^{k-2}\otimes{\mathbb{E}}$$
but it is no longer exact. Instead, the kernel of $\partial_H$ is
${\mathbb{S}}_\perp^k\otimes{\mathbb{E}}$ and we have a short exact
sequence
$$0\to{\mathbb{V}}_k\to{\mathbb{S}}_\perp^k\otimes{\mathbb{E}}\to{\mathbb{F}}
\to 0$$
in accordance with~(\ref{Hone}.) It follows that
${\mathbb{V}}_k={\mathbb{K}}_H$ and the remaining statements follow by
induction from the exactness of (\ref{series}) for $j>k$.\end{proof}
Although the proofs are slightly different, with suitable caveats the results
for $n=1$ are essentially the same. The main difference is that
${\mathbb{S}}_\perp^3$ must be defined separately. Since
$\Lambda_\perp^2{\mathbb{S}}_\perp=0$,
$$\textstyle{\mathbb{S}}_\perp^2=\bigotimes^2{\mathbb{S}}_\perp=
\bigodot^2{\mathbb{S}}_\perp\oplus{\mathbb{R}}$$
in accordance with (\ref{symplecticdecomposition}), but instead of
(\ref{defofSperpell}), we
define
$$\textstyle{\mathbb{S}}_\perp^3\equiv
\bigodot^3{\mathbb{S}}_\perp\oplus\{Q_aL_{bc}+Q_bL_{ac}+Q_cL_{ab}\mbox{ s.t.\ }
Q_a\in{\mathbb{S}}_\perp\},$$
where $L_{ab}\in\Lambda^2{\mathbb{S}}_\perp$ is the symplectic form inverse to
the Lie bracket
$$\Lambda^2{\mathbb{S}}_\perp=\Lambda^2{\mathfrak{g}}_1\to{\mathfrak{g}}_2=
{\mathbb{R}}.$$
With this definition in place and
$${\mathbb{S}}_\perp^\ell\equiv
({\mathbb{S}}_\perp\otimes{\mathbb{S}}_\perp^{\ell-1})\cap
({\mathbb{S}}_\perp^{\ell-1}\otimes{\mathbb{S}}_\perp)\quad\forall\,
\ell\geq 4, $$
the conclusions are unchanged:--
\begin{proposition}\label{alg_con_three} For $n=1$,
$$\begin{array}{rcll}{\mathbb{V}}_j&=&{\mathbb{S}}_\perp^j\otimes{\mathbb{E}},&
\forall\,j<k\\
{\mathbb{V}}_k&=&{\mathbb{K}}_H\\
{\mathbb{V}}_j&=&({\mathbb{S}}_\perp^{j-k}\otimes{\mathbb{K}}_H)
\cap({\mathbb{S}}_\perp^j\otimes{\mathbb{E}}),&\forall\,j>k,
\end{array}$$
where ${\mathbb{K}}_H$ is the kernel of the natural projection
${\mathbb{S}}_\perp^k\otimes{\mathbb{E}}\to{\mathbb{F}}$.
\end{proposition}
\begin{proof} The proof of Proposition~\ref{algebraicconsequences} need only be
modified as follows. We are obliged to use (\ref{bestLiecomplexthree})
rather than~(\ref{bestLiecomplex}). The homomorphisms $\partial_H$ must
again respect gradings but ${\mathfrak{g}}_2$ has weight $2$ with respect to
the action of the grading element so we obtain complexes
\begin{equation}\label{seriesthree}0\to{\mathbb{V}}_j
\xrightarrow{\,\partial_H\,}{\mathfrak{g}}_1\otimes{\mathbb{V}}_{j-1}
\xrightarrow{\,\partial_H\,}
{\mathfrak{g}}_1\otimes{\mathfrak{g}}_2\otimes{\mathbb{V}}_{j-3}\end{equation}
replacing~(\ref{series}). These complexes makes themselves felt through a
new version of the homomorphism~$\eth$. Specifically, the
short exact sequence
$$\textstyle 0\to{\mathbb{S}}_\perp^3\to\bigotimes^3{\mathbb{S}}_\perp
\xrightarrow{\,\Sigma\,}{\mathbb{S}}_\perp\to 0,$$
where $\Sigma$ is the homomorphism
$$\textstyle\bigotimes^3{\mathbb{S}}_\perp=\bigotimes^3{\mathfrak{g}}_1
\ni\phi_{abc}\longmapsto
L^{ab}\big(\phi_{abc}-\phi_{cab}\big)\in{\mathfrak{g}}_1\otimes{\mathfrak{g}}_2
={\mathbb{S}}_\perp\otimes{\mathbb{R}}={\mathbb{S}}_\perp$$
and $L:{\mathfrak{g}}_1\otimes{\mathfrak{g}}_1\to{\mathfrak{g}}_2$ is Lie
bracket, replaces
$$\textstyle 0\to{\mathbb{S}}_\perp^2\to\bigotimes^2{\mathbb{S}}_\perp
\xrightarrow{\,\pi_\perp\,}\Lambda_\perp^2{\mathbb{S}}_\perp\to 0$$
and $\eth$ is defined as the composition
$$\textstyle{\mathbb{S}}_\perp\otimes{\mathbb{S}}_\perp^{j-1}\hookrightarrow
\bigotimes^3{\mathbb{S}}_\perp\otimes{\mathbb{S}}_\perp^{j-3}
\xrightarrow{\,\Sigma\otimes{\mathrm{Id}}\,}
{\mathbb{S}}_\perp\otimes{\mathbb{S}}_\perp^{j-3}.$$
Parallel to (\ref{writepartialH}) is composition
$${\mathfrak{g}}_1\otimes{\mathbb{V}}
\xrightarrow{\,{\mathrm{Id}}\otimes\partial_H\,}
{\mathfrak{g}}_1\otimes{\mathfrak{g}}_1\otimes{\mathbb{V}}
\xrightarrow{\,{\mathrm{Id}}\otimes{\mathrm{Id}}\otimes\partial_H\,}
{\mathfrak{g}}_1\otimes{\mathfrak{g}}_1\otimes{\mathfrak{g}}_1
\otimes{\mathbb{V}}\xrightarrow{\,\Sigma\otimes{\mathrm{Id}}\,}
{\mathfrak{g}}_1\otimes{\mathfrak{g}}_2\otimes{\mathbb{V}}$$
as a way of writing $\partial_H:{\mathfrak{g}}_1\otimes{\mathbb{V}}\to
{\mathfrak{g}}_1\otimes{\mathfrak{g}}_2\otimes{\mathbb{V}}$
in~(\ref{bestLiecomplexthree}). Apart from this key alteration, the proof
follows exactly the same lines and details are left to the reader.
\end{proof}

Propositions~\ref{algebraicconsequences} and~\ref{alg_con_three} have immediate
geometric consequences as follows. Fix a contact manifold $M$ of dimension
$2n+1$ with contact distribution $H$ as usual. Any finite-dimensional
representation ${\mathbb{E}}$ of~${\mathrm{Sp}}(2n,{\mathbb{R}})$, gives rises
to a vector bundle $E$ on $M$ by induction from the bundle of symplectic
co-frames for~$H$. We shall refer to such $E$ as {\em symplectic\/} vector
bundles. In particular, the defining representation, which we have been writing
above as~${\mathbb{S}}_\perp$, gives rise to the bundle $\Lambda_H^1$ dual
to~$H$. More generally, the representations ${\mathbb{S}}_\perp^\ell$ give rise
to bundles that we have already been writing as $S_\perp^\ell$
in~\S\ref{contact} and~\S\ref{highercontact}. We shall refer to the bundle $E$
as {\em irreducible\/} if the representation ${\mathbb{E}}$ is irreducible.
Thus, the irreducible symplectic bundles on $M$ can be parameterised
exactly as for irreducible representations
of~${\mathfrak{sp}}(2n,{\mathbb{R}})$, namely as
$$\begin{picture}(65,20)(0,5)
\put(2,13){\makebox(0,0)[b]{$\scriptstyle a$}}
\put(12,13){\makebox(0,0)[b]{$\scriptstyle b$}}
\put(22,13){\makebox(0,0)[b]{$\scriptstyle c$}}
\put(50,13){\makebox(0,0)[b]{$\scriptstyle d$}}
\put(60,13){\makebox(0,0)[b]{$\scriptstyle e$}}
\put(2,8){\makebox(0,0){$\bullet$}}
\put(12,8){\makebox(0,0){$\bullet$}}
\put(22,8){\makebox(0,0){$\bullet$}}
\put(2,8){\line(1,0){25}}
\put(50,8){\line(-1,0){5}}
\put(37,8){\makebox(0,0){$\cdots$}}
\put(55,8){\makebox(0,0){$\scriptstyle\langle$}}
\put(50,8){\makebox(0,0){$\bullet$}}
\put(60,8){\makebox(0,0){$\bullet$}}
\put(50,9){\line(1,0){10}}
\put(50,7){\line(1,0){10}}
\end{picture}\quad\mbox{for }a,b,c,\ldots,d,e\in{\mathbb{Z}}_{\geq 0}.$$
In particular,
$${\mathbb{S}}_\perp=\begin{picture}(65,20)(0,5)
\put(2,13){\makebox(0,0)[b]{$\scriptstyle 1$}}
\put(12,13){\makebox(0,0)[b]{$\scriptstyle 0$}}
\put(22,13){\makebox(0,0)[b]{$\scriptstyle 0$}}
\put(50,13){\makebox(0,0)[b]{$\scriptstyle 0$}}
\put(60,13){\makebox(0,0)[b]{$\scriptstyle 0$}}
\put(2,8){\makebox(0,0){$\bullet$}}
\put(12,8){\makebox(0,0){$\bullet$}}
\put(22,8){\makebox(0,0){$\bullet$}}
\put(2,8){\line(1,0){25}}
\put(50,8){\line(-1,0){5}}
\put(37,8){\makebox(0,0){$\cdots$}}
\put(55,8){\makebox(0,0){$\scriptstyle\langle$}}
\put(50,8){\makebox(0,0){$\bullet$}}
\put(60,8){\makebox(0,0){$\bullet$}}
\put(50,9){\line(1,0){10}}
\put(50,7){\line(1,0){10}}
\end{picture}\qquad\textstyle\bigodot^k{\mathbb{S}}_\perp=
\begin{picture}(65,20)(0,5)
\put(2,13){\makebox(0,0)[b]{$\scriptstyle k$}}
\put(12,13){\makebox(0,0)[b]{$\scriptstyle 0$}}
\put(22,13){\makebox(0,0)[b]{$\scriptstyle 0$}}
\put(50,13){\makebox(0,0)[b]{$\scriptstyle 0$}}
\put(60,13){\makebox(0,0)[b]{$\scriptstyle 0$}}
\put(2,8){\makebox(0,0){$\bullet$}}
\put(12,8){\makebox(0,0){$\bullet$}}
\put(22,8){\makebox(0,0){$\bullet$}}
\put(2,8){\line(1,0){25}}
\put(50,8){\line(-1,0){5}}
\put(37,8){\makebox(0,0){$\cdots$}}
\put(55,8){\makebox(0,0){$\scriptstyle\langle$}}
\put(50,8){\makebox(0,0){$\bullet$}}
\put(60,8){\makebox(0,0){$\bullet$}}
\put(50,9){\line(1,0){10}}
\put(50,7){\line(1,0){10}}
\end{picture}
$$
and
$${\mathbb{E}}=\begin{picture}(65,20)(0,5)
\put(2,13){\makebox(0,0)[b]{$\scriptstyle a$}}
\put(12,13){\makebox(0,0)[b]{$\scriptstyle b$}}
\put(22,13){\makebox(0,0)[b]{$\scriptstyle c$}}
\put(50,13){\makebox(0,0)[b]{$\scriptstyle d$}}
\put(60,13){\makebox(0,0)[b]{$\scriptstyle e$}}
\put(2,8){\makebox(0,0){$\bullet$}}
\put(12,8){\makebox(0,0){$\bullet$}}
\put(22,8){\makebox(0,0){$\bullet$}}
\put(2,8){\line(1,0){25}}
\put(50,8){\line(-1,0){5}}
\put(37,8){\makebox(0,0){$\cdots$}}
\put(55,8){\makebox(0,0){$\scriptstyle\langle$}}
\put(50,8){\makebox(0,0){$\bullet$}}
\put(60,8){\makebox(0,0){$\bullet$}}
\put(50,9){\line(1,0){10}}
\put(50,7){\line(1,0){10}}
\end{picture}\implies\begin{picture}(75,20)(-10,5)
\put(-3,13){\makebox(0,0)[b]{$\scriptstyle a+k$}}
\put(12,13){\makebox(0,0)[b]{$\scriptstyle b$}}
\put(22,13){\makebox(0,0)[b]{$\scriptstyle c$}}
\put(50,13){\makebox(0,0)[b]{$\scriptstyle d$}}
\put(60,13){\makebox(0,0)[b]{$\scriptstyle e$}}
\put(-3,8){\makebox(0,0){$\bullet$}}
\put(12,8){\makebox(0,0){$\bullet$}}
\put(22,8){\makebox(0,0){$\bullet$}}
\put(-3,8){\line(1,0){30}}
\put(50,8){\line(-1,0){5}}
\put(37,8){\makebox(0,0){$\cdots$}}
\put(55,8){\makebox(0,0){$\scriptstyle\langle$}}
\put(50,8){\makebox(0,0){$\bullet$}}
\put(60,8){\makebox(0,0){$\bullet$}}
\put(50,9){\line(1,0){10}}
\put(50,7){\line(1,0){10}}
\end{picture}=\textstyle\bigodot^k{\mathbb{S}}_\perp\circledcirc{\mathbb{E}},$$
where $\circledcirc$ denotes the {\em Cartan product\/}, namely the unique
irreducible summand of the tensor product whose highest weight is the sum of
the highest weights of the two factors. We shall use the same notation for the
corresponding Cartan product of irreducible symplectic bundles. Finally let us
notice that
\begin{equation}\label{alg_counterpart}\textstyle{\mathbb{S}}_\perp^k\cong
\bigodot^k{\mathbb{S}}_\perp\oplus\bigodot^{k-2}{\mathbb{S}}_\perp
\oplus\bigodot^{k-4}{\mathbb{S}}_\perp\oplus\cdots\end{equation}
as the algebraic counterpart of~(\ref{bundlesum}).
A key geometric consequence alluded to above is as follows.
\begin{theorem}\label{plainBGGbound}
Suppose $E$ is an irreducible symplectic bundle on a contact manifold of
dimension $2n+1$ corresponding to the irreducible representation
$${\mathbb{E}}=\begin{picture}(65,20)(0,5)
\put(2,13){\makebox(0,0)[b]{$\scriptstyle a$}}
\put(12,13){\makebox(0,0)[b]{$\scriptstyle b$}}
\put(22,13){\makebox(0,0)[b]{$\scriptstyle c$}}
\put(50,13){\makebox(0,0)[b]{$\scriptstyle d$}}
\put(60,13){\makebox(0,0)[b]{$\scriptstyle e$}}
\put(2,8){\makebox(0,0){$\bullet$}}
\put(12,8){\makebox(0,0){$\bullet$}}
\put(22,8){\makebox(0,0){$\bullet$}}
\put(2,8){\line(1,0){25}}
\put(50,8){\line(-1,0){5}}
\put(37,8){\makebox(0,0){$\cdots$}}
\put(55,8){\makebox(0,0){$\scriptstyle\langle$}}
\put(50,8){\makebox(0,0){$\bullet$}}
\put(60,8){\makebox(0,0){$\bullet$}}
\put(50,9){\line(1,0){10}}
\put(50,7){\line(1,0){10}}
\end{picture}\quad\mbox{of}\enskip{\mathrm{Sp}}(2n,{\mathbb{R}}).$$
Let $F=\bigodot^k\!\Lambda_H^1\circledcirc E$. Suppose $D:E\to F$ is a
$k^{\mathrm{th}}$ order linear differential operator compatible with the
contact structure whose enhanced symbol is the composition
$$\textstyle S_\perp^k\otimes E\to\bigodot^k\!\Lambda_H^1\otimes E
\xrightarrow{\,\circledcirc\,}\bigodot^k\!\Lambda_H^1\circledcirc E=F.$$
Then $D$ is of finite-type and the dimension of its solution space is bounded
by the dimension of the representation
$$\begin{picture}(85,20)(-20,5)
\put(-13,8){\makebox(0,0){$\bullet$}}
\put(-13,13){\makebox(0,0)[b]{$\scriptstyle k-1$}}
\put(2,13){\makebox(0,0)[b]{$\scriptstyle a$}}
\put(12,13){\makebox(0,0)[b]{$\scriptstyle b$}}
\put(22,13){\makebox(0,0)[b]{$\scriptstyle c$}}
\put(50,13){\makebox(0,0)[b]{$\scriptstyle d$}}
\put(60,13){\makebox(0,0)[b]{$\scriptstyle e$}}
\put(2,8){\makebox(0,0){$\bullet$}}
\put(12,8){\makebox(0,0){$\bullet$}}
\put(22,8){\makebox(0,0){$\bullet$}}
\put(-13,8){\line(1,0){40}}
\put(50,8){\line(-1,0){5}}
\put(37,8){\makebox(0,0){$\cdots$}}
\put(55,8){\makebox(0,0){$\scriptstyle\langle$}}
\put(50,8){\makebox(0,0){$\bullet$}}
\put(60,8){\makebox(0,0){$\bullet$}}
\put(50,9){\line(1,0){10}}
\put(50,7){\line(1,0){10}}
\end{picture}$$
of\/ ${\mathfrak{sp}}(2(n+1),{\mathbb{R}})$.
\end{theorem}
\begin{remark}
On a symplectic manifold of dimension $2n$ the structure group of the full
co-frame bundle is reduced to ${\mathrm{Sp}}(2n,{\mathbb{R}})$. Hence, one can
similarly define sympletic vector bundles on symplectic manifolds as those
induced by the finite-dimensional representations
of~${\mathrm{Sp}}(2n,{\mathbb{R}})$. Suppose $E$ is such a bundle, induced
by the irreducible representation
$${\mathbb{E}}=\begin{picture}(65,20)(0,5)
\put(2,13){\makebox(0,0)[b]{$\scriptstyle a$}}
\put(12,13){\makebox(0,0)[b]{$\scriptstyle b$}}
\put(22,13){\makebox(0,0)[b]{$\scriptstyle c$}}
\put(50,13){\makebox(0,0)[b]{$\scriptstyle d$}}
\put(60,13){\makebox(0,0)[b]{$\scriptstyle e$}}
\put(2,8){\makebox(0,0){$\bullet$}}
\put(12,8){\makebox(0,0){$\bullet$}}
\put(22,8){\makebox(0,0){$\bullet$}}
\put(2,8){\line(1,0){25}}
\put(50,8){\line(-1,0){5}}
\put(37,8){\makebox(0,0){$\cdots$}}
\put(55,8){\makebox(0,0){$\scriptstyle\langle$}}
\put(50,8){\makebox(0,0){$\bullet$}}
\put(60,8){\makebox(0,0){$\bullet$}}
\put(50,9){\line(1,0){10}}
\put(50,7){\line(1,0){10}}
\end{picture}\quad\mbox{of}\enskip{\mathrm{Sp}}(2n,{\mathbb{R}}).$$
Let $F=\bigodot^k\!\Lambda^1\circledcirc E$, in other words $F$ is induced by
the representation
\begin{picture}(75,20)(-10,5)
\put(-3,13){\makebox(0,0)[b]{$\scriptstyle a+k$}}
\put(12,13){\makebox(0,0)[b]{$\scriptstyle b$}}
\put(22,13){\makebox(0,0)[b]{$\scriptstyle c$}}
\put(50,13){\makebox(0,0)[b]{$\scriptstyle d$}}
\put(60,13){\makebox(0,0)[b]{$\scriptstyle e$}}
\put(-3,8){\makebox(0,0){$\bullet$}}
\put(12,8){\makebox(0,0){$\bullet$}}
\put(22,8){\makebox(0,0){$\bullet$}}
\put(-3,8){\line(1,0){30}}
\put(50,8){\line(-1,0){5}}
\put(37,8){\makebox(0,0){$\cdots$}}
\put(55,8){\makebox(0,0){$\scriptstyle\langle$}}
\put(50,8){\makebox(0,0){$\bullet$}}
\put(60,8){\makebox(0,0){$\bullet$}}
\put(50,9){\line(1,0){10}}
\put(50,7){\line(1,0){10}}
\end{picture}, which we shall denote by~${\mathbb{F}}$.
Suppose $D:E\to F$ is a $k^{\mathrm{th}}$ order linear differential operator
whose symbol is the composition
$$\textstyle\bigodot^k\!\Lambda^1\otimes E
\xrightarrow{\,\circledcirc\,}\bigodot^k\!\Lambda^1\circledcirc E=F.$$
Then the hypotheses of Theorem~\ref{improvedspencer} apply owing to the
following observations.
\begin{itemize}
\item $\Lambda^1$ is induced by ${\mathbb{S}}_\perp$, the defining
representation of ${\mathrm{Sp}}(2n,{\mathbb{R}})$
\item $K$ is induced by ${\mathbb{K}}\equiv\ker\circledcirc:
\bigodot^k{\mathbb{S}}_\perp
\otimes{\mathbb{E}}\to{\mathbb{F}}$
\item $K^\ell$ is induced by ${\mathbb{K}}^\ell\equiv
(\bigodot^\ell{\mathbb{S}}_\perp\otimes{\mathbb{K}})\cap
(\bigodot^{k+\ell}{\mathbb{S}}_\perp\otimes{\mathbb{E}})$
\item ${\mathbb{K}}^\ell\subseteq
({\mathbb{S}}_\perp^\ell\otimes{\mathbb{K}}_H)\cap
({\mathbb{S}}_\perp^{k+\ell}\otimes{\mathbb{E}})\;\forall\ell\geq 0$, where
${\mathbb{K}}_H\equiv\ker\circledcirc:{\mathbb{S}}_\perp^k\otimes{\mathbb{E}}
\to{\mathbb{F}}.$
\end{itemize}
In other words, the consequences of Kostant's Theorem detailed in
Propositions~\ref{algebraicconsequences} and ~\ref{alg_con_three} are clearly
stronger than needed (for $D$ to be finite-type in the sense of
Theorem~\ref{improvedspencer}) by dint of~(\ref{alg_counterpart}). In
particular, we obtain a bound on the dimension of the solution space for~$D$.
Presumably, this bound is not at all sharp.
\end{remark}

It is also possible to adapt the theory to deal with contact manifolds endowed
with certain additional structures following a similar set of variations
concerning general prolongations. In the general case, the main examples are
affine manifolds and Riemannian manifolds. However, as detailed in~\cite{bceg},
the theory also applies to geometries derived from any $|1|$-graded simple Lie
algebra. For contact manifolds, the corresponding results are as follows.

Recall the decomposition (\ref{Cdeomposition}) of ${\mathfrak{sp}}(2(n+1))$:--
$${\mathfrak{sp}}(2(n+1))=
\underbrace{{\mathfrak{g}}_{-2}\oplus{\mathfrak{g}}_{-1}}_{
\makebox[50pt][r]{\footnotesize Heisenberg algebra
of dimension $2n+1$}{}}
\oplus\:
\underbrace{{\mathbb{R}}\oplus{\mathfrak{sp}}(2n)}_{
\mbox{${\mathfrak{g}}_0$}}\oplus\;{\mathfrak{g}}_{1}
\oplus{\mathfrak{g}}_{2}$$
These salient features pertain for any simple Lie algebra other than
${\mathfrak{sl}}(2)$.
For the remaining classical Lie algebras
$$\begin{array}l{\mathfrak{sl}}(n+2)=
\underbrace{{\mathfrak{g}}_{-2}\oplus{\mathfrak{g}}_{-1}}_{
\makebox[50pt][r]{\footnotesize Heisenberg algebra
of dimension $2n+1\rule[-10pt]{0pt}{10pt}$}{}}
\oplus\:
\underbrace{{\mathbb{R}}\oplus{\mathbb{R}}\oplus{\mathfrak{sl}}(n)}_{
\mbox{${\mathfrak{g}}_0$}}\oplus\;{\mathfrak{g}}_{1}
\oplus{\mathfrak{g}}_{2}\\
{\mathfrak{so}}(n+4)=
\underbrace{{\mathfrak{g}}_{-2}\oplus{\mathfrak{g}}_{-1}}_{
\makebox[50pt][r]{\footnotesize Heisenberg algebra
of dimension $2n+1$}{}}
\oplus\:
\underbrace{{\mathbb{R}}\oplus{\mathfrak{sl}}(2)\oplus{\mathfrak{so}}(n)}_{
\mbox{${\mathfrak{g}}_0$}}\oplus\;{\mathfrak{g}}_{1}
\oplus{\mathfrak{g}}_{2}\end{array}$$
and for the exceptional Lie algebras:--
$$\begin{array}lE_6=
\underbrace{{\mathfrak{g}}_{-2}\oplus{\mathfrak{g}}_{-1}}_{
\makebox[50pt][r]{\footnotesize Heisenberg algebra
of dimension $21\rule[-10pt]{0pt}{10pt}$}{}}
\oplus\:
\underbrace{{\mathbb{R}}\oplus{\mathfrak{sl}}(6)}_{
\mbox{${\mathfrak{g}}_0$}}\oplus\;{\mathfrak{g}}_{1}
\oplus{\mathfrak{g}}_{2}\\
E_7=
\underbrace{{\mathfrak{g}}_{-2}\oplus{\mathfrak{g}}_{-1}}_{
\makebox[50pt][r]{\footnotesize Heisenberg algebra
of dimension $33\rule[-10pt]{0pt}{10pt}$}{}}
\oplus\:
\underbrace{{\mathbb{R}}\oplus{\mathfrak{so}}(12)}_{
\mbox{${\mathfrak{g}}_0$}}\oplus\;{\mathfrak{g}}_{1}
\oplus{\mathfrak{g}}_{2}\\
E_8=
\underbrace{{\mathfrak{g}}_{-2}\oplus{\mathfrak{g}}_{-1}}_{
\makebox[50pt][r]{\footnotesize Heisenberg algebra
of dimension $57\rule[-10pt]{0pt}{10pt}$}{}}
\oplus\:
\underbrace{{\mathbb{R}}\oplus E_7}_{
\mbox{${\mathfrak{g}}_0$}}\oplus\;{\mathfrak{g}}_{1}
\oplus{\mathfrak{g}}_{2}\\
F_4=
\underbrace{{\mathfrak{g}}_{-2}\oplus{\mathfrak{g}}_{-1}}_{
\makebox[50pt][r]{\footnotesize Heisenberg algebra
of dimension $15\rule[-10pt]{0pt}{10pt}$}{}}
\oplus\:
\underbrace{{\mathbb{R}}\oplus{\mathfrak{sp}}(6)}_{
\mbox{${\mathfrak{g}}_0$}}\oplus\;{\mathfrak{g}}_{1}
\oplus{\mathfrak{g}}_{2}\\
G_2=
\underbrace{{\mathfrak{g}}_{-2}\oplus{\mathfrak{g}}_{-1}}_{
\makebox[50pt][r]{\footnotesize Heisenberg algebra
of dimension $5$}{}}
\oplus\:
\underbrace{{\mathbb{R}}\oplus{\mathfrak{sl}}(2)}_{
\mbox{${\mathfrak{g}}_0$}}\oplus\;{\mathfrak{g}}_{1}
\oplus{\mathfrak{g}}_{2}.\end{array}$$
In each case, the adjoint action of ${\mathfrak{g}}_0$ on ${\mathfrak{g}}_{-1}$
is compatible with the Lie bracket having values in the $1$-dimensional
${\mathfrak{g}}_{-2}$. We obtain a subalgebra of the symplectic Lie
algebra and corresponding subgroups by exponentiation:--
$$\begin{array}{lll}{\mathrm{Sp}}(2n,{\mathbb{R}})
\stackrel{\underbar{\enskip}\;\;}{\to}
{\mathrm{Sp}}(2n,{\mathbb{R}})&
{\mathrm{SL}}(n,{\mathbb{R}})\hookrightarrow{\mathrm{Sp}}(2n,{\mathbb{R}})&
{\mathrm{SL}}(2,{\mathbb{R}})\times{\mathrm{SO}}(n)\hookrightarrow
{\mathrm{Sp}}(2n,{\mathbb{R}})\\[4pt]
{\mathrm{SL}}(6,{\mathbb{R}})\hookrightarrow{\mathrm{Sp}}(20,{\mathbb{R}})&
{\mathrm{Spin}}(12)\hookrightarrow{\mathrm{Sp}}(32,{\mathbb{R}})&
{\mathrm{E}}_7\hookrightarrow{\mathrm{Sp}}(56,{\mathbb{R}})\\[4pt]
{\mathrm{Sp}}(6,{\mathbb{R}})\hookrightarrow{\mathrm{Sp}}(14,{\mathbb{R}})&
{\mathrm{SL}}(2,{\mathbb{R}})\hookrightarrow{\mathrm{Sp}}(4,{\mathbb{R}}).
\end{array}$$
In most cases, it is straightforward to describe these embeddings by explicit
formul{\ae}. For these purposes, let us realise
$${\mathrm{Sp}}(2n,{\mathbb{R}})=
\left\{S\in{\mathrm{SL}}(2n,{\mathbb{R}})\mbox{ s.t.\ }SJS^t=J\right\},
\quad\mbox{where }J\equiv
\mbox{\small$\left[\!\!\begin{array}{cc}0&{\mathrm{Id}}\\
-{\mathrm{Id}}&0\end{array}\!\right]$}.$$
Then, for example,
$$\begin{array}l
{\mathrm{SL}}(n,{\mathbb{R}})\ni M\mapsto
\mbox{\small$\left[\!\begin{array}{cc}M&0\\0&(M^t)^{-1}\end{array}\!\right]$}
\in{\mathrm{Sp}}(2n,{\mathbb{R}})
\refstepcounter{equation}\qquad\qquad(\theequation)\label{realCR}\\
{\mathrm{SL}}(2,{\mathbb{R}})\times{\mathrm{SO}}(n)\ni\left(
\mbox{\small$\left[\!\begin{array}{cc}a&b\\c&d\end{array}\!\right]$},
M\right)\mapsto
\mbox{\small$\left[\!\begin{array}{cc}aM&bM\\cM&dM\end{array}\!\right]$}
\in{\mathrm{Sp}}(2n,{\mathbb{R}})\\
{\mathrm{SL}}(2,{\mathbb{R}})\ni
\mbox{\small$\left[\!\begin{array}{cc}a&b\\c&d\end{array}\!\right]$}\mapsto
\mbox{\small$\left[\!\begin{array}{cccc}
a^3&\sqrt{3}a^2b&-b^3&\sqrt{3}ab^2\\
\sqrt{3}a^2c&2abc+a^2d&-\sqrt{3}b^2d&2abd+b^2c\\
-c^3&-\sqrt{3}c^2d&d^3&-\sqrt{3}cd^2\\
\sqrt{3}ac^2&2acd+bc^2&-\sqrt{3}bd^2&2bcd+ad^2
\end{array}\!\right]$}\in{\mathrm{Sp}}(4,{\mathbb{R}}).\end{array}$$
There are alternative real forms of some of these embeddings. For example,
instead of
${\mathrm{SL}}(n,{\mathbb{R}})\hookrightarrow{\mathrm{Sp}}(2n,{\mathbb{R}})$ we
may consider the standard embedding
${\mathrm{SU}}(n)\hookrightarrow{\mathrm{Sp}}(2n,{\mathbb{R}})$ obtained by
regarding $J$ as a complex structure.

For each of these subgroups $G\hookrightarrow{\mathrm{Sp}}(2n,{\mathbb{R}})$ we
may define an additional structure on a contact manifold of dimension $2n+1$ by
reducing the structure group on the contact distribution to~$G$. Often there is
a simple geometric interpretation of such an additional reduction. For example,
it is clear from the embedding (\ref{realCR}) that such a reduction corresponds
to choosing a pair of transverse Lagrangian subdistributions of the contact
distribution. A contact manifold equipped with a reduction of structure group
to ${\mathrm{SU}}(n)\subset{\mathrm{Sp}}(2n,{\mathbb{R}})$ is the same as an
almost CR-structure together with a choice of contact form, i.e.\ an almost
pseudo-Hermitian structure. There are many natural differential operators
compatible with these various structures. Here are some simple examples.

\begin{example} For any partial connection $\nabla_H$ on $\Lambda_H^1$ on a
contact manifold, consider the differential operator $D$ defined as the
composition
$$\textstyle\bigodot^k\!\Lambda_H^1\xrightarrow{\,\nabla_H\,}
\Lambda_H^1\otimes\bigodot^k\!\Lambda_H^1\xrightarrow{\,\odot\,}
\bigodot^{k+1}\!\Lambda_H^1.$$
Using abstract indices as in \S\ref{three},
$$\sigma_{bc\cdots d}\stackrel{D}{\longmapsto}
\nabla_{(a}\sigma_{bc\cdots d)}.$$
This is an operator between symplectic bundles whose symbol in dimension $2n+1$
is induced by the homomorphism of ${\mathrm{Sp}}(2n,{\mathbb{R}})$-modules
$$\begin{picture}(65,20)(0,5)
\put(2,13){\makebox(0,0)[b]{$\scriptstyle 1$}}
\put(12,13){\makebox(0,0)[b]{$\scriptstyle 0$}}
\put(22,13){\makebox(0,0)[b]{$\scriptstyle 0$}}
\put(50,13){\makebox(0,0)[b]{$\scriptstyle 0$}}
\put(60,13){\makebox(0,0)[b]{$\scriptstyle 0$}}
\put(2,8){\makebox(0,0){$\bullet$}}
\put(12,8){\makebox(0,0){$\bullet$}}
\put(22,8){\makebox(0,0){$\bullet$}}
\put(2,8){\line(1,0){25}}
\put(50,8){\line(-1,0){5}}
\put(37,8){\makebox(0,0){$\cdots$}}
\put(55,8){\makebox(0,0){$\scriptstyle\langle$}}
\put(50,8){\makebox(0,0){$\bullet$}}
\put(60,8){\makebox(0,0){$\bullet$}}
\put(50,9){\line(1,0){10}}
\put(50,7){\line(1,0){10}}
\end{picture}\otimes
\begin{picture}(65,20)(0,5)
\put(2,13){\makebox(0,0)[b]{$\scriptstyle k$}}
\put(12,13){\makebox(0,0)[b]{$\scriptstyle 0$}}
\put(22,13){\makebox(0,0)[b]{$\scriptstyle 0$}}
\put(50,13){\makebox(0,0)[b]{$\scriptstyle 0$}}
\put(60,13){\makebox(0,0)[b]{$\scriptstyle 0$}}
\put(2,8){\makebox(0,0){$\bullet$}}
\put(12,8){\makebox(0,0){$\bullet$}}
\put(22,8){\makebox(0,0){$\bullet$}}
\put(2,8){\line(1,0){25}}
\put(50,8){\line(-1,0){5}}
\put(37,8){\makebox(0,0){$\cdots$}}
\put(55,8){\makebox(0,0){$\scriptstyle\langle$}}
\put(50,8){\makebox(0,0){$\bullet$}}
\put(60,8){\makebox(0,0){$\bullet$}}
\put(50,9){\line(1,0){10}}
\put(50,7){\line(1,0){10}}
\end{picture}\longrightarrow
\begin{picture}(75,20)(-10,5)
\put(-3,13){\makebox(0,0)[b]{$\scriptstyle k+1$}}
\put(12,13){\makebox(0,0)[b]{$\scriptstyle 0$}}
\put(22,13){\makebox(0,0)[b]{$\scriptstyle 0$}}
\put(50,13){\makebox(0,0)[b]{$\scriptstyle 0$}}
\put(60,13){\makebox(0,0)[b]{$\scriptstyle 0$}}
\put(-3,8){\makebox(0,0){$\bullet$}}
\put(12,8){\makebox(0,0){$\bullet$}}
\put(22,8){\makebox(0,0){$\bullet$}}
\put(-3,8){\line(1,0){30}}
\put(50,8){\line(-1,0){5}}
\put(37,8){\makebox(0,0){$\cdots$}}
\put(55,8){\makebox(0,0){$\scriptstyle\langle$}}
\put(50,8){\makebox(0,0){$\bullet$}}
\put(60,8){\makebox(0,0){$\bullet$}}
\put(50,9){\line(1,0){10}}
\put(50,7){\line(1,0){10}}
\end{picture}.
$$
Theorem~\ref{plainBGGbound} applies and we conclude that
$$\dim\{D\sigma=0\}\leq\dim\:
\begin{picture}(75,20)(0,5)
\put(2,13){\makebox(0,0)[b]{$\scriptstyle 0$}}
\put(12,13){\makebox(0,0)[b]{$\scriptstyle k$}}
\put(22,13){\makebox(0,0)[b]{$\scriptstyle 0$}}
\put(32,13){\makebox(0,0)[b]{$\scriptstyle 0$}}
\put(60,13){\makebox(0,0)[b]{$\scriptstyle 0$}}
\put(70,13){\makebox(0,0)[b]{$\scriptstyle 0$}}
\put(2,8){\makebox(0,0){$\bullet$}}
\put(12,8){\makebox(0,0){$\bullet$}}
\put(22,8){\makebox(0,0){$\bullet$}}
\put(32,8){\makebox(0,0){$\bullet$}}
\put(2,8){\line(1,0){35}}
\put(60,8){\line(-1,0){5}}
\put(47,8){\makebox(0,0){$\cdots$}}
\put(65,8){\makebox(0,0){$\scriptstyle\langle$}}
\put(60,8){\makebox(0,0){$\bullet$}}
\put(70,8){\makebox(0,0){$\bullet$}}
\put(60,9){\line(1,0){10}}
\put(60,7){\line(1,0){10}}
\end{picture}=\frac{(2n+k)!(2n+k-1)!(2n+2k+1)}{(2n-1)!(2n+1)!k!(k+1)!}.
$$
The case $n=1$ gives the series of bounds presented at the end
of~\S\ref{contact}.
\end{example}
\begin{example}
Suppose that $\nabla_H$ is a partial connection
on $\Lambda_H^1$ and consider the differential operator
$$\sigma_{cde}\stackrel{D}{\longmapsto}
\nabla_{(a}\nabla_b\sigma_{cde)}$$
for $\sigma_{cde}$ symmetric. It is a differential operator
$$D:\begin{picture}(65,20)(0,5)
\put(2,13){\makebox(0,0)[b]{$\scriptstyle 3$}}
\put(12,13){\makebox(0,0)[b]{$\scriptstyle 0$}}
\put(22,13){\makebox(0,0)[b]{$\scriptstyle 0$}}
\put(50,13){\makebox(0,0)[b]{$\scriptstyle 0$}}
\put(60,13){\makebox(0,0)[b]{$\scriptstyle 0$}}
\put(2,8){\makebox(0,0){$\bullet$}}
\put(12,8){\makebox(0,0){$\bullet$}}
\put(22,8){\makebox(0,0){$\bullet$}}
\put(2,8){\line(1,0){25}}
\put(50,8){\line(-1,0){5}}
\put(37,8){\makebox(0,0){$\cdots$}}
\put(55,8){\makebox(0,0){$\scriptstyle\langle$}}
\put(50,8){\makebox(0,0){$\bullet$}}
\put(60,8){\makebox(0,0){$\bullet$}}
\put(50,9){\line(1,0){10}}
\put(50,7){\line(1,0){10}}
\end{picture}\to\begin{picture}(65,20)(0,5)
\put(2,13){\makebox(0,0)[b]{$\scriptstyle 5$}}
\put(12,13){\makebox(0,0)[b]{$\scriptstyle 0$}}
\put(22,13){\makebox(0,0)[b]{$\scriptstyle 0$}}
\put(50,13){\makebox(0,0)[b]{$\scriptstyle 0$}}
\put(60,13){\makebox(0,0)[b]{$\scriptstyle 0$}}
\put(2,8){\makebox(0,0){$\bullet$}}
\put(12,8){\makebox(0,0){$\bullet$}}
\put(22,8){\makebox(0,0){$\bullet$}}
\put(2,8){\line(1,0){25}}
\put(50,8){\line(-1,0){5}}
\put(37,8){\makebox(0,0){$\cdots$}}
\put(55,8){\makebox(0,0){$\scriptstyle\langle$}}
\put(50,8){\makebox(0,0){$\bullet$}}
\put(60,8){\makebox(0,0){$\bullet$}}
\put(50,9){\line(1,0){10}}
\put(50,7){\line(1,0){10}}
\end{picture}$$
to which Theorem~\ref{plainBGGbound} applies, whence
$$\dim\{D\sigma=0\}\leq\dim:
\begin{picture}(75,20)(0,5)
\put(2,13){\makebox(0,0)[b]{$\scriptstyle 1$}}
\put(12,13){\makebox(0,0)[b]{$\scriptstyle 3$}}
\put(22,13){\makebox(0,0)[b]{$\scriptstyle 0$}}
\put(32,13){\makebox(0,0)[b]{$\scriptstyle 0$}}
\put(60,13){\makebox(0,0)[b]{$\scriptstyle 0$}}
\put(70,13){\makebox(0,0)[b]{$\scriptstyle 0$}}
\put(2,8){\makebox(0,0){$\bullet$}}
\put(12,8){\makebox(0,0){$\bullet$}}
\put(22,8){\makebox(0,0){$\bullet$}}
\put(32,8){\makebox(0,0){$\bullet$}}
\put(2,8){\line(1,0){35}}
\put(60,8){\line(-1,0){5}}
\put(47,8){\makebox(0,0){$\cdots$}}
\put(65,8){\makebox(0,0){$\scriptstyle\langle$}}
\put(60,8){\makebox(0,0){$\bullet$}}
\put(70,8){\makebox(0,0){$\bullet$}}
\put(60,9){\line(1,0){10}}
\put(60,7){\line(1,0){10}}
\end{picture}=\frac{4n(n+1)^2(n+2)(n+4)(2n+1)(2n+3)}{45}.$$
\end{example}
\begin{example}
Suppose that $\nabla_H$ is a partial connection
on $\Lambda_H^1$ and consider the differential operator
$$\sigma_{bc}\stackrel{D}{\longmapsto}
\nabla_{(a}\sigma_{b)c}+
\mbox{\large$\frac{1}{2n+1}$}\left(L^{de}\nabla_d\sigma_{e(a}\right)L_{b)c}+
\mbox{\large$\frac{1}{2n+1}$}L_{c(a}L^{de}\nabla_{b)}\sigma_{de}$$
for $\sigma_{bc}$ skew and trace-free with respect to the Levi form, i.e.\ a
section of $\Lambda_{H\perp}^2$ (we must suppose that $n\geq 2$). It is
designed as an operator
$$D:\begin{picture}(65,20)(0,5)
\put(2,13){\makebox(0,0)[b]{$\scriptstyle 0$}}
\put(12,13){\makebox(0,0)[b]{$\scriptstyle 1$}}
\put(22,13){\makebox(0,0)[b]{$\scriptstyle 0$}}
\put(50,13){\makebox(0,0)[b]{$\scriptstyle 0$}}
\put(60,13){\makebox(0,0)[b]{$\scriptstyle 0$}}
\put(2,8){\makebox(0,0){$\bullet$}}
\put(12,8){\makebox(0,0){$\bullet$}}
\put(22,8){\makebox(0,0){$\bullet$}}
\put(2,8){\line(1,0){25}}
\put(50,8){\line(-1,0){5}}
\put(37,8){\makebox(0,0){$\cdots$}}
\put(55,8){\makebox(0,0){$\scriptstyle\langle$}}
\put(50,8){\makebox(0,0){$\bullet$}}
\put(60,8){\makebox(0,0){$\bullet$}}
\put(50,9){\line(1,0){10}}
\put(50,7){\line(1,0){10}}
\end{picture}\to\begin{picture}(65,20)(0,5)
\put(2,13){\makebox(0,0)[b]{$\scriptstyle 1$}}
\put(12,13){\makebox(0,0)[b]{$\scriptstyle 1$}}
\put(22,13){\makebox(0,0)[b]{$\scriptstyle 0$}}
\put(50,13){\makebox(0,0)[b]{$\scriptstyle 0$}}
\put(60,13){\makebox(0,0)[b]{$\scriptstyle 0$}}
\put(2,8){\makebox(0,0){$\bullet$}}
\put(12,8){\makebox(0,0){$\bullet$}}
\put(22,8){\makebox(0,0){$\bullet$}}
\put(2,8){\line(1,0){25}}
\put(50,8){\line(-1,0){5}}
\put(37,8){\makebox(0,0){$\cdots$}}
\put(55,8){\makebox(0,0){$\scriptstyle\langle$}}
\put(50,8){\makebox(0,0){$\bullet$}}
\put(60,8){\makebox(0,0){$\bullet$}}
\put(50,9){\line(1,0){10}}
\put(50,7){\line(1,0){10}}
\end{picture}$$
to which Theorem~\ref{plainBGGbound} applies, whence
$$\dim\{D\sigma=0\}\leq\dim:
\begin{picture}(75,20)(0,5)
\put(2,13){\makebox(0,0)[b]{$\scriptstyle 0$}}
\put(12,13){\makebox(0,0)[b]{$\scriptstyle 0$}}
\put(22,13){\makebox(0,0)[b]{$\scriptstyle 1$}}
\put(32,13){\makebox(0,0)[b]{$\scriptstyle 0$}}
\put(60,13){\makebox(0,0)[b]{$\scriptstyle 0$}}
\put(70,13){\makebox(0,0)[b]{$\scriptstyle 0$}}
\put(2,8){\makebox(0,0){$\bullet$}}
\put(12,8){\makebox(0,0){$\bullet$}}
\put(22,8){\makebox(0,0){$\bullet$}}
\put(32,8){\makebox(0,0){$\bullet$}}
\put(2,8){\line(1,0){35}}
\put(60,8){\line(-1,0){5}}
\put(47,8){\makebox(0,0){$\cdots$}}
\put(65,8){\makebox(0,0){$\scriptstyle\langle$}}
\put(60,8){\makebox(0,0){$\bullet$}}
\put(70,8){\makebox(0,0){$\bullet$}}
\put(60,9){\line(1,0){10}}
\put(60,7){\line(1,0){10}}
\end{picture}=\frac{2(n-1)(n+1)(2n+3)}{3}.$$
\end{example}

\begin{example} In~\cite{hos}, the authors consider a second order differential
operator compatible with a contact Lagrangian structure~\cite[\S4.2.3]{thebook}
in $3$-dimensions whose leading terms in Darboux local co\"ordinates are
$$f\mapsto(X^2f,Y^2f).$$
By taking ${\mathbb{E}}=\begin{picture}(15,20)(0,5)
\put(2,13){\makebox(0,0)[b]{$\scriptstyle 1$}}
\put(12,13){\makebox(0,0)[b]{$\scriptstyle 1$}}
\put(2,8){\makebox(0,0){$\bullet$}}
\put(12,8){\makebox(0,0){$\bullet$}}
\put(2,8){\line(1,0){10}}
\end{picture}$ in the contact Lagrangian counterpart of 
Theorem~\ref{plainBGGbound}, it follows that the solutions space has dimension
bounded by~$8$. This dimension bound agrees with \cite[Theorem~3.1]{hos}, 
which the authors establish by an ad hoc prolongation.
\end{example}

\begin{remark}
There is a close parallel between the methods used in this article and the
methods of {\em parabolic geometry\/} as described in~\cite{thebook}. These
methods are informally and collectively known as the {\em
Bernstein-Gelfand-Gelfand\/} machinery and Kostant's computation \cite{kostant}
of Lie algebra cohomologies is a key ingredient in this machinery. The {\em
homogeneous models\/} and their first {\em Bernstein-Gelfand-Gelfand
operators\/} show that the dimension bounds in Theorem~\ref{plainBGGbound} and
its parabolic variants are sharp. The prolongations of \cite{bceg} compare to
$|1|$-graded parabolic geometries \cite[\S4.1]{thebook} in an entirely
analogous fashion.
\end{remark}

\end{document}